\documentclass[a4paper,10pt,twoside,english]{scrartcl}
\usepackage{marginnote}
\usepackage[utf8x]{inputenc}
\usepackage{amsmath,amssymb,amsthm}
\usepackage{mathrsfs}  
\usepackage{dsfont} % used for \id 
\usepackage[unicode=true]{hyperref}
\usepackage[all]{hypcap}
\usepackage[T1]{fontenc}
\usepackage{lmodern}
\usepackage{graphicx}
\usepackage{caption}
\usepackage{subcaption}
\usepackage{enumerate}
\usepackage{setspace}
\usepackage{xspace} 
\usepackage{blkarray}
\usepackage[subnum]{cases}
\usepackage{pgfplots}
  \pgfplotsset{compat=newest}
  \usetikzlibrary{plotmarks}
  \usetikzlibrary{arrows.meta}
  \usepgfplotslibrary{patchplots}
  \usepackage{grffile}

\newcommand{\titel}{Local inhomogeneous circular law}
\title{\titel} 
\author{
Johannes Alt\footnote{Partially funded by ERC Advanced Grant RANMAT No. 338804.} \addtocounter{footnote}{-1}\addtocounter{Hfootnote}{-1}\\
{\small \begin{tabular}{c}{IST Austria}\\{johannes.alt@ist.ac.at} \end{tabular}} 
\and László Erd\H{o}s\footnotemark \addtocounter{footnote}{-1}\addtocounter{Hfootnote}{-1} 
\\{\small \begin{tabular}{c} IST Austria\\ lerdos@ist.ac.at\end{tabular}} 
\and Torben Krüger\footnotemark\\
{\small \begin{tabular}{c} IST Austria\\ torben.krueger@ist.ac.at\end{tabular}}
}
\date{}

\hypersetup{
	pdftitle={\titel},
	pdfauthor={Johannes Alt, László Erd\H{o}s, Torben Krüger}%,
}

\numberwithin{equation}{section}

\newcommand{\R}{\mathbb{R}}  % The real numbers.
\ifdefined\C\renewcommand{\C}{\mathbb{C}}\else\newcommand{\C}{\mathbb{C}}\fi % Complex numbers
\renewcommand{\Im}{\mathrm{Im}\,} %imaginary part of a complex number
\renewcommand{\Re}{\mathrm{Re}\,} %real part of a complex number
\newcommand{\N}{\mathbb{N}}  % Positive integers.	
\newcommand{\E}{\mathbb{E}}  % expected value of random variable	
\renewcommand{\P}{\mathbb{P}}  % probability measure
\newcommand{\di}{\mathrm{d}} % differential
\newcommand{\eps}{\varepsilon} % "correct" epsilon
\newcommand*{\defeq}{\mathrel{\vcenter{\baselineskip0.5ex \lineskiplimit0pt\hbox{\scriptsize.}\hbox{\scriptsize.}}}=}

\newcommand{\pt}{\partial}

\DeclareMathOperator{\supp}{supp}

\DeclareMathOperator{\linspan}{span}
\newcommand{\Rnon}{\ensuremath{\R^{+}_{0}}}

\newcommand{\emin}{\ensuremath{\boldsymbol{e}_-}}

\newcommand{\df}{\boldsymbol d}
\newcommand{\vf}{\boldsymbol v}

\newcommand{\uf}{\boldsymbol u}
\newcommand{\ffp}{\boldsymbol f_+}
\newcommand{\ffm}{\boldsymbol f_-}
\newcommand{\gf}{\boldsymbol g}
\newcommand{\yf}{\boldsymbol y}
\newcommand{\xf}{\boldsymbol x}
\newcommand{\af}{\boldsymbol a}

\newcommand{\bb}{\boldsymbol b}
\newcommand{\hf}{\boldsymbol h}

\newcommand{\pf}{\boldsymbol p}
\newcommand{\rf}{\ensuremath{\boldsymbol r}} 

\newcommand{\diM}{\mathscr D}

\newcommand{\Af}{\boldsymbol A}
\newcommand{\Bf}{\boldsymbol B}
\newcommand{\Df}{\boldsymbol D}
\newcommand{\Ef}{\boldsymbol E}
\newcommand{\Ff}{\boldsymbol F}
\newcommand{\Gf}{\boldsymbol G}
\newcommand{\Hf}{\boldsymbol H}
\newcommand{\Lf}{\boldsymbol L}
\newcommand{\Mf}{\boldsymbol M}
\newcommand{\Qf}{\boldsymbol Q}
\newcommand{\Rf}{\boldsymbol R}
\newcommand{\Sf}{\boldsymbol S}
\newcommand{\Tf}{\boldsymbol T}

\newcommand{\Vf}{\boldsymbol V}
\newcommand{\Wf}{\boldsymbol W}

\newcommand{\Hb}{\mathbb H}

\newcommand{\id}{\mathds{1}}
\renewcommand{\char}{\ensuremath{\chi}} %\mathbf{1}}}
	
\newcommand{\bs}[1]{\boldsymbol{{#1}}} %bold
\renewcommand{\rm}{\mathrm} %upright

\newcommand{\Dsma}{\ensuremath{\mathbb D_{<}}}
\newcommand{\Dbig}{\ensuremath{\mathbb D_{>}}}
\newcommand{\Isma}{\ensuremath{[0,1-\zsq_*]}}
\newcommand{\Ibig}{\ensuremath{[1+\zsq_*,\zsq^*]}}
\newcommand{\para}{\ensuremath{\mathcal P}}

\newcommand{\normtwo}[1]{\lVert #1 \rVert_{2}}
\newcommand{\normtwoa}[1]{\left\lVert #1 \right\rVert_{2}}
\newcommand{\normtwoinf}[1]{\lVert #1 \rVert_{2\to\infty}}
\newcommand{\norminf}[1]{\lVert #1 \rVert_{\infty}}
\newcommand{\norminfa}[1]{\left\lVert #1 \right\rVert_{\infty}}

\newtheoremstyle{test}% name
  {}%      Space above, empty = `usual value'
  {}%      Space below
  {\itshape}% Body font
  {}%         Indent amount (empty = no indent, \parindent = para indent)
  {\bfseries}% Thm head font
  {.}%        Punctuation after thm head 
  { }% Space after thm head: \newline = linebreak
  {}%         Thm head spec

\theoremstyle{test}
\newtheorem{defi}{Definition}[section]

\newtheorem{rem}[defi]{Remark}

\newtheorem{thm}[defi]{Theorem}
\newtheorem{lem}[defi]{Lemma}
\newtheorem{coro}[defi]{Corollary}
\newtheorem{pro}[defi]{Proposition}
\newtheorem*{rem*}{Remark}   %no numbering
\newtheorem*{ex*}{Example}   %no numbering
\newtheorem*{pro*}{Proposition} %no numbering
\newtheorem*{def*}{Definition}
\newtheorem*{coro*}{Corollary}
\newtheorem*{thm*}{Theorem}

\theoremstyle{test}
    \newtheorem{theorem}[defi]{Theorem}
    \newtheorem{proposition}[defi]{Proposition}
    \newtheorem{corollary}[defi]{Corollary}
    \newtheorem{lemma}[defi]{Lemma}
    \newtheorem{definition}[defi]{Definition}% howto make rm-style text inside definitions
    
    \newtheorem{convention}[defi]{Convention}
    \newtheorem{remark}[defi]{Remark}

\newcommand{\bels}[2] {
        \begin{equation} \label{#1} \begin{split} 
                #2 
        \end{split} \end{equation}
        }

\renewcommand{\bf}[1]{\boldsymbol{\mathrm{#1}}} %bold
\renewcommand{\cal}{\mathcal}

\newcommand{\wh}{\widehat}
\newcommand{\wt}{\widetilde}

\renewcommand{\P}{\mathbb{P}}

\newcommand{\ee}{\mathrm{e}} %\newcommand{\me}{\mathrm{e}}
\newcommand{\ii}{\mathrm{i}} %\newcommand{\mi}{\mathrm{i}}

\newcommand{\abs}[1]{\lvert #1 \rvert}

\newcommand{\absa}[1]{\left\lvert #1 \right\rvert}

\newcommand{\norm}[1]{\lVert #1 \rVert}
\newcommand{\normb}[1]{\big\lVert #1 \big\rVert}

\newcommand{\avg}[1]{\langle #1 \rangle}

\newcommand{\avga}[1]{\left\langle #1 \right\rangle}

\newcommand{\scalar}[2]{\langle{#1} \mspace{2mu}, {#2}\rangle}

\newcommand{\scalara}[2]{\left\langle{#1} \,\mspace{2mu},\, {#2}\right\rangle}

\DeclareMathOperator{\diag}{diag}
\DeclareMathOperator{\tr}{Tr}

\DeclareMathOperator{\im}{Im}

\DeclareMathOperator*{\spec}{Spec}						%Spectrum

\newcommand{\1} {\mspace{1 mu}}
\newcommand{\2} {\mspace{2 mu}}

\newcommand{\zsq}{\tau}

\newcommand{\eigX}{\sigma}
\newcommand{\eigH}{\lambda}

\begin{document}

\maketitle
\begin{abstract}
We consider large random matrices $X$ with centered, independent entries which have comparable but not necessarily identical 
variances. Girko's circular law asserts that the spectrum is supported in a disk and in case of identical
variances, the limiting density is uniform. In this special case, the \emph{local circular law} by Bourgade \emph{et. al.} \cite{Bourgade2014, BYY_circular2}
shows that  the empirical density 
converges even locally on scales slightly above the  typical eigenvalue spacing.
 In the general case, the limiting density  is  typically inhomogeneous and it
is obtained via 
solving a system of deterministic equations. 
Our main result is  the local \emph{inhomogeneous} circular law
in the bulk spectrum
on the optimal scale for a general variance profile of the entries of $X$. 
\end{abstract}


\section{Introduction}

The density of eigenvalues  of large random matrices typically 
converges to a deterministic limit as the dimension $n$ of the matrix tends to infinity.  In the Hermitian case, the best
known  examples are the Wigner semicircle law for Wigner ensembles and 
the Marchenko-Pastur law for sample covariance matrices. In both cases 
 the spectrum is real, and these laws state that the empirical eigenvalue distribution 
converges to an explicit density on the real line.

 The spectra of non-Hermitian random matrices
concentrate on a domain of the complex plane. The most prominent case is the \emph{circular law},
asserting that for an $n\times n$ matrix $X$ with independent, identically distributed
entries, satisfying $\E x_{ij} =0$, $\E |x_{ij}|^2 =n^{-1}$, the  empirical density converges
to the uniform distribution on the unit disk  $\{ z \, : \,  |z| < 1\}\subset\C$.
Despite the apparent similarity in the statements, it is considerably harder to analyze 
 non-Hermitian random matrices  than  their Hermitian counterparts since  eigenvalues 
 of non-Hermitian matrices  may respond very drastically to small  perturbations. 
 This instability is one reason why the universality of local eigenvalue statistics in the bulk spectrum, exactly on the scale of the eigenvalue
 spacing,  is not yet established  for $X$ with independent (even for i.i.d.) entries, while the corresponding 
 statement for Hermitian Wigner matrices, known as the Wigner-Dyson-Mehta
 universality conjecture,  has been proven recently, see \cite{ErdoesYau2012} for an overview.

The circular law for i.i.d. entries has a long history, we refer to the extensive review \cite{bordenave2012}.
The complex Gaussian case (Ginibre ensemble) has been settled in the sixties by Mehta using explicit computations.
Girko in \cite{Girko1984} found  a key  formula to relate
 linear statistics of  eigenvalues of $X$  to  eigenvalues of the  family of Hermitian matrices $(X-z)^*(X-z)$ where $z\in \C$ is a complex
 parameter.  
 Technical difficulties still remained until Bai \cite{bai1997} 
  presented a complete proof 
 under two additional assumptions requiring  higher moments and bounded density for the single entry distribution.
 After a series of further partial results \cite{goetze2010, Pan2010, tao2008}
  the circular law for i.i.d. entries  under the optimal condition, assuming only the existence of the second moment, 
was established by  Tao and Vu \cite{tao2010}.

Another line of research focused on the local version of the circular law with 
constant variances, $\E |x_{ij}|^2=n^{-1}$, which asserts that the local density of eigenvalues  is still uniform 
on scales $n^{-1/2+\epsilon}$, i.e., slightly above the typical spacing between neighboring eigenvalues.  The optimal result 
was achieved in Bourgade, Yau and Yin \cite{Bourgade2014, BYY_circular2} and Yin \cite{Y_circularlaw} both inside the
unit disk  (``bulk regime'') and at the edge $|z|=1$. If the first three moments match those of a standard complex Gaussian, then a similar result has also been obtained 
by Tao and Vu in \cite{tao2015}.
In \cite{tao2015}, this result was used to prove the universality of local eigenvalue statistics under the assumption that the first four moments match those of a complex Gaussian. 
While there is no proof of universality for general distributions without moment matching conditions yet,
similarly to the development in the Hermitian case, the local law  is expected to be one of the key ingredients of such a proof in the future.

In this paper we study non-Hermitian matrices  $X$ with a general matrix of variances $S=(s_{ij})$, i.e., we assume that $x_{ij}$ are centered, independent,
but $s_{ij} \defeq \E |x_{ij}|^2$ may 
depend non-trivially on the indices $i,j$. 
We show that the eigenvalue density is close to a deterministic density $\sigma$ on the smallest possible scale.
As a direct application, our local law implies that the spectral radius $\rho(X)$ of $X$ is arbitrarily close to $\sqrt{\rho(S)}$, where
$\rho(S)$ is the spectral radius of $S$. More precisely, we prove that for every $\eps>0$
\[ \sqrt{\rho(S)} - \eps \leq \rho(X) \leq \sqrt{\rho(S)} +\eps \]
with a very high probability as $n$ tends to infinity. 
The fact that the spectral radius of $X$ becomes essentially
deterministic is the key mathematical mechanism behind the sharp ``transition to chaos'' in a commonly
studied mean field model of dynamical neural networks \cite{ChaosInRandomNeuralNetworks}. This transition is described by the stability/instability of 
the system of ordinary differential equations
\[ \dot{q_i}(t) = q_i(t) -\lambda \sum_{j=1}^n x_{ij}q_j(t) \]
for $i=1, \ldots, n$ as $\lambda$ varies. 
Moreover, the number of unstable modes close to the critical value of the parameter $\lambda$ is determined by 
the behaviour of $\sigma$ at the spectral edge which we also analyze.
Such systems have originally been studied under the assumption that the coefficients $x_{ij}$ are independent and identically distributed \cite{may1972will}. 
More recently, however, it was argued \cite{PhysRevLett.114.088101,Aljadeff2015} 
that for more realistic applications in neuroscience one should allow $x_{ij}$ to have 
 varying distributions with an arbitrary variance profile $S$.

After Girko's Hermitization, understanding the spectrum of $X$ reduces to analyzing the spectrum
 of the family of Hermitian matrices 
\begin{equation}\label{Hdef}
   \Hf^z \defeq \begin{pmatrix} 0 & X- z\id \\ X^*-\bar z \id & 0 \end{pmatrix} 
\end{equation}
of double dimension, where $z\in \C$. 
The Stieltjes transform of the spectral density of $\Hf^z$  at  any spectral parameter
 $w$ in the upper half plane $\Hb\defeq\{ w\in \C\; : \; \im w >0\}$ 
 is approximated via the solution of a system of $2n$ nonlinear equations, written concisely as
 \begin{equation}\begin{split}\label{Seq}
     -\frac{1}{m_1}= &w + Sm_2 - \frac{|z|^2}{w+ S^t m_1}, \\
     -\frac{1}{m_2}= & w + S^t m_1 - \frac{|z|^2}{w+ S m_2},
 \end{split}
 \end{equation}
where $m_a= m_a^z(w) \in \Hb^n$, $a=1,2$ are $n$-vectors with each component   in the upper half plane.
The normalized trace of the resolvent, $\frac{1}{2n} \mbox{trace} (\Hf^z-w)^{-1}$, is approximately equal to $\frac{1}{n}\sum_j [m_1^{z}(w)]_j 
=\frac{1}{n}\sum_j [m_2^{z}(w)]_j $ in the $n\to\infty$ limit. The spectral density of $\Hf^z$ at any $E\in \R$ is then given by setting $w=E+\ii\eta$ and 
taking the limit $\eta\to 0+$ for the imaginary part of these averages. In fact, for Girko's formula it is sufficient to study
the resolvent only along the positive imaginary axis $w\in \ii\R_+$.
Heuristically, the  equations \eqref{Seq} arise from second order perturbation theory and in physics they are commonly called \emph{Dyson equations}.
Their analogues for general  Hermitian ensembles with independent or weakly dependent entries play  an essential
role in random matrix theory. They have been systematically studied by Girko, for example,  equation \eqref{Seq}
in the current random matrix context appears as the \emph{canonical equation of type}  $K_{25}$ in Theorem 25.1 in 
\cite{girko2012theory}. In particular, under the condition
that all $s_{ij}$ variances are comparable, i.e., $c/n \le s_{ij}\le C/n$ with some positive constants $c, C$,
Girko identifies the limiting density.
 From his formulas it is clear that this density is rotationally symmetric. 
He also presents a proof for the weak convergence of the empirical eigenvalue distribution but the argument was considered incomplete. This deficiency 
can be resolved in a similar manner as for the circular law assuming a bounded density of the single entry distribution using the argument 
from Section~4.4 of \cite{bordenave2012}. 
 In a recent preprint \cite{CookNonHermitianRM} Cook \emph{et. al.}  
substantially relax   the condition on the uniform  bound $s_{ij}\ge c/n$ by replacing it with a
concept of \emph{robust irreducibility}.  Moreover, relying on the bound by Cook \cite{CookSmallestSingularValue}
on the smallest singular value of $X$, they also remove any condition on the regularity of the single
entry distribution and prove weak convergence on the global scale.

The matrix $\Hf^z$ may be viewed as the sum of a \emph{Wigner-type matrix} \cite{Ajankirandommatrix} with centered, independent (up to Hermitian symmetry) entries 
and a deterministic matrix whose two off-diagonal blocks are $-z\id$ and $-\bar z \id$, respectively. 
Disregarding  these $z$ terms  for the moment,  \eqref{Seq} has the structure of the \emph{Quadratic Vector Equations}
that were extensively studied in \cite{AjankiQVE,AjankiCPAM}.  Including  the $z$-terms, $\Hf^z$ at first sight seems to be a 
 special case
of the random matrix ensembles with nonzero expectations analyzed in \cite{AjankiCorrelated}
and \eqref{Seq} is the diagonal part of the corresponding \emph{Matrix Dyson Equation (MDE)}. 
In \cite{AjankiCorrelated} an optimal local law was proved for such ensembles. However,  the large zero blocks in the diagonal
prevent us from applying these results to $\Hf^z$ or even to $\Hf^{z=0}$. In fact, the flatness condition ${\bf A1}$
 in \cite{AjankiCorrelated}  (see \eqref{MDE:flatness} later) 
or even its relaxed version ${\bf A1'}$   in \cite{AjankiCorrelated} prohibit such large zero diagonal blocks. 
These conditions  are essential for the proofs in \cite{AjankiCorrelated} since
they ensure the stability of  the corresponding Dyson equation against \emph{any} small perturbation. 
In this case, there 
is  only one   potentially unstable
direction, that is associated to a certain  Perron-Frobenius eigenvector,  and this 
 direction  is  regularized by the positivity of the density of states at least in the bulk regime of the spectrum.

If the flatness condition ${\bf A1}$   is not satisfied, then the MDE can possess further unstable directions.
In particular, in our setup, the MDE is not stable in the previously described strong sense;
there is at least one  additional unstable direction which cannot be regularized by the positivity of the density of states. 
Owing to the specific structure of $\Hf^z$, the \emph{matrix} Dyson equation decouples and its diagonal parts 
satisfy a closed  system  of \emph{vector} equations \eqref{Seq}. Compared to the MDE,
the  reduced vector equations \eqref{Seq} are rather cubic than quadratic in nature.
 For this reduced system, however,  we can show
that there is only one further unstable direction, at least when $S$ is entrywise bounded from
below by some $c/n$.  The system is not stable against an arbitrary perturbation, but 
for the perturbation arising in the random matrix problem we reveal
a key cancellation  in the leading contribution  to the unstable direction. 
Armed with this new insight we will perform a detailed stability analysis of \eqref{Seq}.

This delicate stability analysis is the key ingredient for the proof of our main result, the optimal local law for $X$ with an optimal 
speed of convergence as $n\to\infty$.
In this paper we consider the bulk regime, i.e., spectral parameter $z$ 
inside the disk with boundary $|z|^2=\rho(S)$, where $\rho(S)$ is the spectral radius of $S$.
We defer the analysis of the edge of the spectrum of $X$ to later works.

In the special case $z=0$, we thoroughly studied
the system of equations \eqref{Seq} 
 even for the case when $S$ is a rectangular matrix in \cite{AltGram}; the main motivation was to prove the local law for random
\emph{Gram matrices}, i.e., matrices of the form $XX^*$.  Note that    in \cite{AltGram} we needed to 
tackle a much simpler quadratic system since 
taking $z=0$ in \eqref{Seq} removes the most complicated nonlinearity. 

Finally, we list two related recent results. Local circular law on the optimal scale  in the bulk has been proven in \cite{XYY_circularlaw}
for ensembles of the form $TX$, where $T$ is a deterministic $N \times M$ matrix and $X$ is a random $M \times N$ matrix with independent, centered entries whose variances are constant and have vanishing third moments.
The structure of the product matrix $TX$ is very different from our matrices that could be
viewed as the Hadamard (entrywise) product of the matrix $(s_{ij}^{1/2})$ and a random matrix with identical variances.
The  approach of \cite{XYY_circularlaw}  is also very different from ours: it relies on first assuming that $X$ is Gaussian
and using its invariance to reduce the problem to the case when $T^*T$ is diagonal. Then the corresponding Dyson equations
are much simpler, in fact they consist of only two scalar equations
and  they are characterized by a vector of parameters (of the singular values of $T$) instead of an entire matrix of parameters $S$.
The vanishing third moment condition in \cite{XYY_circularlaw} is necessary to compare  the general distribution with the Gaussian case 
via a  moment matching argument. We also mention the  recent proof of the local \emph{single ring theorem}
on optimal scale in the bulk \cite{ErdosSchnelli2016}. This concerns another prominent non-Hermitian random matrix ensemble that
 consists of matrices of the form $U\Sigma V$, where  $U$, $V$ are two independent Haar distributed unitaries
 and $\Sigma$ is deterministic (may be assumed to be diagonal). The spectrum lies in a ring about the origin 
 and the limiting density can be computed via free convolution \cite{guionnet2011single}.

\paragraph{Acknowledgement} We are grateful to David Renfrew for discussing some applications of our results with us and to Dominik Schr\"oder for helping us visualizing our results.

\paragraph{Notation}
For vectors $v, w \in \C^l$, we write their componentwise product as $vw=(v_i w_i)_{i=1}^l$. If $f\colon U \to \C$ is a function on $U \subset \C$, then we define $f(v)\in \C^l$ for $v\in U^l$ to be the vector 
with components $f(v)_i = f(v_i)$ for $i = 1, \ldots, l$. We will in particular apply this notation with $f(z) = 1/z$ for $z \in \C\setminus \{ 0\}$. We say that a vector $v \in \C^l$ is positive, $v >0$, if $v_i>0$ for all 
$i=1, \ldots, l$. 
Similarly, the notation $v \leq w$ means $v_i \leq w_i$ for all $i=1, \ldots, l$. 
For vectors $v, w \in \C^l$, we define $\avg{w} = l^{-1} \sum_{i=1}^l w_i$, $\scalar{v}{w} = l^{-1} \sum_{i=1}^l \overline{v_i} w_i$, $\norm{w}_2^2 = l^{-1} \sum_{i=1}^l \abs{w_i}^2$ and $\norm{w}_\infty = \max_{i=1, \ldots, l} \abs{w_i}$, $\norm{v}_1 \defeq \avg{\abs{v}}$. Note that $\avg{w} = \scalar{1}{w}$, where we used the convention that $1$ also denotes the vector $(1,\ldots, 1) \in \C^l$. 
In general, we use the notation that if a scalar $\alpha$ appears in a vector-valued relation, then it denotes the constant vector $(\alpha, \ldots, \alpha)$. 
In most cases we will work in $n$ or $2n$ dimensional spaces. 
Vectors in $\C^{2n}$ will usually be denoted by boldface symbols like $\vf$, $\uf$ or $\yf$. Correspondingly, capitalized boldface symbols denote matrices in $\C^{2n\times 2n}$, for example $\Rf$. 
We use the symbol $\id$ for the identity matrix in $\C^{l\times l}$, where the dimension $l=n$ or $l=2n$  is  understood from the context. 
For a matrix $A \in \C^{l \times l}$, we use the short notations $\norminf{A} \defeq \norm{A}_{\infty \to \infty}$ and $\normtwo{A} \defeq \norm{A}_{2 \to 2}$ if the domain and the target are equipped with the same norm
whereas we use $\normtwoinf{A}$ to denote the matrix norm of $A$ when it is understood as a map $(\C^l, \norm{\cdot}_2) \to (\C^l, \norm{\cdot}_\infty)$.    
We define the normalized trace of an $l\times l$ matrix $B = (b_{ij})_{i,j=1}^l\in \C^{l\times l}$ as 
\begin{equation} \label{eq:def_trace}
 \tr B \defeq \frac{1}{l} \sum_{j=1}^{l} b_{jj}. 
\end{equation}
For a vector $y \in \C^{l}$, we write $\diag y$ or $\diag(y)$ for the diagonal $l\times l$ matrix with $y$ on its diagonal, i.e., this matrix acts on any vector $x \in \C^l$ as
\begin{equation} \label{eq:def_diag}
 \diag(y)x = y x.
\end{equation}
We write $\di^2z$ for indicating integration with respect to 
the Lebesgue measure on $\C$. For $a\in \C$ and $\eps>0$, the open disk in the complex plane centered at $a$ with radius $\eps$ is denoted by
 $D(a,\eps)  \defeq\{ b \in \C\,\mid\, \abs{a-b} < \eps\}$.
Furthermore, we denote the characteristic function of some event $A$ by $\char(A)$, the positive real numbers by $\R_+ \defeq (0,\infty)$ and the nonnegative real numbers by $\Rnon\defeq [0,\infty)$.

\section{Main results}

Let $X$ be a random $n\times n$ matrix with centered entries, $\E x_{ij} = 0$, and $s_{ij} \defeq \E \abs{x_{ij}}^2$ the corresponding variances. We introduce its variance matrix $S\defeq (s_{ij})_{i,j=1}^n$. 

\noindent \textbf{Assumptions}: 
\begin{enumerate}[(A)]
\item The variance matrix $S$ is \emph{flat}, i.e., there are $0<s_*<s^*$ such that 
\begin{equation} \label{eq:assumption_A}
 \frac{s_*}{n} \leq s_{ij} \leq \frac{s^*}{n} 
\end{equation}
for all $i,j=1, \ldots, n$.
\item All entries of $X$ have bounded moments in the sense that
there are $\mu_m >0$ for $m \in \N$ such that 
\begin{equation} \label{eq:assumption_B}
 \E \abs{x_{ij}}^m \leq \mu_m n^{-m/2}
\end{equation}
for all $i,j=1, \ldots, n$.
\item Each entry of $\sqrt{n}\; X$ has a density, i.e., there are probability densities $f_{ij} \colon \C \to [0,\infty)$ such that
\[ \P \left( \sqrt n \; x_{ij} \in B \right) = \int_B f_{ij} (z) \di^2 z \]
for all $i,j=1,\ldots, n$ and $B \subset \C$ a Borel set. 
There are $\alpha,\beta>0$ such that $f_{ij} \in L^{1+\alpha}(\C)$ and  
\begin{equation} \label{eq:bounded_density}
\norm{f_{ij}}_{1+\alpha} \leq n^\beta
\end{equation}
for all $i,j=1,\ldots, n$.
\end{enumerate}

In the following, we will assume that $s_*$, $s^*$, $\alpha$, $\beta$ and the sequence $(\mu_m)_m$ are fixed constants 
which we will call \emph{model parameters}. 
 The constants in all our estimates will depend on the model parameters 
without further notice. 

\begin{rem}
The Assumption (C) is used in our proof solely for controlling the smallest singular value of $X-z\id$ with very high probability uniformly 
for $z \in  D(0,\zsq^*)$  with some fixed $\zsq^*>0$ in Proposition \ref{pro:least_singular_value}. All our other results do not make use of 
Assumption (C). Provided a version of Proposition \ref{pro:least_singular_value} that tracks the $z$-dependence can effectively be obtained without (C), our main result, the local inhomogeneous circular law in Theorem \ref{thm:local_circular_law}, 
will hold true solely assuming (A) and (B). For example a very high probability estimate uniform in $z$ in a statement similar to Corollary 1.22 of \cite{CookSmallestSingularValue} would be sufficient.
\end{rem}

The density of states of $X$ will be expressed in terms of 
$v_1^\zsq$ and $v_2^\zsq$  which are the positive solutions of the following two coupled vector equations
\begin{subequations} \label{eq:v}
\begin{align}
\frac{1}{v^\zsq_1} & = \eta + S v^\zsq_2 + \frac{\zsq}{\eta + S^tv^\zsq_1}, \label{eq:v_1} \\ 
\frac{1}{v^\zsq_2} & = \eta + S^t v^\zsq_1 + \frac{\zsq}{\eta + S v^\zsq_2}. \label{eq:v_2}
\end{align}
\end{subequations}
for all $\eta\in\R_+$ and $\zsq\in \Rnon$. Here, $v_1^\zsq, v_2^\zsq\in \R_+^n$ and recall that the algebraic operations are understood componentwise, e.g., $(1/v)_i = 1/v_i$ for the $i^{\text{th}}$ component of the vector $v$.
The system \eqref{eq:v} is a special case of \eqref{Seq} with $w = \ii\eta$, $\zsq = \abs{z}^2$ and $v_a = \Im m_a$ for $a=1,2$. 
The existence and uniqueness of solutions to equations of the type \eqref{eq:v} are considered standard knowledge in the literature \cite{girko2012theory}. The equations can be viewed as a special case of the matrix 
Dyson equation for which existence and uniqueness was proven in \cite{Helton01012007}. We explain this connection in more detail in the appendix where we give the proof of Lemma~\ref{lem:existence_uniqueness_vf_equation} 
for the convenience of the reader.

\begin{lem}[Existence and uniqueness] \label{lem:existence_uniqueness_vf_equation}
For every $\zsq \in \Rnon$, there exist two uniquely determined functions $v_1^\zsq\colon \R_+ \to \R_+^n$, $v_2^\zsq\colon \R_+ \to \R_+^n$ which satisfy \eqref{eq:v}.
\end{lem}

 We denote the spectral radius of $S$ by $\rho(S)$, 
 i.e., \[ \rho(S) \defeq \max \abs{\spec(S)}. \]
Now, we define the density of states of $X$ through the solution to \eqref{eq:v}. 

\begin{defi}[Density of states of $X$]
Let $v_1^\zsq$ and $v_2^\zsq$ be the unique positive solutions of \eqref{eq:v}.  The \emph{density of states} $\sigma\colon \C \to \R$ of $X$ is defined through
\begin{equation} \label{eq:def_sigma}
 \sigma(z) \defeq - \frac{1}{2\pi}\int_0^\infty \Delta_z\avga{v_1^\zsq(\eta)\middle|_{\zsq=\abs{z}^2}} \di \eta
\end{equation}
for $ \abs{z}^2 <\rho(S)$ and $\sigma(z) \defeq 0$ for $\abs{z}^2 \geq\rho(S)$. The right-hand side of \eqref{eq:def_sigma} is well-defined by part (i) of the following proposition. 
\end{defi}

In the following proposition, we present some key properties of the density of states $\sigma$ of $X$. For an alternative representation of $\sigma$, 
see \eqref{eq:sigma_in_terms_of_vf_0} later. 	

\begin{pro}[Properties of $\sigma$]  \label{pro:properties_sigma}
Let $v_1^\zsq$ and $v_2^\zsq$ be the unique positive solutions of \eqref{eq:v}. Then 
\begin{enumerate}[(i)]
\item The function $\R_+ \times \C \to \R_+^{2n}, (\eta, z) \mapsto \left(v_1^\zsq(\eta),v_2^{\zsq}(\eta)\right)|_{\zsq=\abs{z}^2}$ is infinitely often differentiable and $\eta \mapsto  
\Delta_z\avga{v_1^\zsq(\eta)\middle|_{\zsq=\abs{z}^2}}$ is integrable 
on $\R_+$ for each $z \in D(0,\sqrt{\rho(S)})$. 
\item The function $\sigma$, defined in \eqref{eq:def_sigma}, is a rotationally symmetric probability density on $\C$.
\item 
The restriction $\sigma|_{D(0, \sqrt{\rho(S)})}$ is infinitely often differentiable 
such that for every $\eps>0$ each derivative is bounded uniformly in $n$ on $D(0,\sqrt{\rho(S)}-\eps)$.
Moreover, there exist constants $c_1 > c_2 >0$, which depend only on $s_*$ and $s^*$, such that 
\begin{equation}  \label{eq:uniform_lower_bound_sigma}
c_1 \geq \sigma(z) \geq c_2
\end{equation} 
for all $z \in D(0, \sqrt{\rho(S)})$.
In particular, the support of $\sigma$ is the closed disk of radius $\sqrt{\rho(S)}$ around zero. 
\end{enumerate}
\end{pro}

The next theorem, the main result of the present article, states that  the eigenvalue distribution of $X$, with a very high 
probability, can 
be approximated by $\sigma$ on the mesoscopic scales $n^{-a}$ for any $a \in (0,1/2)$. Note that $n^{-1/2}$ is the typical eigenvalue spacing so our result holds down to the optimal local scale. 
 To study the local scale, we shift and rescale the test functions as follows. 
Let $f \in C_0^2(\C)$. 
For $z_0 \in \C$ and $a >0$, we define \[f_{z_0,a} \colon \C \to \C, \quad f_{z_0,a}( z) \defeq n^{2a}f(n^a(z-z_0)).\]

We denote the eigenvalues of $X$ by $\eigX_1, \ldots, \eigX_n$. 

\begin{thm}[Local inhomogeneous circular law] \label{thm:local_circular_law}
Let $X$ be a random matrix which has independent centered entries and satisfies (A), (B) and (C). Furthermore, let $a \in (0,1/2)$, $\varphi>0$, $\zsq_*>0$ and 
$\sigma$ defined as in \eqref{eq:def_sigma}. 
\begin{enumerate}[(i)]
\item (Bulk spectrum) 
For every $\eps>0$, $D>0$, there is a positive constant $C_{\eps,D}$ such that 
\begin{equation} \label{eq:Local_law_X}
\P\left(\absa{\frac{1}{n} \sum_{i=1}^n f_{z_0,a}(\eigX_i) - \int_{\C} f_{z_0,a}(z) \sigma(z) \di^2 z} \geq n^{-1+2a+\eps} \norm{\Delta f}_{L^1}\right) \leq \frac{C_{\eps, D}}{n^D}
\end{equation}
holds true for all $n \in \N$, for every $z_0 \in \C$ satisfying $\abs{z_0}^2 \leq \rho(S)- \zsq_*$  
and for every $f \in C_0^2(\C)$ satisfying $\supp f \subset D(0,\varphi)$.
The point $z_0$ and the function $f$ may depend on $n$. 
\item (Away from the spectrum) 
For every $D>0$, there exists a positive constant $C_D$ such that
\begin{equation} \label{eq:Local_law_X_outside}
\P\left(\exists i \in\{1, \ldots, n\} \;\middle|\; \abs{\eigX_i}^2\geq \rho(S) +\zsq_*\right) \leq \frac{C_D}{n^D}
\end{equation}
holds true for all $n \in \N$.
\end{enumerate}
In addition to the model parameters, the constant $C_{\eps,D}$ in \eqref{eq:Local_law_X} depends only on $a$, $\varphi$ and $\zsq_*$ (apart from $\eps$ and $D$) and
the constant $C_D$ in \eqref{eq:Local_law_X_outside} only on $\zsq_*$ (apart from $D$). 
\end{thm}

The key technical input for the proof of Theorem \ref{thm:local_circular_law} is the local law for $\Hf^z$ (see Theorem \ref{thm:local_law_H_z}).  
We now state a simple corollary of the local law for $\Hf^z$ on the complete  delocalization of the bulk eigenvectors of $X$. 
\begin{coro}[Eigenvector delocalization]  \label{coro:eigenvector_delocalization}
Let $\zsq_*>0$. 
For all $\eps >0$ and $D>0$, there is a positive constant $C_{\eps,D}$ such that  
\begin{equation} \label{eq:eigenvector_delocalization}
\P\left( \norminf{y} \geq n^{-1/2 + \eps}  \right) \leq \frac{C_{\eps, D}}{n^D}
\end{equation}
holds true for all $n \in \N$ and for all eigenvectors $y \in \C^{n}$  of $X$, normalized as 
$ \sum_{i=1}^n \abs{y_i}^2 =1$, 
corresponding to an eigenvalue $\eigX \in \spec X$ with $\abs{\eigX}^2 \leq \rho(S) - \zsq_*$.
The constant $C_{\eps,D}$ in \eqref{eq:eigenvector_delocalization} depends only on $\zsq_*$ and the model parameters (in addition to $\eps$ and $D$).
\end{coro}

The proof of Corollary \ref{coro:eigenvector_delocalization} will be given after the statement of Theorem \ref{thm:local_law_H_z}.
We remark that eigenvector delocalization for random matrices with independent entries was first proved by Rudelson and Vershynin in \cite{rudelson2015}.

\begin{figure}[ht!]
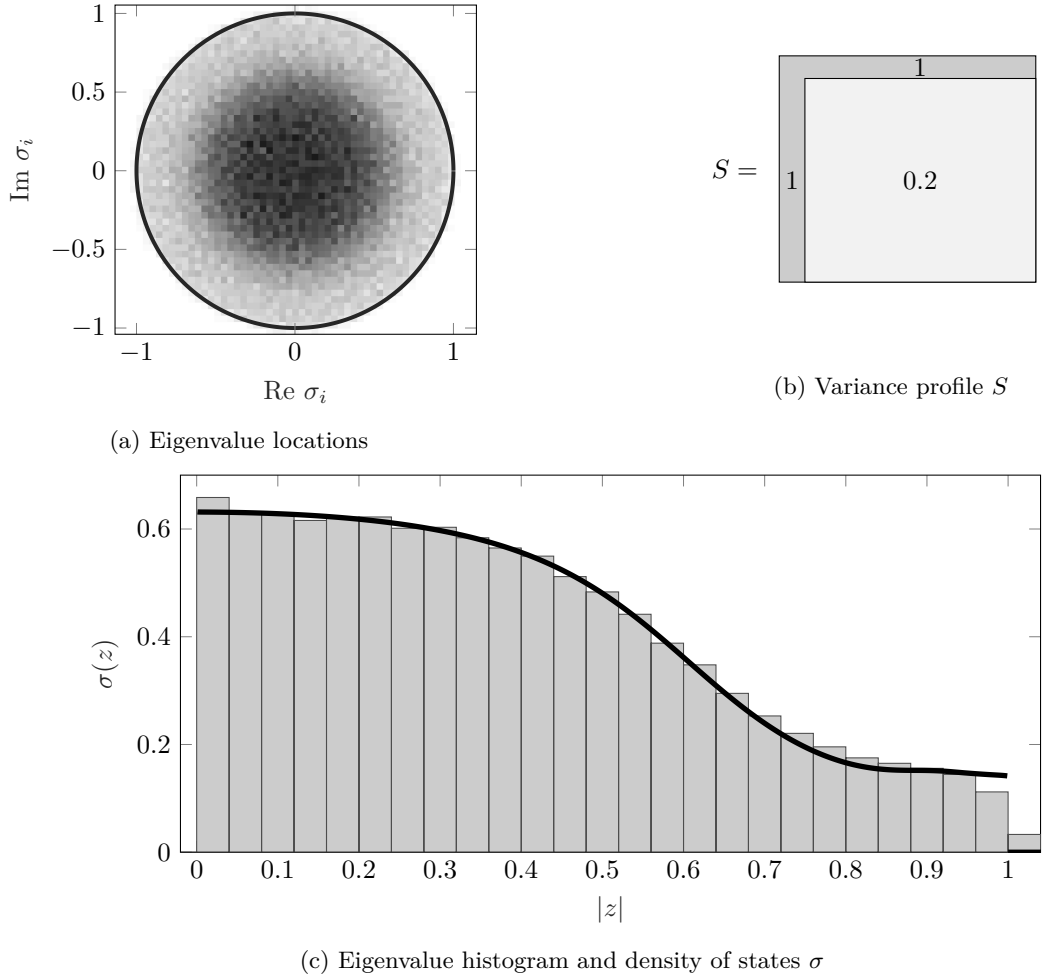

\newlength\fwidth
\newlength\fheight
\begin{subfigure}{.5\textwidth}
\centering
\setlength\fwidth{5cm}
\input{circle2.tikz} % note that in the circle file does not have a height!!!
\caption{Eigenvalue locations\label{subfig:a}}
\end{subfigure}
\begin{subfigure}{.5\textwidth}
\centering
\setlength\fwidth{4.5cm}
\setlength\fheight{3cm}
% This file was created by matlab2tikz.
%
%The latest updates can be retrieved from
%  http://www.mathworks.com/matlabcentral/fileexchange/22022-matlab2tikz-matlab2tikz
%where you can also make suggestions and rate matlab2tikz.
%
\begin{tikzpicture}

\begin{axis}[%
width=0.951\fwidth,
height=\fheight,
at={(0\fwidth,0\fheight)},
scale only axis,
clip=false,
xmin=-0.133946830265849,
xmax=1.13394683026585,
ymin=0,
ymax=1,
axis line style={draw=none},
ticks=none,
axis x line*=bottom,
axis y line*=left
]
\node[right, align=left]
at (axis cs:-0.3,0.5) {$S=$};
\draw[fill=white!80!black, draw=black] (axis cs:0,0) rectangle (axis cs:1,1);
\node[align=center]
at (axis cs:0.05,0.45) {1};
\node[align=center]
at (axis cs:0.55,0.95) {1};
\draw[fill=white!95!black, draw=black] (axis cs:0.1,0) rectangle (axis cs:1,0.9);
\node[align=center]
at (axis cs:0.55,0.45) {0.2};
\end{axis}
\end{tikzpicture}%
\vspace{1cm}
\caption{Variance profile $S$\label{subfig:b}}
\end{subfigure}

\vspace{0.2cm}

\begin{subfigure}{\textwidth}
\centering
\setlength\fwidth{.7\textwidth}
\setlength\fheight{5cm}
\input{inhom.tikz}
\caption{Eigenvalue histogram and density of states $\sigma$\label{subfig:c}}
\end{subfigure}
\caption{These figures were obtained by sampling $100$ matrices of size $1000 \times 1000$ with centered complex Gaussian entries and the variance profile $S$.
Figure (\subref{subfig:a}) shows the eigenvalue density for the variance profile $S$ given in (\subref{subfig:b}) (We rescaled $S$ such that $\rho(S) = 1$). 
The eigenvalue density is rotationally invariant and almost all eigenvalues are contained in the disk of radius $1$ around zero. Moreover, the eigenvalue density is considerably higher around $0$. 
Figure (\subref{subfig:c}) compares the histogram of the eigenvalue with the density of states $\sigma$ obtained from \eqref{eq:v} and \eqref{eq:def_sigma}.}
\end{figure}
\subsection{Short outline of the proof} \label{subsec:Idea_Proof_Circular_Law}

We start with the Hermitization trick due to Girko which expresses $\sum_{i=1}^n f_{z_0,a}(\eigX_i)$ in terms of an integral of the log-determinant of $X-z\id$ for any $z \in\C$. 
Furthermore, the log-determinant of $X-z\id$ can be rewritten as the log-determinant of a Hermitian matrix $\Hf^z$. 

Using the log-transform of the empirical spectral measure of $X$, we obtain 
\begin{equation} \label{eq:avg_f_lambda_i_as_log_det}
 \frac{1}{n} \sum_{i=1}^n f_{z_0,a}(\eigX_i)  = \frac{1}{2\pi n}\int_{\C} \Delta f_{z_0,a}(z) \log \abs{\det (X-z\id)} \di^2 z. 
\end{equation}
To express the log-determinant of $X-z\id$ in terms of a Hermitian matrix, we introduce the $2n\times 2n$ matrix
\begin{equation} \label{eq:def_H_z}
\Hf^z \defeq \begin{pmatrix} 0 & X- z\id \\ X^*-\bar z \id & 0 \end{pmatrix} 
\end{equation}
for all $z \in \C$. 
Note that the eigenvalues of $\Hf^z$ come in opposite pairs and we denote them by $\eigH_{2n} \leq \ldots \leq \eigH_{n+1} \leq 0 \leq \eigH_{n} \leq \ldots \leq \eigH_{1}$ 
with $\eigH_i = - \eigH_{2n+1-i}$ for $i =1, \ldots, 2n$. We remark that the moduli of these real numbers are the singular values of $X-z\id$. 
The Stieltjes transform of its empirical spectral measure is denoted by $m^z$, i.e., 
\begin{equation} \label{eq:def_Stieltjes_transform}
 m^z(w) = \frac{1}{2n} \sum_{i=1}^{2n} \frac{1}{\eigH_i(z)-w} 
\end{equation}
for $w \in \C$ satisfying $\Im w>0$. 
It will turn out that on the imaginary axis $\Im m^z(\ii\eta)$ is very well approximated by $\avg{v_1^\zsq(\eta)} = \avg{v_2^\zsq(\eta)}$, where $\zsq=\abs{z}^2$ 
and $(v_1^\zsq,v_2^\zsq)$ is the solution of \eqref{eq:v}. This fact is commonly 
called a \emph{local law} for $\Hf^z$. 
With this notation, we have the following relation between the determinant of $X-z\id$ and the determinant of $\Hf^z$
\begin{equation} \label{eq:log_det_X_H_z}
 \log \abs{\det(X-z\id)} = \frac{1}{2} \log \abs{\det \Hf^z}. 
\end{equation}
We write the log-determinant in terms of the Stieltjes transform (this formula was used by Tao and Vu \cite{tao2015} in a similar context)
\begin{equation} \label{eq:log_det_Stieltjes}
 \log\abs{\det \Hf^z} = \log\abs{\det(\Hf^z - \ii T\id )} - 2 n \int_0^T \Im m^z (\ii \eta)\di \eta, 
\end{equation}
for any $T>0$.
Combining \eqref{eq:def_sigma}, \eqref{eq:avg_f_lambda_i_as_log_det}, \eqref{eq:log_det_X_H_z} and \eqref{eq:log_det_Stieltjes} as well as substracting $1/(1+\eta)$ freely 
and using integration by parts, we obtain
\begin{align}
 \frac{1}{n} \sum_{i=1}^n f_{z_0,a}(\eigX_i) - \int_\C f_{z_0,a}(z) \sigma(z) \di^2 z = &  \frac{1}{4\pi n} \int_{\C} \Delta f_{z_0,a}(z) \log \abs{\det(\Hf^z-\ii T\id)} \di^2 z  \nonumber\\ 
& - \frac{1}{2\pi}\int_{\C} \Delta f_{z_0,a} (z) \int_0^T \Big[ \Im m^z(\ii \eta) - \avga{v_1^\zsq(\eta)\middle|_{\zsq=\abs{z}^2}}\Big]\,\di \eta\,\di^2 z \nonumber\\
 & + \frac{1}{2\pi} \int_\C \Delta f_{z_0,a}(z) \int_T^\infty \left(\avga{v_1^\zsq(\eta)\middle|_{\zsq=\abs{z}^2}} -\frac{1}{\eta+1}\right)\,\di \eta \,\di^2 z. \label{eq:master_formula}
\end{align}

The task is then to prove that each of the terms on the right-hand side of \eqref{eq:master_formula} is dominated by $n^{-1+2a}\norm{\Delta f}_1$  
with very high probability. The parameter $T$ will be chosen to be a large power of $n$, so that
 the first and the third term will easily satisfy this bound.
Estimating the second term on the right-hand side of \eqref{eq:master_formula} is much more involved
 and we focus only on this term in this outline. 

We split its $\di \eta$ - integral into two parts.
For $\eta \leq n^{-1+\eps}$, $\eps \in (0,1/2)$, the integral is controlled by an estimate on the smallest singular value of $X-z\id$.
This is the only step in our proof which uses assumption (C), i.e., that the entries of $X$ have bounded densities \eqref{eq:bounded_density}.

For $\eta\geq n^{-1+\eps}$,
 we use a local law for $\Hf^z$, i.e., an optimal pointwise estimate (up to negligible $n^\eps$-factors) on 
\begin{equation} \label{eq:local_law_heuristics}
\Im m^z(\ii \eta) - \avga{v_1^\zsq(\eta)\middle|_{\zsq=\abs{z}^2}},
\end{equation}
uniformly in $\eta$ and $z$ 
(see Theorem \ref{thm:local_law_H_z} for the precise formulation). Note that a local law for $\Hf^z$ is needed only at spectral parameters
on the imaginary axis. This will simplify the proof of the local law we need in this paper.

The proof of the local law is based on a stability estimate of \eqref{eq:v}. To write these equations in a more concise form, we introduce the $2n \times 2n$ matrices
\begin{equation} \label{eq:def_Sf}
 \Sf_o = \begin{pmatrix} 0 & S \\ S^t & 0 \end{pmatrix}, \quad \Sf_d = \begin{pmatrix} S^t & 0 \\ 0 & S \end{pmatrix}.  
\end{equation}
With this notation the system of equations \eqref{eq:v} can be written as
\begin{equation} \label{eq:iv_combined}
\ii \vf + \left( \ii\eta + \Sf_o\ii\vf - \frac{\zsq}{\ii\eta + \Sf_d \ii\vf}\right)^{-1} = 0,
\end{equation}
where we introduced $\vf \defeq (v_1, v_2) \in \R^{2n}$. 

Let $\Gf^z(\eta) \defeq (\Hf^z-\ii\eta\id)^{-1}$, $\eta >0$, be the resolvent of $\Hf^z$ at spectral parameter $\ii\eta$.  
We will prove that its diagonal $\gf(\eta) = (\scalar{\boldsymbol{e}_i}{\Gf^z(\eta)\boldsymbol{e}_i})_{i=1}^{2n}$, where $\boldsymbol{e}_i$ denotes the $i^{\text{th}}$ standard basis vector in $\C^{2n}$, 
satisfies a perturbed version of \eqref{eq:iv_combined}, 
\begin{equation} \label{eq:combined_v_perturbed} 
 \gf + \left( \ii\eta + \Sf_o\gf - \frac{\zsq}{\ii\eta +\Sf_d \gf} \right)^{-1} = \df,
\end{equation}
with $\zsq = \abs{z}^2$ and a small random error term $\df$.
As $m^z(\ii\eta) = \avg{\gf(\eta)}$ (cf. \eqref{eq:def_Stieltjes_transform}) obtaining a local law, i.e., an optimal pointwise estimate on \eqref{eq:local_law_heuristics}, reduces to a stability problem for the \emph{Dyson equation}  \eqref{eq:iv_combined}.

Computing the difference of \eqref{eq:combined_v_perturbed} and \eqref{eq:iv_combined}, we obtain 
\begin{equation} \label{eq:first_diff_eq}
 \Lf \left( \gf -\ii \vf\right) = \rf 
\end{equation}
for some error vector $\rf=O(\|\df\|)$
(for the precise definition we refer to \eqref{eq:def_rf} below) and with the matrix $\Lf$ defined through its action on $\yf \in \C^{2n}$ via
\begin{equation}
\Lf\yf \defeq \yf + \vf^2 (\Sf_o \yf) - \zsq \frac{\vf^2}{(\eta + \Sf_d\vf)^2} (\Sf_d \yf). \label{eq:def_Df} 
\end{equation}
Therefore, a bound on $\gf -\ii\vf$ uniformly for $\eta \ge n^{-1+\eps}$ requires a 
uniform bound on the inverse of $\Lf$ down to this local spectral scale.

In fact, the mere invertibility of $\Lf$ even for $\eta$ bounded away from zero is a nontrivial fact that is not easily seen from \eqref{eq:def_Df}. In Section \ref{sec:self_consistent_equation} 
we will factorize  $\Lf$ into the form
\[
\Lf = \Vf^{-1}(\id - \Tf \Ff)\Vf
\]
for some invertible matrix $\Vf$ and self-adjoint matrices $\Tf$ and $\Ff$ with the properties $\normtwo{\Tf}=1$ and $\normtwo{\Ff} \le 1-c\2\eta$ for some $c>0$. In particular, this representation shows the a priori bound 
$\normtwo{\Lf^{-1}}\le C\eta^{-1}$ for some $C>0$. The blow-up in the norm of $\Lf^{-1}$ is potentially caused by the two extremal eigendirections $\ffp$ and $\ffm$ of $\Ff$, which satisfy
\[
\Ff\boldsymbol{f}_{\pm} \,=\, \pm \normtwo{\Ff} \boldsymbol{f}_\pm\,.
\]
However, it turns out that the positivity of the solutions $v_1$, $v_2$ of \eqref{eq:v} implies that $\normtwo{\Tf\ffp}$ is strictly smaller than $1$, so that 
$\normtwo{(\id-\Tf\Ff)\ffp} \geq c \normtwo{\ffp}$ for some constant $c>0$. In this sense the solution of the Dyson equation regularizes the potentially unstable direction $\ffp$. 

In contrast, the other instability caused by $\ffm$ persists since we will find that
 $(\id-\Tf\Ff)\ffm = O(\eta)$. This problem can only be resolved by exploiting an extra cancellation that originates from the special structure of the 
random matrix $\Hf^z$. 
 The leading contribution of the random error $\rf=O(\norm{\df})$ from \eqref{eq:first_diff_eq} pointing in the unstable direction
happens to vanish with a remaining subleading term 
of order $\eta \norm{\df}$. The extra $\eta$-factor cancels the $\eta^{-1}$-divergence of $\normtwo{\Lf^{-1}}$ and allows us to invert the stability operator $\Lf$ in \eqref{eq:first_diff_eq}. 

From this analysis, we conclude $\norm{\gf-\ii\vf} \leq C \norm{\df}$. This result allows 
us to follow the general arguments developed in \cite{AjankiCorrelated} for verifying the optimal local law for $\Hf^z$. 
These steps are presented only briefly in Section \ref{sec:local_law}.

\section{Dyson equation for the inhomogeneous circular law} \label{sec:self_consistent_equation}

As explained in Section~\ref{subsec:Idea_Proof_Circular_Law} a main ingredient in the proof of Theorem~\ref{thm:local_circular_law} is the local law for the self-adjoint random matrix $\Hf^z$ with non-centered independent entries above the diagonal. In \cite{AjankiCorrelated} such a local law  was proven for a large class of self-adjoint random matrices with  non-centered entries and general short range correlations. 
For any fixed $z \in \C$, the matrix $\Hf^z$ satisfies the assumptions made for the class of random matrices covered in \cite{AjankiCorrelated} with one crucial exception: $\Hf^z$ is not \emph{flat}
(cf. (2.28) in \cite{AjankiCorrelated}), i.e., for any constant $c>0$, the inequality
\bels{MDE:flatness}{
\frac{1}{n}\E\2\abs{\scalar{{\boldsymbol a}}{(\Hf-\E\2\Hf){\boldsymbol b}}}^2\,\ge\, c\2\norm{{\boldsymbol a}}_2^2\norm{{\boldsymbol b}}_2^2,
}
is not satisfied for $\Hf=\Hf^z$ and vectors ${\boldsymbol a},{\boldsymbol b}$ that both have support either in $\{1,\dots,n\}$ or $\{n+1,\dots,2n\}$. 
Nevertheless we will show that the conclusion from Theorem~2.9 of \cite{AjankiCorrelated} remains true for spectral parameters $\ii\eta$ on the imaginary axis, namely that 
 the resolvent $\Gf^z(\eta)\defeq(\Hf^z -\ii\eta\id)^{-1}$ approaches the solution $\Mf^z(\eta)$ of the \emph{Matrix Dyson Equation (MDE)}
\bels{MDE on imaginary axis}{
-\Mf^z(\eta)^{-1}\,=\, \ii\1\eta\2\id- {\boldsymbol A}^z + \cal{S}[\Mf^z(\eta)]\,, \qquad \eta  > 0\,,
}
as $n\to \infty$.
In fact, the solution of \eqref{MDE on imaginary axis} is unique under the constraint that the imaginary part $\Im \Mf \defeq (\Mf - \Mf^*)/(2\ii)$ is positive definite \cite{Helton01012007}. 
The data $\bs{A}^z \in \C^{2n \times 2n}$
and $\cal{S}\colon\C^{2n \times 2n} \to \C^{2n \times 2n}$ determining \eqref{MDE on imaginary axis} are given in terms of the first and second moments of the entries of $\Hf^z$,
\bels{MDE:Data}{
{\boldsymbol A}^z\,\defeq\,\E \2\Hf^z\, =\,
 \begin{pmatrix} 0 & -z \\ -\overline{z} & 0 \end{pmatrix} \,,\qquad
 \cal{S}[\Wf]\,\defeq\, 
\begin{pmatrix} \diag (S w_2 )&0 \\ 0 & \diag (S^t w_1) \end{pmatrix}\,,
}
for an arbitrary $2n \times 2n$ matrix
\begin{equation} \label{eq:convetion_Wf}
\Wf \,=\, (w_{i j})_{i,j=1}^{2n}\,=\, 
\begin{pmatrix}W_{11}& W_{12} \\  W_{21} & W_{22} \end{pmatrix}\,,\qquad 
 w_1 \,\defeq\, (w_{ii})_{i=1}^n\,,\qquad  w_2 \,\defeq\, (w_{ii})_{i=n+1}^{2n}\,.
\end{equation}

In the following, we will not keep the $z$-dependence in our notation and just write $\Mf$, $\Af$ and $\Gf$ instead of $\Mf^z$, $\Af^z$ and $\Gf^z$.
A simple calculation (cf. the proof of Lemma~\ref{lem:existence_uniqueness_vf_equation} in the appendix)
shows that $\Mf\colon\R_+\to \C^{2n \times 2n}$ is given by
\begin{equation} \label{eq:def_Mf}
 \Mf^z( \eta) \defeq \begin{pmatrix} \ii \diag\left( v_1^\zsq(\eta)\right) & - z \diag\left(u^{\zsq}(\eta)\right) \\ - \bar z \diag\left(u^\zsq(\eta)\right) & \ii \diag \left(v_2^\zsq(\eta)\right) \end{pmatrix}, 
\end{equation}
where $z \in \C$, $\zsq=\abs{z}^2$, $(v_1^\zsq,v_2^\zsq)$ is the solution of \eqref{eq:v} and $u^\zsq \defeq v_1^\zsq/(\eta + S^tv_1^\zsq)$.  
In this section we will therefore analyze the solution and the stability of \eqref{eq:v}.

\subsection{Analysis of the Dyson equation \eqref{eq:v}}

Combining the equations in \eqref{eq:v}, recalling $\vf = (v_1, v_2)$ and the definitions of $\Sf_o$ and $\Sf_d$ in \eqref{eq:def_Sf}, we obtain
\begin{equation} \label{eq:v_combined}
\frac{1}{\vf} = \eta + \Sf_o \vf + \frac{\zsq}{\eta + \Sf_d \vf}
\end{equation}
for $\eta >0$ and $\zsq \in \Rnon$, where $\vf\colon \R_+ \to \R_+^{2n}$.  This equation is equivalent to \eqref{eq:iv_combined}. 
The $\zsq$-dependence of $\vf$, $v_1$ and $v_2$ will mostly be suppressed but sometimes we view $\vf=\vf^\zsq(\eta)$ as a function of both parameters. 

The equation \eqref{eq:v_combined} has an obvious scaling invariance when $S$ is rescaled to $\lambda S$ for $\lambda>0$. 
If $\vf^\zsq(\eta)$ is the positive solution of \eqref{eq:v_combined}, then $\vf_\lambda^\zsq(\eta) \defeq \lambda^{-1/2} \vf^{\zsq\lambda^{-1}}(\eta \lambda^{-1/2})$ is the positive solution 
of 
\begin{equation} \label{eq:v_rescaled}
 \frac{1}{\vf_\lambda} = \eta + \lambda \Sf_o \vf_\lambda + \frac{\zsq}{\eta + \lambda \Sf_d \vf_\lambda}. 
\end{equation}
Therefore, without loss of generality, we may assume that the spectral radius of $S$ is one,  \[\rho(S) = 1, \] in the 
remainder of the paper.

The following proposition, the first main result of this section, collects some basic estimates on the solution $\vf$ of \eqref{eq:v_combined}. 
For the whole section, we fix $\zsq_*>0$ and $\zsq^*>\zsq_*+1$ and except for Proposition \ref{pro:estimates_v_small_z}, we exclude 
the small interval $[1-\zsq_*,1+\zsq_*]$ from our analysis of $\vf^\zsq$. 
Because of the definition of $\sigma$ in \eqref{eq:def_sigma} -- recall $\zsq = \abs{z}^2$ in the definition -- 
we will talk about inside and outside regimes for $\zsq \in [0,1-\zsq_*]$ and $\zsq \in [1+\zsq_*, \zsq^*]$, respectively. 

Recalling $s_*$ and $s^*$ from \eqref{eq:assumption_A} we make the following convention in order to suppress irrelevant constants from the notation.

\begin{convention}
For nonnegative scalars or vectors $f$ and $g$, we will use the notation $f \lesssim g$ if there is a constant $c>0$, depending only on $\zsq_*$, $\zsq^*$, $s_*$ and $s^*$ such that $f \leq cg$ and 
$f \sim g$ if $f \lesssim g$ and $f \gtrsim g$ both hold true. 
If $f,g$ and $h$ are scalars or vectors and $h \geq 0$ such that $\abs{f-g}\lesssim h$, then we write $f = g + O(h)$.  
Moreover, we define 
\[ \para \defeq \{ \zsq_*, \zsq^*, s_*, s^* \} \]
because many constants in the following will depend only on $\para$.
\end{convention}

\begin{pro} \label{pro:estimates_v_small_z}
The solution $\vf^\zsq$ of \eqref{eq:v_combined} satisfies 
\begin{equation} \label{eq:avg_v1_equal_avg_v2}
\avg{v_1^\zsq(\eta)} = \avg{v_2^\zsq(\eta)}. 
\end{equation} 
for all $\eta>0$ and $\zsq \in \Rnon$ as well as the following estimates:
\begin{enumerate}[(i)]
\item (Large $\eta$) Uniformly for $\eta \geq 1$ and $\zsq \in [0,\zsq^*]$, we have 
\begin{equation} \label{eq:bound_vf_large_eta}
 \vf^\zsq(\eta)\sim \eta^{-1}. 
\end{equation}
\item (Inside regime) Uniformly for $\eta\leq 1$ and $\zsq \in [0,1]$, we have 
\begin{equation} \label{eq:bound_vf_small_eta}
\vf^\zsq(\eta) \sim \eta^{1/3} + (1-\zsq)^{1/2}.
\end{equation}
\item (Outside regime) Uniformly for $\eta\leq 1$ and $\zsq \in [1,\zsq^*]$, we have 
\begin{equation} \label{eq:bound_vf_small_eta_z_bigger_rho_S}
\vf^\zsq(\eta) \sim \frac{\eta}{\zsq-1 + \eta^{2/3} }.
\end{equation}
\end{enumerate}
\end{pro}

\begin{proof}[Proof of Proposition \ref{pro:estimates_v_small_z}]
We start with proving \eqref{eq:avg_v1_equal_avg_v2}.
By multiplying \eqref{eq:v_1} by $(\eta + S^t v_1)$ and \eqref{eq:v_2} by $(\eta + Sv_2)$ and realizing that both right-hand sides agree, we obtain
\begin{equation} \label{eq:step_to_def_u}
\frac{v_1}{\eta + S^tv_1} = \frac{v_2}{\eta + S v_2}. 
\end{equation}
From \eqref{eq:step_to_def_u}, we also get
\[ 0 = \eta(v_1 - v_2) + v_1 S v_2 - v_2 S^t v_1. \]
We take the average on both sides, use $\avg{v_1 Sv_2} = \scalar{v_1}{Sv_2} = \avg{v_2S^t v_1}$ and divide by $\eta >0$ to infer \eqref{eq:avg_v1_equal_avg_v2}.

From \eqref{eq:assumption_A}, we immediately deduce the following auxiliary bounds
\begin{equation} \label{eq:Sv_sim_avg_v}
 \avg{v_1} \lesssim S^t v_1 \lesssim \avg{v_1}, \quad \avg{v_{ 2 }} \lesssim Sv_2 \lesssim \avg{v_{ 2 }}.
\end{equation}
We start with establishing $\vf \sim \avg{\vf}$. 
Since the entries of $S$ are strictly positive and $\rho(S)=1$ there is a unique vector $p \in \R_+^n$ which has strictly positive entries such that 
\begin{equation} \label{eq:eigenvector_S_properties}
Sp=p, \quad \avg{p}=1, \quad p \sim 1 
\end{equation}
by the Perron-Frobenius Theorem and \eqref{eq:assumption_A}.
We multiply \eqref{eq:v_1} by $v_1$ as well as $\eta + S^tv_1$ and obtain $\eta+ S^tv_1= v_1 (\eta+ Sv_2)(\eta+S^tv_1) + \zsq v_1$. Taking the scalar product with $p$ and using 
$\avg{p} =1$ and $\rho(S)=1$ yield 
\begin{equation} \label{eq:vf_sim_avg_aux1}
 \eta + \avg{pv_1} = \avga{pv_1(\eta + S^t v_1)(\eta + Sv_2)} +  \zsq \avg{pv_1}. 
\end{equation}
Therefore, \eqref{eq:Sv_sim_avg_v}, $\avg{v_1} = \avg{v_2} = \avg{\vf}$ by \eqref{eq:avg_v1_equal_avg_v2} and \eqref{eq:eigenvector_S_properties} imply 
\begin{equation} \label{eq:vf_sim_avg_aux2}
 \eta + \avg{\vf} \sim \left[ (\eta + \avg{\vf})^2 + \zsq \right] \avg{\vf}. 
\end{equation}
We use \eqref{eq:Sv_sim_avg_v} in  \eqref{eq:v_1} and \eqref{eq:v_2} to conclude 
\begin{equation} \label{eq:estimates_on_vf_aux_est}
  \vf \sim \frac{1}{\eta + \langle \vf \rangle  + \frac{\zsq}{\eta + \langle \vf \rangle}} = \frac{\eta + \avg{\vf}}{(\eta + \avg{\vf})^2 + \zsq} \sim \avg{\vf},
\end{equation}
where we applied \eqref{eq:vf_sim_avg_aux2} in the last step. 
Hence, it suffices to prove all estimates \eqref{eq:bound_vf_large_eta}, \eqref{eq:bound_vf_small_eta} and \eqref{eq:bound_vf_small_eta_z_bigger_rho_S} for $\vf$ replaced by $\avg{\vf}$ only. 

We start with an auxiliary upper bound on $\avg{\vf}$. By multiplying \eqref{eq:v_combined} with $\vf$, we get $1 = \eta \vf + \vf \Sf_o \vf + \zsq \vf/(\eta + \Sf_d \vf) \geq \vf \Sf_o \vf.$
Hence, $1 \geq \avg{v_1 S v_2} \gtrsim \avg{v_1} \avg{v_2} = \avg{\vf}^2, $
where we used  \eqref{eq:Sv_sim_avg_v} in the second step and \eqref{eq:avg_v1_equal_avg_v2} in the last step. 

Next, we show \eqref{eq:bound_vf_large_eta}.
Clearly, \eqref{eq:v_combined} implies $\vf \leq \eta^{-1}$.  
Moreover, as $\zsq \leq \zsq^*$ and $\eta \geq 1 \gtrsim \avg{\vf}$ we find $\eta \lesssim \eta^2 \avg{\vf}$ from \eqref{eq:vf_sim_avg_aux2}. This gives the lower bound on $\vf$ in \eqref{eq:bound_vf_large_eta} 
when combined with \eqref{eq:estimates_on_vf_aux_est}.

We note that \eqref{eq:vf_sim_avg_aux2} immediately implies $\avg{\vf} \gtrsim \eta$ for $\eta \leq 1$. 
Now, we show \eqref{eq:bound_vf_small_eta}. 
For $\zsq \in [0,1]$, we bring the term $\zsq\avg{pv_1}$ to the left-hand side in \eqref{eq:vf_sim_avg_aux1} and use $v_1 \sim v_2 \sim \avg{\vf}$
 and \eqref{eq:Sv_sim_avg_v} as well as $\avg{\vf} \gtrsim \eta$ to obtain
\begin{equation} \label{eq:aux_1}
\eta + (1-\zsq) \avg{\vf} \sim \avg{\vf}^3.
\end{equation}
From \eqref{eq:aux_1}, it is an elementary exercise to conclude \eqref{eq:bound_vf_small_eta} for $\eta \leq 1$.

Similarly, for $1\leq \zsq \leq \zsq^*$, we bring $\avg{pv_1}$ to the right-hand side of \eqref{eq:vf_sim_avg_aux1}, use $\avg{\vf} \gtrsim \eta$ for $\eta \leq 1$ and conclude
\begin{equation} \label{eq:aux_4}
\eta  \sim \avg{\vf}^3 + (\zsq-1) \avg{\vf}.
\end{equation}
As before it is easy to conclude \eqref{eq:bound_vf_small_eta_z_bigger_rho_S} from \eqref{eq:aux_4}. We leave this to the reader. 
This finishes the proof of Proposition~\ref{pro:estimates_v_small_z}.
\end{proof}

Our next goal is a stability result for \eqref{eq:v_combined} in the regime $\zsq\in \Isma \cup \Ibig$.
In the following proposition, the second main result of this section, we prove that  $\ii \vf(\eta) $ 
 well approximates $\gf(\eta)$ for all $\eta >0$ if $\gf$ satisfies \eqref{eq:combined_v_perturbed} 
 and as long as $\df$ is small. 
However, we will need an additional assumption on $\gf=(g_1,g_2)$, namely that $\avg{g_1}=\avg{g_2}$ (see \eqref{eq:assumption_stability}
 below).
Note that this is imposed on the solution $\gf$ of \eqref{eq:combined_v_perturbed} and not directly on the perturbation $\df$. 
Nevertheless, in our applications, the constraint \eqref{eq:assumption_stability} will be automatically satisfied owing 
to the specific block structure of the matrix $\Hf^z$ from \eqref{eq:def_H_z}.

\begin{pro}[Stability]   \label{pro:stab_lemma}
Suppose that some functions $\df \colon \R_+ \to \C^{2n}$ and $\gf=(g_1, g_2) \colon \R_+ \to \Hb^{2n}$ satisfy \eqref{eq:combined_v_perturbed} and 
\begin{equation} \label{eq:assumption_stability}
\avg{g_1(\eta)} = \avg{g_2(\eta)}
\end{equation}
for all $\eta >0$. There is a number $\lambda_* \gtrsim 1$, depending only on $\para$, such that 
\begin{equation} \label{eq:stability_estimate1}
\norm{\gf(\eta) - \ii \vf(\eta)}_\infty \cdot \char\Big(\norm{\gf(\eta) - \ii \vf(\eta)}_\infty \leq \lambda_*\Big) \lesssim \norm{\df(w)}_\infty 
\end{equation}
uniformly for $\eta >0$ and $\zsq \in \Isma\cup\Ibig$.

Moreover, there is a matrix-valued function $\Rf \colon \R_+ \to \C^{2n\times 2n}$, depending only on $\zsq$ and $S$ and satisfying $\norminf{\Rf(\eta)} \lesssim 1$, such that 
\begin{equation} \label{eq:stability_estimate_average}
\abs{\langle \yf, \gf(\eta)-\ii\vf(\eta)\rangle} \cdot \char\Big(\norm{\gf(\eta) - \ii\vf(\eta)}_\infty \leq \lambda_*\Big) \lesssim \norm{\yf}_\infty\norm{\df(\eta)}_\infty^2 + \abs{\langle \Rf(\eta)\yf, \df(\eta)\rangle}
\end{equation}
uniformly for all $\yf \in \C^{2n}$, $\eta >0$ and $\zsq \in \Isma\cup\Ibig$.
\end{pro}

The proof of this result is based on deriving a quadratic equation for the difference $\hf\defeq\gf - \ii \vf$ and establishing a quantitative estimate on $\hf$ in terms of the perturbation $\df$. 
Computing the difference of \eqref{eq:combined_v_perturbed} and \eqref{eq:iv_combined}, we obtain an equation for $\gf - \ii \vf$. A straightforward calculation yields 

\begin{equation} \label{eq:stability_equation}
\Lf  \hf = \rf, \quad \text{for } \hf = \gf - \ii\vf,
\end{equation}
where we used $\Lf$ defined in \eqref{eq:def_Df} and introduced the vector $\rf$ through
\begin{equation} 
\rf \defeq \df + \ii\vf(\hf-\df)\Sf_o\hf - \zsq \uf \left[ \frac{\df-\gf}{\ii\eta+\Sf_d\gf} + \uf\right] \Sf_d \hf.\label{eq:def_rf}
\end{equation}
The vector $\uf$ in \eqref{eq:def_rf} is defined through 
\begin{equation} \label{eq:def_u}
 u \defeq \frac{v_1}{\eta + S^tv_1} = \frac{v_2}{\eta + S v_2}, \qquad \uf \defeq (u, u) = \frac{\vf}{\eta+ \Sf_d \vf}
\end{equation}
which is consistent by \eqref{eq:step_to_def_u}. 

Notice that all terms on the right-hand side of \eqref{eq:def_rf} are either second order in $\hf$ or they are of order $\df$, so \eqref{eq:stability_equation}
 is the linearization of \eqref{eq:combined_v_perturbed}  around  \eqref{eq:iv_combined}. 

In the following estimates, we need a bound on $\uf$ as well. Indeed, Proposition~\ref{pro:estimates_v_small_z} yields 
\begin{equation} \label{eq:estimate_uf}
\uf = \frac{\vf}{\eta + \Sf_d \vf} \sim \frac{1}{1 +\eta^2} 
\end{equation}
uniformly for $\eta >0$ and $\zsq \in [0,\zsq^*]$.

To shorten the upcoming relations, we introduce the vector 
\[ \wt\vf \defeq (v_2, v_1) \]
 and the matrices $\Tf$, $\Ff$ and $\Vf$ defined by their action on a vector $\yf =(y_1,y_2)$, $y_1, y_2 \in \C^n$ as follows
\begin{subequations} \label{eq:def_matrices}
\begin{align}
\Tf \yf  & \defeq  \frac{1}{\uf} \begin{pmatrix} - v_1 v_2 y_1 + \zsq u^2 y_2 \\ \zsq u^2 y_1 -v_1 v_2 y_2 \end{pmatrix}, \label{eq:def_Tf} \\
\Ff\yf  & \defeq  
  \sqrt \frac{\vf\uf}{\wt\vf} \Sf_o\left(\sqrt{\frac{\vf\uf}{\wt\vf}}\, \yf \right)
, \label{FF} \\
\Vf \yf   & \defeq 
 \sqrt\frac{\wt\vf}{\uf \vf}\, \yf.
\end{align}
\end{subequations}
All these matrices are functions of $\eta$ and $\zsq$. 
 They provide a crucial factorization of the  stability operator $\Lf$; indeed, a simple calculation shows that
\begin{equation} \label{eq:transformation_Lf}
\Lf = \Vf^{-1} ( \id -\Tf\Ff)\Vf.
\end{equation}
This factorization reveals many properties of $\Lf$ which are difficult to observe directly. 
 Owing to \eqref{eq:stability_equation}, the stability analysis of \eqref{eq:v_combined} requires a control on the invertibility of the matrix $\Lf$.
The matrices $\Vf$ and $\Vf^{-1}$ are harmless.  
A good understanding of the spectral decompositions of  the simpler matrices $\Ff$ and $\Tf$ will then yield 
that $\Lf$ has only one direction, in which its inverse is not bounded.  We remark that 
the factorization \eqref{eq:transformation_Lf} is the diagonal part of the one used  in the stability analysis of the matrix Dyson equation in \cite{AjankiCorrelated}. 

Because of \eqref{eq:transformation_Lf}, we can study the stability of 
\begin{equation}  \label{eq:transformed_stability_equation}
(\id-\Tf\Ff) (\Vf \hf) = \Vf\rf
\end{equation} 
instead of \eqref{eq:stability_equation}. 
From Proposition \ref{pro:estimates_v_small_z} and \eqref{eq:estimate_uf}, we conclude that 
\begin{equation} \label{eq:control_norm_V_V_inverse}
\norminf{\Vf}\norminf{\Vf^{-1}} \lesssim 1
\end{equation}
uniformly for all $\eta >0$ and $\zsq \in \Isma \cup \Ibig$. Hence, it suffices to control the invertibility of $\id - \Tf\Ff$. 

For later usage, we derive two relations for $\uf$. From \eqref{eq:def_u}, recalling $\wt\vf = (v_2, v_1)$, we immediately get
\begin{equation} \label{eq:relation_wt_vf_div_uf}
 \frac{\wt \vf}{\uf} =  \eta + \Sf_o \vf. 
\end{equation}
We multiply \eqref{eq:v_combined} by $\vf \uf$ and use \eqref{eq:relation_wt_vf_div_uf} to obtain 
\begin{equation} \label{eq:defining_eq_v_wt_v_u}
\uf = \vf \wt \vf + \zsq \uf ^2, \qquad 1 = \frac{\vf\wt\vf}{\uf} + \zsq \uf. 
\end{equation}

The next lemma collects some properties of $\Ff$. For this formulation, we introduce 
\[\emin \defeq (1,-1) \in \C^{2n}. \]
\begin{lem}[Spectral properties of $\Ff$]  \label{lem:prop_Ff}
The eigenspace of $\Ff$ corresponding to its largest eigenvalue $\normtwo{\Ff}$ is one dimensional. 
It is spanned by a unique positive normalized eigenvector $\ffp$, i.e., $\Ff \ffp = \normtwo{\Ff} \ffp$ and $\normtwo{\ffp}=1$. 
For every $\eta >0$, the norm of $\Ff$ is given by 
\begin{equation} \label{eq:norm_F}
\normtwo{\Ff} = 1 - \eta \frac{\avga{\ffp \sqrt{\vf/(\eta + \Sf_o \vf)}}}{\avga{\ffp \sqrt{\vf(\eta + \Sf_o \vf)}}}.
\end{equation}
Defining $\ffm \defeq \ffp \emin$, we have 
\begin{equation} \label{eq:eigenvalue_relations_ffm}
 \Ff\ffm = - \normtwo{\Ff} \ffm.
\end{equation}
\begin{enumerate}[(i)]
\item  (Inside regime)
The following estimates hold true uniformly for $\zsq\in\Isma$. We have 
\begin{equation}  \label{eq:bound_norm_F_small_eta}
1 -\normtwo{\Ff} \sim \eta.
\end{equation}
uniformly for $\eta \in(0,1]$. Furthermore, uniformly for $\eta \geq 1$, we have
\begin{equation}  \label{eq:estimate_norm_F_small_z_large_eta}
1-\normtwo{\Ff} \sim 1. 
\end{equation}

Moreover, uniformly for $\eta \in(0,1]$, 
$\ffp$ satisfies 
\begin{equation} \label{eq:ff_+_order_1}
\ffp \sim 1 
\end{equation} 
and there is $\eps \sim 1$ such that 
\begin{equation} \label{eq:gap_F_order_1}
\normtwo{\Ff \xf} \leq (1-\eps)\normtwo{\xf}
\end{equation}
for all $\xf \in \C^{2n}$ satisfying $\xf \perp \ffp$ and $\xf \perp \ffm$.  
\item (Outside regime)
Uniformly for all $\eta >0$ and $\zsq\in \Ibig$, we have 
\begin{equation}  \label{eq:estimate_norm_F_large_z}
1 -\normtwo{\Ff} \sim 1. 
\end{equation} 
\end{enumerate}
\end{lem}

\begin{proof} 
The statements about the eigenspace corresponding to $\normtwo{\Ff}$ and $\ffp$ follow from Lemma 3.3 in \cite{AltGram}.

For the proof of \eqref{eq:norm_F}, we multiply \eqref{eq:v_combined} by $\vf$ and take the scalar product of the resulting relation with 
$\ffp \sqrt{\uf/(\vf\wt \vf)}$.  Using that
$$
     \scalara{\ffp\sqrt{\frac{\uf}{\vf\wt \vf}} }{\vf \Sf_o \vf} =  \scalara{\ffp\sqrt{\frac{\vf\uf}{\wt \vf}} }{\Sf_o \vf}
     =  \scalara{ \Sf_o\Bigg( \ffp\sqrt{\frac{\vf\uf}{\wt \vf}}\Bigg) }{\vf} =  \scalara{ \sqrt{\frac{\wt \vf}{\vf\uf}}  \Ff\ffp}{\vf}=
      \normtwo{\Ff}\scalara{\ffp}{\sqrt{\frac{\vf\wt \vf}{\uf}}},
$$
this yields 
\[ \normtwo{\Ff} \scalara{\ffp}{\sqrt{\frac{\vf\wt \vf}{\uf}}} = \scalara{\ffp\sqrt{\frac{\uf}{\vf\wt \vf}}}{1- \zsq \uf} - \eta \scalara{\ffp \sqrt{\frac{\uf}{\vf\wt\vf}}}{\vf}. 
\]
We conclude \eqref{eq:norm_F} from applying \eqref{eq:defining_eq_v_wt_v_u} and \eqref{eq:relation_wt_vf_div_uf} to the last relation.

Since $\Ff$ from \eqref{FF} has the form \[ \Ff = \begin{pmatrix} 0 & F \\ F^t & 0 \end{pmatrix} , \] for some $F \in \C^{n\times n}$ 
we have $\Ff(\emin \yf)=-\emin (\Ff\yf)$ for all $\yf\in \C^{2n}$. Thus, we get \eqref{eq:eigenvalue_relations_ffm} from $\Ff\ffp = \normtwo{\Ff} \ffp$. 

In the regime $\zsq\in \Isma$ and $\eta\in (0,1]$,
 we have uniform lower and upper bounds on $\vf$ from Proposition~\ref{pro:estimates_v_small_z}. Therefore, 
the estimates in \eqref{eq:ff_+_order_1} and \eqref{eq:gap_F_order_1} follow from Lemma 3.3 in \cite{AltGram}.
Combining \eqref{eq:ff_+_order_1}, \eqref{eq:norm_F} and Proposition \ref{pro:estimates_v_small_z} yields \eqref{eq:bound_norm_F_small_eta}. 
In the large $\eta$ regime, i.e., for $\eta\ge 1$, since $\vf \sim \eta^{-1}$ by Proposition~\ref{pro:estimates_v_small_z} we obtain
\begin{equation} \label{eq:estimate_F_aux_1}
 \frac{\vf}{\eta + \Sf_o \vf} \sim \eta^{-2}, \quad \vf (\eta + \Sf_o \vf) \sim 1. 
\end{equation}
Hence, as $\ffp >0$ we conclude
\begin{equation} \label{eq:estimate_F_aux_2}
 \frac{\avga{\ffp \sqrt{\vf/(\eta + \Sf_o \vf)}}}{\avga{\ffp \sqrt{\vf(\eta + \Sf_o \vf)}}} \sim \frac{\avg{\ffp}}{ \avg{\ffp}} \frac{1}{\eta} =  \frac{1}{\eta}, 
\end{equation}
uniformly for all $\eta \geq 1$. This shows that \eqref{eq:estimate_norm_F_small_z_large_eta} holds true for all $\eta \geq 1$ and $\zsq\in \Isma$.

We now turn to the proof of (ii). If $\zsq \in \Ibig$, then $\vf \sim \eta$ by \eqref{eq:bound_vf_small_eta_z_bigger_rho_S} for $\eta \leq 1$ and therefore
 \[ \frac{\vf}{\eta + \Sf_o \vf } \sim 1, \quad \vf (\eta + \Sf_o \vf) \sim \eta^{2}. \]
As $\ffp >0$, we thus have 
\begin{equation}\label{etafraction}
\eta \frac{\avga{\ffp \sqrt{\vf/(\eta + \Sf_o \vf)}}}{\avga{\ffp \sqrt{\vf(\eta + \Sf_o \vf)}}}    \sim \frac{\avg{\ffp}}{\avg{\ffp}} = 1. 
\end{equation}
For $\eta \geq 1$, we argue as in \eqref{eq:estimate_F_aux_1} and \eqref{eq:estimate_F_aux_2} and arrive 
at the same conclusion \eqref{etafraction}.  Thus, because of \eqref{eq:norm_F} the estimate \eqref{eq:estimate_norm_F_large_z} holds true for all $\eta >0$ and $\zsq\in \Ibig$.
\end{proof} 
Next, we give an approximation for the eigenvector $\ffm$ belonging to the isolated single eigenvalue $-\normtwo{\Ff}$ of $\Ff$ by constructing an approximate eigenvector. 
For $\eta \leq 1$ and $\zsq \in \Isma$, we define
\begin{equation} \label{eq:def_af}
\af \defeq \frac{\emin (\Vf\vf)}{\normtwo{\Vf\vf}}
\end{equation}
which is normalized as $\normtwo{\emin (\Vf\vf)}=\normtwo{\Vf\vf}$. We compute
\begin{equation} \label{eq:ffp}
\Ff(\Vf \vf) = \sqrt{\frac{ \uf}{\vf\wt \vf}}  \vf \left(\Sf_o \vf\right) = \sqrt{\frac{\uf}{\vf\wt \vf}} \left( 1 - \eta \vf - \zsq \uf \right) = \sqrt{\frac{\vf \wt \vf}{\uf}} - \eta \vf \sqrt{\frac{\uf}{\vf\wt \vf}} 
= \normtwo{\Ff} \Vf\vf + O(\eta). 
\end{equation}
Here, we used $\vf\Sf_o \vf = -\eta \vf + \vf\wt \vf/\uf$ by \eqref{eq:relation_wt_vf_div_uf}. 
 For estimating the $O(\eta)$ term we applied \eqref{eq:bound_vf_small_eta}, \eqref{eq:estimate_uf} and \eqref{eq:bound_norm_F_small_eta} since $\zsq \in \Isma$ and $\eta \leq 1$.
Using the block structure of $\Ff$ as in the proof of \eqref{eq:eigenvalue_relations_ffm}, we obtain  
\begin{equation} \label{eq:af_approximating_eigenvector}
\Ff(\emin (\Vf \vf)) = - \normtwo{\Ff} \emin (\Vf\vf) + O(\eta). 
\end{equation}
The following lemma states that $\af$ approximates the nondegenerate eigenvector $\ffm$. 

\begin{lem} \label{lem:approximating_eigenvector}
The eigenvector $\ffm$ can be approximated by $\af$ in the $\ell^\infty$-norm, i.e., 
\begin{equation}
\norminf{\ffm - \af} = O(\eta) \label{eq:estimate_ff_-_inf}
\end{equation}
uniformly for $\eta \leq 1$ and $\zsq\in \Isma$.
\end{lem}

This lemma is proved in Appendix \ref{proof:approximating_eigenvalue}.
In the following lemma, we show some properties of $\Tf$.

\begin{lem}[Spectral properties of $\Tf$] \label{lem:pro_Tf}
The symmetric operator $\Tf$, defined in \eqref{eq:def_Tf}, satisfies 
\begin{enumerate}[(i)]
\item $\normtwo{\Tf} = 1$, $\norminf{\Tf} =1$. 
\item The spectrum of $\Tf$ is given by 
\[ \spec(\Tf) = \{-1\} \cup \left\{ \zsq \uf_i - \frac{(\vf\wt \vf)_i}{\uf_i} \;\middle| \; i=1, \ldots, n \right\}.\]
\item For all $\eta >0$, we have $\Tf(\zsq=0) = -\id$ and if $\zsq> 0$, then the eigenspace of $\Tf$ corresponding to the eigenvalue $-1$ is  $n$-fold degenerate and given by 
\begin{equation} \label{eq:Eigenspace_Tf_minus_one}
 \mathrm{Eig}(-1, \Tf) = \left\{ (y,-y)\middle| y \in \C^n \right\}.  
\end{equation}
\item The spectrum of $\Tf$ is strictly away from one, i.e., there is $\eps>0$, depending only on $\para$, 
such that 
\begin{equation}  \label{eq:spec_Tf}
\spec(\Tf) \subset [-1, 1-\eps] 
\end{equation}
uniformly for $\zsq\in \Isma$ and $\eta\in(0,1]$.  
\end{enumerate}
\end{lem}

\begin{proof}
The second relation in \eqref{eq:defining_eq_v_wt_v_u} implies $\norminf{\Tf} =1$ and $\Tf(\zsq=0) = -\id$.
Moreover, it yields that all vectors of the form 
$(y,-y)$ for $y \in \C^n$ are contained in $\mathrm{Eig}(-1, \Tf)$. 
We define the vector $\yf^{(j)} \in \C^{2n}$ by $\yf^{(j)} \defeq (\delta_{i,j} + \delta_{i,j+n})_{i=1}^{2n}$ and observe that 
\[\Tf \yf^{(j)} = \left(\zsq \uf_j - \frac{(\vf\wt\vf)_j}{\uf_j}\right) \yf^{(j)} \]
for $j = 1, \ldots, n$.
Counting dimensions implies that we have found all eigenvalues, hence (ii) follows. 
For $\zsq>0$, we have $\zsq\uf_j - (\vf\wt\vf)_j/\uf_j = 2\zsq \uf_j -1 >-1$ by \eqref{eq:defining_eq_v_wt_v_u} and $\uf_j >0$ for all $j=1, \ldots, n$.
This yields the missing inclusion in \eqref{eq:Eigenspace_Tf_minus_one}.
Since $\Tf$ is a symmetric operator, $\normtwo{\Tf}=1$ follows from (ii) and $\abs{\zsq \uf - \vf\wt\vf/\uf} \leq 1$ by \eqref{eq:defining_eq_v_wt_v_u}. 

For the proof of (iv), we remark that there is $\eps >0$, depending only on $\para$, such that $2\vf\wt\vf/\uf \geq \eps$ for all $\eta \in (0,1]$ and $\zsq\in \Isma$ by \eqref{eq:bound_vf_small_eta}
and \eqref{eq:estimate_uf}.
Thus, \[ \zsq \uf - \frac{\vf\wt\vf}{\uf} = 1 - 2 \frac{\vf\wt\vf}{\uf} \leq 1- \eps \]
 by \eqref{eq:defining_eq_v_wt_v_u}. This concludes the proof of the lemma.
\end{proof}
Now we are ready to give a proof of Proposition \ref{pro:stab_lemma} based on inverting $\id - \Tf\Ff$.

\begin{proof}[Proof of Proposition \ref{pro:stab_lemma}]
We recall that $\hf = \gf-\ii \vf$. 
Throughout the proof we will omit arguments, but we keep in mind that $\gf$, $\df$, $\hf$ and $\vf$ depend on $\eta$ and $\zsq$.  
The proof will be given in three steps. 

The first step is to control $\norminf{\rf}$ from \eqref{eq:def_rf} in terms of $\norminf{\hf}^2$ and $\norminf{\df}$, i.e., to show 
\begin{equation} \label{eq:first_step}
 \norminf{\rf}\char\big(\norminf{\hf} \leq 1\big)\lesssim \norminf{\hf}^2 + \norminf{\df}.
\end{equation}
Inverting  $\Vf^{-1}(\id-\Tf\Ff)\Vf$ in \eqref{eq:transformed_stability_equation}, controlling the norm of the
inverse and choosing $\lambda_*\leq 1$ small enough, we will conclude Proposition \ref{pro:stab_lemma} from \eqref{eq:first_step}. 
For any $\eta_*\in(0,1]$, 
depending only on $\para$, this argument will be done in the second step
 for $\zsq \in \Isma\cup\Ibig$ and $\eta \geq \eta_*$ as well as for $\zsq \in \Ibig$ and $\eta \in (0,\eta_*]$. 
In the third step, we consider the most interesting regime $\zsq \in \Isma$ and $\eta \leq \eta_*$ for a sufficiently 
small $\eta_*$, depending on $\para$ only. In this regime, we will use an extra cancellation for the contribution of $\rf$ in the unstable direction of $\Lf$. 

\vspace*{0.1cm}
\begin{tabular}{cl}
\emph{Step 1:}& For all $\eta >0$ and $\zsq \in \Isma\cup\Ibig$, \eqref{eq:first_step} holds true.
\end{tabular}
\vspace*{0.1cm}

\noindent From \eqref{eq:combined_v_perturbed}, we obtain
\[ \zsq\frac{\gf-\df}{\ii\eta + \Sf_d \gf}= 1 + (\ii\eta +\Sf_o\gf)(\gf -\df).\]
We start from \eqref{eq:def_rf}, use the previous relation, $\zsq \uf = 1+\ii \vf(\ii\eta + \Sf_o\ii\vf)$ by \eqref{eq:v_combined} and $\wt\vf =(v_2, v_1)=\uf(\eta +\Sf_o\vf)$ by \eqref{eq:defining_eq_v_wt_v_u} and get 
\begin{align}
\rf & = \df + \ii\vf(\hf-\df)\Sf_o\hf - \uf \left[ \ii\vf(\ii\eta +\Sf_o \ii\vf) - (\gf-\df)(\ii\eta+\Sf_o\gf)\right]\Sf_d \hf \nonumber\\
 & = \df + \ii\vf(\hf-\df)\Sf_o\hf + \uf \left[ \hf(\ii\eta+\Sf_o\ii\vf) + \gf\Sf_o\hf \right] \Sf_d \hf  -\df\uf(\ii\eta+\Sf_o\gf)\Sf_d\hf \nonumber\\
& = \ii\vf \hf\Sf_o\hf + \ii \wt \vf \hf\Sf_d\hf +\uf\gf\Sf_o\hf\Sf_d\hf + \df  -\ii\vf\df\Sf_o\hf -\df\uf(\ii\eta+\Sf_o\gf)\Sf_d\hf. \label{eq:represenation_rf_5}
\end{align}
Notice that the first three terms are quadratic in $\hf$ (the linear terms dropped out), while the last three
terms are controlled by $\df$. Now, we show that all other factors are bounded and hence irrelevant 
whenever $\norminf{\gf - \ii \vf} \leq \lambda_*$ for $\eta >0$ and $\zsq \in \Isma \cup \Ibig$.
In this case, we conclude $\norminf{\gf} \lesssim 1$ uniformly for all $\eta >0$ and $\zsq \in \Isma \cup \Ibig$
by \eqref{eq:bound_vf_large_eta} and \eqref{eq:bound_vf_small_eta} from Proposition \ref{pro:estimates_v_small_z}. 
Therefore, starting from \eqref{eq:represenation_rf_5} and using $\norminf{\vf} \lesssim 1$ by \eqref{eq:bound_vf_large_eta} and \eqref{eq:bound_vf_small_eta},
and $\norminf{\uf}\lesssim 1$ by \eqref{eq:estimate_uf}, we obtain \eqref{eq:first_step}.

\vspace*{0.1cm}
\begin{tabular}{cl}
\emph{Step 2:}& 
For any $\eta_* \in(0,1]$, there exists $\lambda_* \gtrsim 1$, depending only on $\para$ and $\eta_*$, such that \eqref{eq:stability_estimate1}
holds true \\ 
& for $\eta \geq \eta_*$ and $\zsq \in \Isma\cup\Ibig$ as well as for $\eta\in (0,\eta_*]$ and $\zsq \in \Ibig$. \\
& Moreover, with this choice of $\lambda_*$, \eqref{eq:stability_estimate_average} holds true in these $(\eta, \zsq)$ parameter regimes as well.
\end{tabular}
\vspace*{0.1cm} \\
Within Step 2, we redefine the comparison relation to depend both on $\para$ and $\eta_*$.
Later in Step 3 we will choose an appropriate $\eta_*$ depending only on $\para$, so eventually the comparison
relations for our choice will depend only on $\para$.

We are now working in the regime, where $\eta \geq \eta_*$ and $\zsq \in \Isma\cup\Ibig$ or  $\eta \in (0,\eta_*]$ and $\zsq \in \Ibig$.
In this case, to prove \eqref{eq:stability_estimate1}, we invert $\Lf = \Vf^{-1}(\id - \Tf\Ff)\Vf$ (cf. \eqref{eq:def_Df}) in $\Lf \hf = \rf$, bound $\norminf{\Lf^{-1}}\lesssim 1$, 
which is proved below, and conclude 
\[ \norminf{\hf} \char\big(\norminf{\hf} \leq 1\big) \lesssim \norminf{\hf}^2  + \norminf{\df} \]
from \eqref{eq:first_step} for $\eta \geq \eta_*$ and $\zsq \in \Isma\cup\Ibig$ as well as for $\eta \in (0,\eta_*]$ and $\zsq\in \Isma$. This means that there are $\Psi_1, \Psi_2 \sim 1$ such that 
\[ \norminf{\hf} \char\big(\norminf{\hf} \leq 1\big) \leq \Psi_1\norminf{\hf}^2 + \Psi_2\norminf{\df}. \]
Choosing $\lambda_* \defeq \min\{1, (2\Psi_1)^{-1}\}$ this yields 
\[ \norminf{\hf} \char\big(\norminf{\hf} \leq \lambda_*\big) \leq 2\Psi_2\norminf{\df}. \]
Thus, we are left with controlling $\norminf{\Lf^{-1}}$, i.e., proving $\norminf{\Lf^{-1}} \lesssim 1$.

In the regime $\eta \geq \eta_*$ and $\zsq \in \Isma\cup\Ibig$, we have $\vf \sim 1/\eta$ by Proposition \ref{pro:estimates_v_small_z} and $\uf \sim 1/\eta^2$ by \eqref{eq:estimate_uf}. 
Hence, $\Vf\sim \eta$ and $\Vf^{-1} \sim 1/\eta$. Therefore, $\norminf{\Vf}\norminf{\Vf^{-1}} \lesssim 1$ and 
due to $\norminf{\Lf^{-1}} \lesssim \norminf{\Vf^{-1}} \norminf{(\id-\Tf\Ff)^{-1}} \norminf{\Vf}$, it suffices to show 
$\norminf{(\id-\Tf\Ff)^{-1}}\lesssim 1$. Basic facts on the operator $\id - \Tf\Ff$ are collected in Lemma \ref{lem:TwoNorms_to_InfNorms} in the appendix. In particular, because of \eqref{eq:control_infnorm_with_twonorm_for_TF}, 
the $\ell^\infty$ bound follows from $\normtwo{(\id-\Tf\Ff)^{-1}}\lesssim 1$. 
Using \eqref{eq:bound_norm_F_small_eta},  \eqref{eq:estimate_norm_F_small_z_large_eta} and \eqref{eq:estimate_norm_F_large_z}, we get that 
$1- \normtwo{\Ff} \sim 1$ for all $\eta \geq \eta_*$ and $\zsq \in \Isma\cup\Ibig$. Hence, $1-\normtwo{\Tf\Ff} \sim 1$ by Lemma \ref{lem:pro_Tf} (i), so the bound $\normtwo{(\id-\Tf\Ff)^{-1}}\lesssim1 $ immediately follows.
This proves \eqref{eq:stability_estimate1} for $\eta \geq \eta_*$ and $\zsq \in \Isma\cup\Ibig$.

For $\eta \leq \eta_*$ and $\zsq \in \Ibig$, we have $\vf \sim \eta$ by \eqref{eq:bound_vf_small_eta_z_bigger_rho_S}, $\uf \sim 1$ by \eqref{eq:estimate_uf}. Thus,
$\Vf \sim 1$, $\Vf^{-1} \sim 1$ as well as $\norminf{\Vf}\norminf{\Vf^{-1}} \lesssim 1$. 
As above it is enough to show $\normtwo{(\id-\Tf\Ff)^{-1}}\lesssim1 $. 
By Lemma \ref{lem:pro_Tf} (i) and \eqref{eq:estimate_norm_F_large_z}, $1- \normtwo{\Tf\Ff} \sim 1$ which again leads to $\normtwo{(\id-\Tf\Ff)^{-1}}\lesssim1 $.
We conclude \eqref{eq:stability_estimate1} for $\eta \leq \eta_*$ and $\zsq \in \Ibig$.

Next, we verify \eqref{eq:stability_estimate_average} in these two regimes. 
Using $\hf \cdot \char(\norminf{\hf} \leq \lambda_*) = O(\norminf{\df})$ by \eqref{eq:stability_estimate1},
$\vf \lesssim 1$ and $\uf\lesssim 1$, 
we see that with the exception of $\df$, all terms in \eqref{eq:represenation_rf_5} are second order in $\df$. Therefore,
\begin{equation} \label{eq:second_estimate_V_rhs}
\rf \cdot \char(\norminf{\hf} \leq \lambda_*)  = \df\cdot \char(\norminf{\hf} \leq \lambda_*) + O\left(\norminf{\df}^2\right)
\end{equation}
uniformly for $\eta \geq \eta_*$ and $\zsq \in \Isma \cup \Ibig$ as well as for $\eta \in (0,\eta_*]$ and $\zsq \in \Ibig$.

We start from $\Lf\hf = \rf$ and compute
\begin{equation}  \label{eq:computation_average_version_stab_estimate}
 \scalar{\yf}{\hf} = \scalar{(\Lf^{-1})^*\yf}{\rf} = \scalar{(\Lf^{-1})^*\yf}{\df} + 
\scalar{(\Lf^{-1})^*\yf}{\rf - \df} = \scalar{\Rf\yf}{\df} + \scalar{(\Lf^{-1})^*\yf}{\rf -\df}. 
\end{equation}
Here, we defined the operator $\Rf = \Rf(\eta)$ on $\C^{2n}$ in the last step through its action on any $\xf \in \C^{2n}$ via 
\begin{equation} \label{eq:def_Rf}
\Rf \xf\defeq \left(\Lf^{-1}\right)^* \xf = \Vf^{-1}( \id - \Ff\Tf)^{-1} \Vf \xf. 
\end{equation}
Now, we establish that $\norminf{(\Lf^{-1})^*} \lesssim 1$ in the two regimes considered in Step 2. 
From this, we conclude that $\norminf{\Rf} \lesssim 1$ and that the last term in \eqref{eq:computation_average_version_stab_estimate}
when multiplied by $\char(\norminf{\hf} \leq \lambda_*)$ is bounded by $\lesssim \norminf{\yf}\norminf{\df}^{2}$ because of \eqref{eq:second_estimate_V_rhs}. 
By Lemma \ref{lem:pro_Tf} (i),  \eqref{eq:bound_norm_F_small_eta}, \eqref{eq:estimate_norm_F_small_z_large_eta} and 
\eqref{eq:estimate_norm_F_large_z} we have 
$1-\normtwo{\Ff\Tf} \sim 1$. Thus, $\normtwo{(\id - \Ff\Tf)^{-1}} \lesssim 1$ and hence 
$\norminf{(\id - \Ff\Tf)^{-1}} \lesssim 1$ by Lemma \ref{lem:TwoNorms_to_InfNorms} (ii). As $\norminf{\Vf}\norminf{\Vf^{-1}} \lesssim 1$ we get $\norminf{(\Lf^{-1})^*} \lesssim 1$. 
Therefore, we conclude that \eqref{eq:stability_estimate_average} holds true uniformly for $\eta \geq \eta_*$ 
and $\zsq\in \Isma\cup\Ibig$ as well as for $\eta \in (0,\eta_*]$ and $\zsq \in \Ibig$.
Thus, we have proved the proposition for these combinations of $\eta$ and $\zsq$. 

Finally, we prove the proposition in the most interesting regime, $\zsq \in \Isma$ and for small $\eta$: 

\vspace*{0.1cm}
\begin{tabular}{cl}
\emph{Step 3:}& There exists $\eta_*>0$, depending only on $\para$, and $\lambda_* \gtrsim 1$ such that \eqref{eq:stability_estimate1}
holds true for $\eta\in (0,\eta_*]$ \\ &and $\zsq \in \Isma$. 
Moreover, with this choice of $\lambda_*$, \eqref{eq:stability_estimate_average} holds true for $\eta\in (0,\eta_*]$ and \\ &$\zsq \in \Isma$. 
\end{tabular}
\vspace*{0.1cm}

\noindent The crucial step for proving \eqref{eq:stability_estimate1} and \eqref{eq:stability_estimate_average} was the 
order one bound on $\normtwo{(\id-\Tf\Ff)^{-1}}$. 
However, in the regime $\zsq \in \Isma$ and small $\eta$ such bound is not available since 
$(\id- \Tf\Ff)\ffm = O(\eta)$ which can be deduced from \eqref{eq:Tffm_close_to_minus_ffm} below.
The simple bound  
\begin{equation} \label{eq:trivial_bound_inverse_1-TF_twonorm}
\normtwo{(\id-\Tf\Ff)^{-1}} \lesssim \eta^{-1}
\end{equation}
which is a consequence of \eqref{eq:bound_norm_F_small_eta} and $\normtwo{\Tf}=1$ is not strong enough.  
 In order to control $\normtwo{(\id - \Tf \Ff)^{-1} \Vf\rf}$ we will need to use a special property of the vector $\Vf\rf$,
namely that it is almost orthogonal to $\ffm$. This mechanism is formulated in the following \emph{Contraction-Inversion Lemma}
which is proved in Appendix \ref{proof:rotation_inversion}. 
It is closely related to the Rotation-Inversion lemmas -- Lemma 5.8 in \cite{AjankiCPAM} and Lemma 3.6 in \cite{AltGram} --  
which control the invertibility of $\id-UF$, where $U$ is a unitary operator and $F$ is symmetric. 

\begin{lem}[Contraction-Inversion Lemma] \label{lem:rotation_inversion}
Let $\eps, \eta, c_1, c_2, c_3 >0$ satisfying $\eta \leq \eps c_1/(2c_2^2)$ and $\Af, \Bf \in \C^{2n\times 2n}$ be two Hermitian matrices such that 
\begin{equation}
 \normtwo{\Af} \leq 1, \quad \normtwo{\Bf} \leq 1 -c_1 \eta. 
 \label{TF}
 \end{equation}
Suppose that there are $\ell^2$-normalized vectors $\bb_\pm \in \C^{2n}$ satisfying 
\begin{equation} \label{eq:properties_F}
\Bf\bb_+ = \normtwo{\Bf} \bb_+, \quad \Bf\bb_- = - \normtwo{\Bf} \bb_-, \quad \normtwo{\Bf\xf} \leq (1 - \eps) \normtwo{\xf}
\end{equation}
for all $\xf \in \C^{2n}$ such that $\xf \perp \linspan\{\bb_+, \bb_-\}$. 

Furthermore, assume that 
\begin{equation} \label{eq:norm_Tf_plus_Tf_minus}
 \scalar{\bb_+}{\Af \bb_+} \leq 1-\eps, \quad \normtwo{(\id+\Af)\bb_-} \leq c_2 \eta. 
\end{equation}
Then there is a constant $C >0$, depending only on $c_1, c_2, c_3$ and $\eps$, such that 
for each $\pf\in \C^{2n}$ satisfying 
\begin{equation} \label{eq:assumption_fminus_d}
 \abs{\scalar{\bb_-}{\pf}}\leq c_3 \eta \normtwo{\pf}, 
\end{equation}
 it holds true that 
\begin{equation} 
 \normtwo{(\id-\Af\Bf)^{-1} \pf}  \leq C \normtwo{\pf}. \label{eq:estimate_stability_lemma}
\end{equation}
\end{lem}

We will apply this lemma with the choices $\Af = \Tf$, $\Bf = \Ff$, $\bb_\pm = \boldsymbol f_\pm$ and $\pf = \Vf\rf$. 
The resulting bound on $\normtwo{(\id - \Tf \Ff)^{-1} \Vf\rf}$ will be lifted to a bound on 
$\norminf{(\id - \Tf \Ff)^{-1} \Vf\rf}$ by \eqref{eq:control_infnorm_with_twonorm_for_TF}. 
All estimates in the remainder of this proof will hold true uniformly for $\zsq\in \Isma$. 
However, we will not stress this fact for each estimate.
Moreover, the estimates will be uniform for $\eta \in (0,\eta_*]$. The threshold $\eta_*\leq 1$ 
will be chosen later such that it depends on $\para$ only and 
the assumptions of Lemma \ref{lem:rotation_inversion} are fulfilled.
We now start checking the assumptions of Lemma \ref{lem:rotation_inversion}. 

By Proposition \ref{pro:estimates_v_small_z}, there is $\Phi_1 \sim 1$ such that 
\begin{equation} \label{eq:vf_bounded_above_below}
\Phi_1^{-1} \leq \vf \leq \Phi_1
\end{equation}
for all $\eta \in(0,1]$.
We recall from  \eqref{eq:bound_norm_F_small_eta}  that there is a constant  $c_1\sim 1$ such that $\normtwo{\Ff} \leq 1 -c_1 \eta$ 
for all $\eta \in (0,1]$. 
Recalling the definition of $\af$ from \eqref{eq:def_af}, we conclude from \eqref{eq:estimate_ff_-_inf} the existence of $\Phi_2 \sim 1$ such that 
\begin{equation}  \label{eq:estimate_b_V_inverse_a}
\normtwo{\ffm - \af } \leq \norminf{\ffm-\af} \leq \Phi_2 \eta
\end{equation}
for all $\eta \in (0,1]$. Here, we used that $\normtwo{\yf} \leq \norminf{\yf}$ for all $\yf \in \C^{2n}$ due to 
the normalization of the $\ell^2$ norm. 

Since the first and the second $n$-component of the vector $\Vf\vf$ are the same we have $\Tf \af = - \af$ by \eqref{eq:def_af} and Lemma \ref{lem:pro_Tf} (iii). Hence, 
\begin{equation} \label{eq:Tffm_close_to_minus_ffm}
\normtwo{\ffm + \Tf \ffm} \leq \normtwo{\ffm-\af} + \normtwo{\Tf} \normtwo{\ffm-\af}\leq  2 \Phi_2 \eta 
\end{equation}
by $\normtwo{\Tf} = 1$ and \eqref{eq:estimate_b_V_inverse_a}. 

Due to \eqref{eq:gap_F_order_1}, there exists $\eps \sim 1$ such that  
\[ \normtwo{\Ff \xf} \leq (1- \eps)\normtwo{\xf}\] 
for all $\xf \in \C^{2n}$ such that $\xf \perp \ffp$ and $\xf \perp \ffm$ and for all $\eta \in (0,1]$. 
As $\Tf$ is Hermitian we can also assume by \eqref{eq:spec_Tf} that 
\[ \scalar{\ffp}{\Tf\ffp} \leq 1- \eps \]
for all $\eta \in (0,1]$ by possibly reducing $\eps$ but keeping $\eps \gtrsim 1$. 

So far we checked the conditions \eqref{TF}--\eqref{eq:norm_Tf_plus_Tf_minus}, it 
 remains to verify \eqref{eq:assumption_fminus_d} with the choice $\pf = \Vf\rf$. 
Assuming that $\scalar{\af}{\Vf\rf}=0$, we deduce from  \eqref{eq:estimate_b_V_inverse_a} that 
\begin{equation} \label{eq:bound_scalar_b_V_rhs}
 \abs{\scalar{\ffm}{\Vf\rf}}  \leq \abs{\scalar{\af}{\Vf\rf}} + \norm{\ffm - \af}_2 \norm{\Vf\rf}_2 \leq \Phi_2 \eta \norm{\Vf\rf}_2. 
\end{equation}
This is the estimate required in \eqref{eq:assumption_fminus_d}. Hence, it suffices to show that $\Vf \rf$ is perpendicular to $\af$, i.e.,

\begin{equation} \label{eq:scalar_product_zero}
\scalar{\emin (\Vf \vf)}{\Vf\rf}  = \scalara{\emin \left(\Vf^2 \vf\right)}{\Lf\hf} = \scalara{\Lf^*\left(\emin \frac{\wt \vf}{\uf}\right)}{\hf} = 0,
\end{equation}
where we used the symmetry of $\Vf$, that $\Vf$ is diagonal and \eqref{eq:stability_equation} in the first equality,
and the notation $\wt \vf = (v_2 , v_1)$.

We compute
\begin{equation} \label{eq:Lf_adjoint_Vf_af}
 \Lf^* \left( \emin \frac{\wt \vf}{\uf}\right) = \emin \frac{\wt \vf}{\uf} + \Sf_o \left( \vf^2 \emin \frac{\wt \vf}{\uf}\right) - \zsq \Sf_d^t \left(\uf^2 \emin \frac{\wt \vf}{\uf}\right) 
= \begin{pmatrix} \eta + Sv_2 - S \left(v_2 \left( \frac{v_1 v_2}{u} +  \zsq u\right)\right) \\ -\eta - S^tv_1 + S^t\left(v_1 \left(\frac{v_1 v_2}{u} + \zsq u \right) \right)  \end{pmatrix} = \eta \emin.
\end{equation}
Here, we used \eqref{eq:relation_wt_vf_div_uf} in the second step and the $n$-component relations of the second identity in \eqref{eq:defining_eq_v_wt_v_u}  in the 
last step. 
Since $\avg{\emin \gf} = \avg{\emin\vf} = 0$ by \eqref{eq:assumption_stability} and \eqref{eq:avg_v1_equal_avg_v2}, respectively, this proves \eqref{eq:scalar_product_zero} and therefore \eqref{eq:bound_scalar_b_V_rhs} as well.
Thus, we checked all conditions  of Lemma~\ref{lem:rotation_inversion}. 

By possibly reducing $\eta_*$ but keeping $\eta_* \gtrsim 1$, we can assume that $\eta_* \leq \eps c_1/(8\Phi_2^2)$. 
Now, we can apply Lemma~\ref{lem:rotation_inversion} with $\eps$, $c_1$, $c_2 = 2\Phi_2$, $c_3 = \Phi_2$ for any $\eta \in (0,\eta_*]$. 
Thus, applying  \eqref{eq:estimate_stability_lemma} in Lemma \ref{lem:rotation_inversion} 
to \eqref{eq:transformed_stability_equation},
we obtain $\normtwo{\Vf\hf} \lesssim \normtwo{\Vf \rf}$ and hence $\norminf{\Vf\hf} \lesssim \norminf{\Vf \rf}$
because of \eqref{eq:control_infnorm_with_twonorm_for_TF}. Therefore, for any $\lambda_*>0$, depending only 
on $\para$, we have
\[ \norminf{\hf}\char\big(\norminf{\hf}\leq \lambda_*\big) \lesssim \norminf{\Vf^{-1}} \norminf{\Vf\rf} 
\char\big(\norminf{\hf}\leq \lambda_*\big)\lesssim \norminf{\hf}^2 + \norminf{\df} \]
uniformly for $\eta \in (0,\eta_*]$ and $\zsq \in \Isma$.
Here, we used \eqref{eq:control_norm_V_V_inverse} and \eqref{eq:first_step} in the second step.
Choosing $\lambda_*>0$ small enough as before, we conclude \eqref{eq:stability_estimate1} for $\eta \in (0,\eta_*]$ and $\zsq \in \Isma$. 
Since $\eta_*>0$ depends only on $\para$, and $\eta_*$ was arbitrary in the proof 
of Step 2 we proved \eqref{eq:stability_estimate1} for all $\eta >0$ and $\zsq \in \Isma \cup \Ibig$.

In order to prove \eqref{eq:stability_estimate_average}, we remark that because of \eqref{eq:stability_estimate1}
 and \eqref{eq:represenation_rf_5} the estimate \eqref{eq:second_estimate_V_rhs} 
holds true for $\eta \in (0,\eta_*]$ and $\zsq \in \Isma$ as well. 
Due to the instability \eqref{eq:trivial_bound_inverse_1-TF_twonorm} of $(\id - \Tf\Ff)^{-1}$ and, correspondingly,
 of its adjoint, the definition 
of $\Rf$ in \eqref{eq:def_Rf} will not yield an operator satisfying $\norminf{\Rf} \lesssim 1$ in this regime.
Therefore, we again employ that the inverse of $\id - \Tf\Ff$ is bounded on the subspace orthogonal 
to $\ffm$ and the blow-up in the direction of $\ffm$ is compensated by the smallness of $\scalar{\ffm}{\Vf\rf}$
following from $\scalar{\af}{\Vf\rf}=0$ and $\norminf{\ffm-\af} =O(\eta)$ by \eqref{eq:estimate_ff_-_inf}. 

Let $\Qf$ be the orthogonal projection onto the 
subspace $\ffm^\perp$, i.e., $\Qf \xf \defeq \xf - \scalar{\ffm}{\xf} \ffm $ for all $\xf \in \C^{2n}$. 
Recalling the definition of $\af$ in \eqref{eq:def_af}, we now define the operator $\Rf=\Rf(\eta)$ on $\C^{2n}$ 
as follows: 
\begin{equation} \label{eq:def_Rf_2}
\Rf\xf \defeq  \Vf\left((\id-\Tf\Ff)^{-1} \Qf\right)^*\Vf^{-1} \xf  - \scalar{\Vf^{-1} (\id-\Tf\Ff)^{-1} \ffm}{\xf}\Vf(\ffm-\af)
\end{equation}
for every $\xf \in \C^{2n}$. Note that this $\Rf$ is different from the one given in \eqref{eq:def_Rf} that is used in the 
other parameter regimes.
Now, we estimate $\norminf{\Rf\xf}$. For the first term, we use the bound \eqref{eq:estimate_stability_inf_projection} whose assumptions we check first.
The first condition, $\normtwo{(\id-\Tf\Ff)^{-1} \Qf} \lesssim 1$, in \eqref{eq:inverse_id_minus_TF_Q_conditions}
follows from \eqref{eq:estimate_stability_lemma} as \eqref{eq:assumption_fminus_d} with $\pf= \Qf\xf$ 
is trivially satisfied and hence $\normtwo{(\id-\Tf\Ff)^{-1}\Qf \xf} \lesssim \normtwo{\Qf\xf} \lesssim \normtwo{\xf}$. The second condition in \eqref{eq:inverse_id_minus_TF_Q_conditions} is met by 
\eqref{eq:bound_norm_F_small_eta} and the third condition is exactly \eqref{eq:Tffm_close_to_minus_ffm}. 
Using $\norminf{\ffm} \lesssim 1$ from  \eqref{eq:ff_+_order_1}, \eqref{eq:estimate_stability_inf_projection} 
and \eqref{eq:control_norm_V_V_inverse}, we conclude that the first term in \eqref{eq:def_Rf_2} is 
$\lesssim \norminf{\xf}$.  In the second term, we use the trivial bound
\begin{equation} \label{eq:trivial_bound_inverse_1-TF}
\norminfa{(\id-\Tf\Ff)^{-1}} \lesssim \eta^{-1} 
\end{equation}
which is a consequence of the corresponding bound on $\normtwo{(1-\Tf\Ff)^{-1}}$ in 
\eqref{eq:trivial_bound_inverse_1-TF_twonorm} and \eqref{eq:control_infnorm_with_twonorm_for_TF}.
The potential blow-up in \eqref{eq:trivial_bound_inverse_1-TF} for small $\eta$
is compensated by the estimate $\norminf{\ffm - \af}= O(\eta)$ from \eqref{eq:estimate_ff_-_inf}.
Altogether this yields
$ \norm{\Rf(\eta)}_\infty \lesssim 1$ for all $\eta \in (0,\eta_*]$. 

From the definition of $\Rf$, we obtain 
\begin{align}\label{yh}
 \scalar{\yf}{\hf}  & =  \scalar{\yf}{\Vf^{-1} (\id-\Tf\Ff)^{-1} \Vf\rf} \\ \nonumber
& = \scalara{\Vf^{-1}\yf}{(\id-\Tf\Ff)^{-1} \Qf \Vf (\rf -\df)}+ \scalara{\yf}{\Vf^{-1}(\id-\Tf\Ff)^{-1} \ffm}
\scalara{\ffm-\af}{\Vf(\rf -\df)} + \scalar{\Rf\yf}{\df}. 
\end{align}
Notice that we  first  inserted $\id=\Qf + |\ffm\rangle \langle \ffm|$ before $\Vf\rf$, then we inserted the vector $\af$ in the
second term for free by using  $\scalar{\af}{\Vf\rf} = 0$ from  \eqref{eq:scalar_product_zero}. 
This brought in the factor $\ffm - \af\sim O(\eta)$ that compensates the $(\id- \Tf\Ff)^{-1}$ on the unstable subspace
parallel to $\ffm$. 
Finally, we subtracted the term $\df$ to $\rf$  freely and we defined the operator $\Rf$ exactly to compensate for it. The reason for 
this counter term  $\df$ is the formula \eqref{eq:second_estimate_V_rhs} showing that $\rf -\df$ is one order better in $\df$ 
than $\rf$. Thus, the first two terms in the right-hand side of \eqref{yh} are 
bounded by $\norminf{\df}^2\norminf{\yf}$. 
  The compensating term, $\scalar{\Rf\yf}{\df}$ remains first order in $\df$
but only in weak sense, tested against the vector $\Rf\yf$, and not in norm sense. This is the essential
improvement of  \eqref{eq:stability_estimate_average} over \eqref{eq:stability_estimate1}. 
Recalling now $\hf = \gf - \ii\vf$, the identity \eqref{yh} together with the bounds  we just explained 
concludes the proof of Proposition \ref{pro:stab_lemma}.
\end{proof}

\section{Proof of Proposition~\ref{pro:properties_sigma}}

As in the previous section, we assume without loss of generality that $\rho(S) =1$. See the remark about \eqref{eq:v_rescaled}.

For $\zsq_* >0$ and $\zsq^* > \zsq_* +1$, we define
\begin{equation} \label{eq:def_Dsma_Dbig}
 \Dsma \defeq \{ z \in \C \;\mid\; \abs{z}^2 \leq 1 - \zsq_* \}, \quad \Dbig \defeq \{ z \in \C \;\mid\; 1 + \zsq_* \leq \abs{z}^2 \leq \zsq^* \}. 
\end{equation}
Via $\zsq = \abs{z}^2$ these sets correspond to the regimes  $[0,1-\zsq_*]$ and $[1+\zsq_*,\zsq^*]$ in the previous section.

\begin{proof}[Proof of Proposition~\ref{pro:properties_sigma}]
Since the defining equations in \eqref{eq:v} are smooth functions of $\eta$, $\zsq$ and $(\vf_i)_{i=1,\ldots, 2n}$ and the operator $\Lf$ is invertible for $\eta>0$ 
the implicit function theorem implies that the function $\vf \colon \R_+\times \Rnon \to \R_+^{2n}$ is smooth. 
Therefore, the function $\R_+ \times \C \to \R_+^{2n}, ~~(\eta, z) \mapsto\vf^\zsq(\eta)|_{\zsq = \abs{z}^2}$ is also smooth. 

For $\alpha=(\alpha_1, \alpha_2) \in \N^2$, we define
\[ \pt^\alpha \vf \defeq \pt_\eta^{\alpha_1} \pt_\zsq^{\alpha_2} \vf . \]
For fixed $\zsq_*>0$ and $\zsq^*>\zsq_* +1$, we first prove that for all $\alpha \in \N^2$, we have 
\begin{equation} \label{eq:vf_bounded_derivates}
\norminf{\pt^\alpha \vf} \lesssim 1
\end{equation}
uniformly for all $\eta >0$ and $\zsq \in \Isma\cup \Ibig$. 

Differentiating \eqref{eq:v} with respect to $\eta$ and $\zsq$, respectively, yields 
\begin{equation} \label{eq:v_eta_and_r_derivative}
\Lf(\pt_\eta \vf) = - \vf^2 + \zsq\uf^2, \quad \Lf(\pt_\zsq \vf) = - \uf \vf .
\end{equation}
By further differentiating with respect to $\eta$ and $\zsq$, we iteratively obtain that for any multiindex $\alpha \in \N^2$
\begin{equation} \label{eq:pt_alpha_vf_rf_alpha}
 \Lf \pt^\alpha \vf = \rf_\alpha, 
\end{equation}
where $\rf_\alpha$ only depends on $\eta$, $\zsq$ and $\pt^\beta \vf$ for $\beta \in \N^2$, $\abs{\beta} = \beta_1 + \beta_2 < \abs{\alpha}$. 
In fact, for all $\alpha \in \N^2$, we have 

\begin{subequations} \label{eq:induction_step_derivatives}
\begin{align}
 \Lf (\pt^{\alpha + e_1} \vf ) & = \pt^\alpha \left(-\vf^2 + \zsq \uf^2\right) 
- \sum_{\nu\leq \alpha, \nu\neq (0,0)} \begin{pmatrix} \alpha \\ \nu \end{pmatrix} \left(\pt^{\nu}\Lf \right) \left( \pt^{\alpha-\nu + e_1} \vf \right),  \label{eq:induction_step_derivatives1}\\
 \Lf (\pt^{\alpha + e_2} \vf ) & = \pt^\alpha \left(-\vf \uf\right) 
 - \sum_{\nu\leq \alpha, \nu \neq (0,0)} \begin{pmatrix} \alpha \\ \nu \end{pmatrix} \left(\pt^{\nu}\Lf \right) \left( \pt^{\alpha-\nu + e_2} \vf \right). \label{eq:induction_step_derivatives2} 
\end{align}
\end{subequations}
As an example, we compute 
\begin{align} 
\Lf \pt^2_\zsq \vf & = - 2 \uf \pt_\zsq \vf + 2 \uf^2 \Sf_d \pt_\zsq \vf -2 \vf \pt_\zsq \vf \Sf_o \pt_\zsq \vf + \frac{2 \zsq\uf^2}{\vf} \pt_\zsq \vf \Sf_d \pt_\zsq \vf   -\frac{2\zsq\uf^3}{\vf} \left(\Sf_d\vf\right)^2 \nonumber \\
& = \frac{2}{\vf} \left( \pt_\zsq \vf\right)^2 +2\uf^2 \Sf_d \pt_\zsq \vf - \frac{2\zsq\uf^3}{\vf} \left(\Sf_d \pt_\zsq \vf\right)^2 ,\label{eq:pt_r^2_vf}
\end{align}
where we used the second relation in \eqref{eq:v_eta_and_r_derivative} in the second step. 

By induction on $\abs{\alpha} = \alpha_1 + \alpha_2$, we prove $\norminf{\rf_\alpha} \lesssim 1$ and $\norminf{\pt^\alpha \vf} \lesssim 1$ simultaneously. 
From \eqref{eq:induction_step_derivatives}, we conclude that $\rf_{\alpha+e_1}$ and $\rf_{\alpha+e_2}$ are bounded in $\ell^\infty$-norm if $\norminf{\pt^\nu \vf} \lesssim 1$ for all $\nu \leq \alpha$
as the first term on the right-hand side of \eqref{eq:induction_step_derivatives1} and \eqref{eq:induction_step_derivatives2}, respectively, and $\pt^\nu \Lf$ for all $\nu \leq \alpha$ are bounded. 
In order to conclude that $\pt^{\alpha + e_1} \vf$ and $\pt^{\alpha + e_2} \vf$ are bounded it suffices to prove that $\norminf{\pt^\alpha \vf} \lesssim \norminf{\rf_\alpha}$ by controlling 
$\Lf^{-1}$ in \eqref{eq:pt_alpha_vf_rf_alpha}. 

As in the proof of Proposition \ref{pro:stab_lemma} the norm of $\Lf^{-1}$ is bounded, $\norminf{\Lf^{-1}} \lesssim 1$, for $\zsq \in \Ibig$ or $\zsq \in \Isma$ and large $\eta$ as well as $\zsq \in \Isma$ and small $\eta$ 
separately. 
We thus focus on the most interesting regime where $\zsq \in \Isma$ and small $\eta$. As for the proof of Proposition \ref{pro:stab_lemma} we apply Lemma \ref{lem:rotation_inversion} in this regime. 
We only check the condition \eqref{eq:assumption_fminus_d} here since the others are established in the same way as in the proof of Proposition \ref{pro:stab_lemma}. 
Recall the definition of $\af$ in \eqref{eq:def_af}. 
Using $\avg{\emin \pt^\alpha\vf} = 0$ from \eqref{eq:avg_v1_equal_avg_v2} for all $\alpha \in \N^2$, we obtain 
\[ \scalara{\af}{\Vf\rf_\alpha} = \scalara{\Lf^* ( \emin \Vf^2 \vf)}{\pt^\alpha\vf} = \scalar{\eta\emin}{\pt^\alpha \vf}= 0 \]
for all $\alpha \in \N^2$. Here, we used $\Lf^*( \emin \Vf^2 \vf) = \eta \emin$ which is shown in \eqref{eq:Lf_adjoint_Vf_af} in the proof of Proposition \ref{pro:stab_lemma}.
This concludes the proof of \eqref{eq:vf_bounded_derivates}. 

Next, we show the integrability of $\Delta_z \avg{v_1^\zsq|_{\zsq = \abs{z}^2}}$ as a function of $\eta$ for $z \in \Dsma$ for fixed $\zsq_*>0$. 
Note that $\avg{v_1^\zsq} = \avg{\vf^\zsq}$ by \eqref{eq:avg_v1_equal_avg_v2}. 
Using 
\[ \Delta_z \left( \vf^\zsq|_{\zsq=\abs{z}^2} \right) = 4 \left( \zsq \pt_\zsq^2 \vf^\zsq + \pt_\zsq \vf^\zsq\right)|_{\zsq=\abs{z}^2}\] together with \eqref{eq:v_eta_and_r_derivative} and \eqref{eq:pt_r^2_vf}, we obtain 
\begin{equation} \label{eq:Lf_Laplace_vf}
 \Lf \Delta_z \left( \vf^\zsq|_{\zsq=\abs{z}^2} \right) = 
4\left( \frac{2\zsq}{\vf} \left(\pt_\zsq \vf\right)^2 + 2 \zsq\uf^2 \Sf_d \pt_\zsq \vf - \frac{2\zsq^2 \uf^3}{\vf}\left(\Sf_d \pt_\zsq \vf\right)^2 - \uf\vf \right).
\end{equation}
From \eqref{eq:bound_vf_large_eta}, \eqref{eq:bound_vf_small_eta} and \eqref{eq:estimate_uf}, we conclude that $\uf \vf \sim (1 +\eta^{3})^{-1}$ and hence $ \abs{\pt_\zsq 
 \vf} \lesssim (1 +\eta^{3})^{-1}$ uniformly for $z \in \Dsma$ 
since $\norminf{\pt^\alpha \vf} \lesssim \norminf{\rf_\alpha}$. Therefore, the right-hand side of \eqref{eq:Lf_Laplace_vf} is of order $(1+\eta^3)^{-1}$ for $z \in \Dsma$ and hence
using the control on $\Lf^{-1}$ as before, we conclude that $\abs{\Delta_z \left( \vf^\zsq|_{\zsq=\abs{z}^2} \right)} \lesssim (1 + \eta^3)^{-1}$ uniformly for $\eta >0$. 
Thus, $\Delta_z \avg{v_1^\zsq|_{\zsq = \abs{z}^2}} = \Delta_z \avg{\vf^\zsq|_{\zsq = \abs{z}^2}}$ as a function of $\eta$ is integrable on $\R_+$  and the integral is a continuous function of $z \in \Dsma$.
As $\zsq_*>0$ was arbitrary, this concludes the proof of part (i) of Proposition~\ref{pro:properties_sigma} and shows that $\sigma$ is a rotationally invariant function on $\C$ which is continuous on $D(0,1)$. 

Now, we establish that for $\zsq <1$, the derivative of the average of $\uf$ with respect to $\zsq$ gives an alternative representation of the density of states as follows 
\begin{equation} \label{eq:sigma_in_terms_of_vf_0}
 \sigma(z) = \frac{1}{\pi} \pt_\zsq \left(\zsq \avg{\uf_0}\right)\big\rvert_{\zsq=\abs{z}^2} = - \frac{2}{\pi} \scalar{\Sf_o \vf_0}{\pt_\zsq \vf_0}\big\rvert_{\zsq=\abs{z}^2},
\end{equation}
where $\uf_0 \defeq \lim_{\eta \downarrow 0} \uf(\eta)$ and $\vf_0 \defeq \lim_{\eta \downarrow 0} \vf(\eta)$.  
The first relation in \eqref{eq:sigma_in_terms_of_vf_0} will be proved below and the second one follows immediately using 
 $\zsq \uf_0 = 1 - \vf_0 \Sf_o \vf_0 $ by  \eqref{eq:v_combined} and \eqref{eq:def_u} for $\eta\downarrow 0$, 
as well as $\Sf_o^t = \Sf_o$.

We first give a heuristic derivation of the first equality in \eqref{eq:sigma_in_terms_of_vf_0}. 
Writing the resolvent $\Gf^z$ of $\Hf^z$ as  
\[ \Gf^z = \begin{pmatrix} G_{11} & G_{12} \\ G_{21} & G_{22} \end{pmatrix}, \]
we obtain 
\[ \tr G_{12} = \tr\left[ \left((X-z)(X^*-\bar z) + \eta^2\right)^{-1} (X-z)\right] = -\pt_{\bar z} \tr\log\left( (X-z)(X^*-\bar z) + \eta^2\right) = -\frac{2}{n}\pt_{\bar z } \log\abs{\det(\Hf^z-\ii\eta)} \]
for the normalized trace of $G_{12}$ (see \eqref{eq:def_trace}). 
Since $\Delta_z= 4 \pt_z \pt_{\bar z}$, taking the $\pt_z$-derivative of the previous identity, we obtain
\begin{equation} \label{eq:pro_2_3_aux_1}
\frac{1}{2n} \Delta_z \log\abs{\det(\Hf^z-\ii\eta)} = - \pt_z \tr G_{12}.  
\end{equation}

Using \eqref{eq:def_sigma}, \eqref{eq:log_det_Stieltjes} and $\Im m^z \approx \avg{v_1^\zsq|_{\zsq=\abs{z}^2}}$, 
the left-hand side of \eqref{eq:pro_2_3_aux_1} is approximately $\pi\sigma(z)$ after taking the $\eta\downarrow 0$ limit. 
On the other hand, $\Gf^z$ converges to $\Mf^z$ for $n\to \infty$. 
Thus, by \eqref{eq:def_Mf} the right-hand side of \eqref{eq:pro_2_3_aux_1} can be approximated by $\pt_z\left( z \avg{u^\zsq|_{\zsq=\abs{z}^2}(\eta)}\right)$. 
Therefore, taking $\eta \downarrow 0$, we conclude 
\[  \pi \sigma(z) \approx  \pt_z z \avg{u_0^\zsq|_{\zsq=\abs{z}^2} } =  \left(\pt_\zsq \zsq \avg{u_0^\zsq}\right)|_{\zsq = \abs{z}^2}. \]
In fact, this approximation holds not only in the $n\to\infty$ limit but it is an identity for any fixed $n$. This completes the heuristic argument for \eqref{eq:sigma_in_terms_of_vf_0}. 

We now turn to the rigorous proof of the first relation in \eqref{eq:sigma_in_terms_of_vf_0}.  In fact, for $\zsq <1$, we prove the following integrated version  
\begin{equation} \label{eq:relation_integral_sigma}
 \int_{\abs{z'}^2 \leq \zsq} \sigma(z') \di^2 z' = \zsq \avg{\uf_0^\zsq}.
\end{equation}
Since $\sigma$ is a continuous function on $D(0,1)$ differentiating \eqref{eq:relation_integral_sigma} with respect 
to $\zsq$ immediately yields \eqref{eq:sigma_in_terms_of_vf_0}. 

In order to justify the existence of the limits of $\vf$ and $\uf$ for $\eta\downarrow 0$ and the computations in the proof of \eqref{eq:relation_integral_sigma}, we remark that 
by \eqref{eq:vf_bounded_derivates}, $(\eta, z) \mapsto \vf^\zsq(\eta)|_{\zsq=\abs{z}^2}$ can be uniquely extended to a positive $C^\infty$ function on $[0,\infty) \times D(0,1)$. 
In the following, $\vf$ and $\vf_0^\zsq \defeq \vf^{\zsq}|_{\eta = 0}$ denote this function and its restriction to $\{0\} \times [0,1)$, respectively.  In particular, the restriction
$\vf_0^\zsq|_{\zsq = \abs{z}^2}$ is a smooth function on $D(0,1)$ which satisfies 
\begin{equation} \label{eq:v_eta=0}
\frac{1}{\vf_0^\zsq} = \Sf_o \vf_0^\zsq + \frac{\zsq}{\Sf_d \vf_0^\zsq}
\end{equation}
with $\zsq=\abs{z}^2$.
Moreover, derivatives of $\vf$ in $\eta$ and $\zsq$ and limits in $\eta$ and $\zsq$ for $\zsq < 1$ can be freely interchanged.  

For the proof of \eqref{eq:relation_integral_sigma}, we use integration by parts to obtain 
\begin{equation} \label{eq:relation_integral_sigma1}
 \int_{\abs{z'}^2 \leq \zsq} \sigma(z') \di^2 z' = -  2\zsq \int_0^\infty \pt_\zsq \avg{\vf} \di \eta = - \zsq \int_0^\infty \pt_\zsq\left( \avg{\vf} + \avg{\wt\vf}\right) \di \eta. 
\end{equation}
We recall $\wt\vf = (v_2, v_1)$ and get
\[ \vf = \frac{\eta +\Sf_d \vf}{(\eta + \Sf_d\vf)(\eta + \Sf_o\vf) + \zsq}, \quad \wt \vf = \frac{\eta +\Sf_o \vf}{(\eta + \Sf_d\vf)(\eta + \Sf_o\vf) + \zsq} \]
from \eqref{eq:v_combined}. This implies the identity
\[ \pt_\eta \log\left( (\eta + \Sf_d \vf)(\eta + \Sf_o \vf) + \zsq \right) = \vf + \wt \vf + \wt \vf \Sf_d \pt_\eta\vf + \vf \Sf_o \pt_\eta \vf.  \]
Using  
\[ \avg{\wt \vf \Sf_d \pt_\eta \vf}  + \avg{\vf \Sf_o\pt_\eta \vf} = \avg{\vf \Sf_o \pt_\eta \vf} + 
\avg{\vf \Sf_o\pt_\eta \vf} = \pt_\eta \avg{ \vf \Sf_o \vf}  \]
and recalling $\vf_0 \defeq \lim_{\eta\downarrow 0} \vf(\eta)$, we find for \eqref{eq:relation_integral_sigma1} the expression
\begin{equation} \label{eq:relation_integral_sigma2}
\int_0^\infty \pt_\zsq \left( \avg{\vf} + \avg{\wt \vf} \right) \di \eta = -  \avg{\pt_\zsq \log\left( (\Sf_d \vf_0)( \Sf_o \vf_0) + \zsq\right)} + \pt_\zsq \avg{\vf_0 \Sf_o \vf_0}. 
\end{equation}
Hence, due to 
\[ \avg{\pt_\zsq \log\left(  (\Sf_d \vf_0)( \Sf_o \vf_0) + \zsq\right)} = \avg{\uf} + \avg{\wt\vf_0 \Sf_d \pt_\zsq \vf_0} + \avg{\vf\Sf_o \pt_\zsq \vf_0} = \avg{\uf} + \pt_\zsq \avg{\vf_0 \Sf_o \vf_0}. \]
we obtain \eqref{eq:relation_integral_sigma} from \eqref{eq:relation_integral_sigma2}.  The formula \eqref{eq:relation_integral_sigma} was also obtained in \cite{CookNonHermitianRM} with a different method.

We prove (iii) before (ii). 
As $\vf_0$ is infinitely often differentiable in $\zsq$ and $\zsq=\abs{z}^2$, we conclude from \eqref{eq:sigma_in_terms_of_vf_0} that $\sigma$ is infinitely often differentiable in $z$. 
The following lemma shows \eqref{eq:uniform_lower_bound_sigma} which finishes the proof of part (iii).

\begin{lem}[Positivity and boundedness of $\sigma$]  \label{lem:sigma_strictly_positive}
Uniformly for $z \in D(0,1)$, we have
\begin{equation} \label{eq:sigma_sim_1}
 \sigma(z) \sim 1, 
\end{equation}
where $\sim$ only depends on $s_*$ and $s^*$.
\end{lem}

\begin{proof}[Proof of Lemma \ref{lem:sigma_strictly_positive}]
We will compute the derivative in \eqref{eq:sigma_in_terms_of_vf_0} and prove the estimate \eqref{eq:sigma_sim_1} first for $z \in \Dsma$ and arbitrary $\zsq_*>0$ depending only on 
$s_*$ and $s^*$. Then we show that there is $\zsq_*>0$ depending only on $s_*$ and $s^*$ such that \eqref{eq:sigma_sim_1} holds true for $z \in D(0,1)\setminus \Dsma$. 

In this proof, we write $\diM(y) \defeq \diag(y)$ for $y \in \C^l$ for brevity. 
Furthermore, we introduce the $2n \times 2n$ matrix
\[ \Ef \defeq \begin{pmatrix} \id  & \id \\ \id & \id \end{pmatrix}. \]
In the following, $\vf$ and all related quantites will be evaluated at $\zsq = \abs{z}^2$. 
We start the proof from \eqref{eq:sigma_in_terms_of_vf_0}, recall $\Lf = \Vf^{-1} ( \id - \Tf \Ff) \Vf$ and use the second relation in \eqref{eq:v_eta_and_r_derivative} as well as \eqref{eq:relation_wt_vf_div_uf} to obtain 
\begin{align}
\sigma(z) = & -\frac{2}{\pi}\scalar{\Sf_o \vf_0}{\pt_\zsq \vf_0}  =  \lim_{\eta\downarrow 0} \frac{2}{\pi}\scalara{\Vf^{-1}\frac{\wt \vf}{\uf}}
{ ( \id - \Tf\Ff)^{-1} \Vf(\vf \uf) } = \lim_{\eta\downarrow 0}  
\frac{2}{\pi} \scalara{\sqrt{\vf \wt \vf}}{ \frac{1}{\sqrt{\uf}} ( \id - \Tf\Ff)^{-1} \sqrt{\uf} \sqrt{\vf \wt \vf}} 
\nonumber \\
 & =  \lim_{\eta\downarrow 0} \frac{2}{\pi} \scalara{\sqrt{\vf \wt \vf}}{ \left ( \id - \diM(\uf^{-1/2})  
\Tf\Ff\diM(\uf^{1/2}) \right)^{-1}\sqrt{\vf \wt \vf}}.  \label{eq:pos_sigma_aux2}
\end{align} 
Note that the inverses of $\id - \Tf \Ff$ and $\id - \zsq \diM(\uf^{-1/2})\Tf\Ff \diM(\uf^{1/2})$ exist by Lemma \ref{lem:pro_Tf} and Lemma \ref{lem:prop_Ff} 
as $\eta >0$ and $\zsq <1$. 

Due to \eqref{eq:def_Tf} and \eqref{eq:defining_eq_v_wt_v_u}, we have $\Tf = -\id + \zsq \uf \Ef$ which implies
\begin{align}
 \id - \diM(\uf^{-1/2})\Tf\Ff\diM(\uf^{1/2}) & =  \id + \diM(\uf^{-1/2})\Ff\diM(\uf^{1/2}) - \zsq\diM(\uf^{1/2}) 
\Ef \Ff\diM(\uf^{1/2})  \nonumber\\
& = \left( \id -  \zsq \diM(\uf^{1/2})\Ef \Ff (\id + \Ff)^{-1}\diM(\uf^{1/2}) \right) \left(\id +
\diM(\uf^{-1/2})\Ff\diM(\uf^{1/2}) \right).  \label{eq:pos_sigma_aux8}
\end{align}
From \eqref{eq:norm_F} and \eqref{eq:ffp}, we deduce $ \sqrt{\uf} \Ff \sqrt{\vf \wt \vf/\uf } = \sqrt{\vf\wt\vf} + O(\eta)$. Hence, due to \eqref{eq:pos_sigma_aux8}, 
 \eqref{eq:pos_sigma_aux2}  yields
\begin{equation} \label{eq:pos_sigma_aux9}
\sigma(z) = \lim_{\eta\downarrow 0}  \frac{1}{\pi}\scalara{\sqrt{\vf \wt\vf} }{\left( \id - \zsq\diM(\uf^{1/2})
\Ef \Ff (\id + \Ff)^{-1} \diM(\uf^{1/2}) \right)^{-1} \sqrt{\vf\wt\vf}}.
\end{equation}
Defining the matrix $F \in \C^{n\times n}$ through $F y = \sqrt{v_1u/v_2} S \sqrt{v_2u/v_1}\, y$ for $y \in \C^n$, we obtain 
\begin{equation} \label{eq:pos_sigma_aux3} 
 \Ff = \begin{pmatrix} 0 & F \\ F^t & 0 \end{pmatrix}, \qquad \left( \id + \Ff\right)^{-1} = \begin{pmatrix} (\id - FF^t)^{-1}
 & - ( \id -FF^t)^{-1} F \\ - F^t ( \id -F F^t)^{-1} & ( \id- F^tF)^{-1} \end{pmatrix}.
\end{equation}
Furthermore, we introduce the $n\times n$ matrix $A$ by 
\[ A \defeq 2\cdot \id + (F^t -\id) (\id - FF^t)^{-1} +   (F-\id) ( \id- F^t F)^{-1}. \]
From the computation
\begin{align*}
 \Ef \Ff ( \id + \Ff)^{-1} 
& = \begin{pmatrix} \id + (F^t -\id) (\id - FF^t)^{-1} & \id + (F-\id) ( \id- F^t F)^{-1} \\
 \id + (F^t -\id) (\id - FF^t)^{-1} & \id + (F-\id) ( \id- F^t F)^{-1} \end{pmatrix},
\end{align*} 
we conclude that 
\begin{equation} \label{eq:pos_sigma_aux10}
 \left ( \id- \zsq \diM(\uf^{1/2})\Ef\Ff(\id + \Ff)^{-1} \diM(\uf^{1/2}) \right)^{-1} \begin{pmatrix} x \\ x \end{pmatrix} = \begin{pmatrix} (\id - \zsq \diM(u^{1/2}) A \diM(u^{1/2}))^{-1} x \\ 
(\id- \zsq \diM(u^{1/2}) A \diM(u^{1/2}))^{-1} x \end{pmatrix} 
\end{equation}
for all $x \in \C^n$. 
Before applying this relation to \eqref{eq:pos_sigma_aux9}, 
we show that $\id- \zsq \diM(u^{1/2}) A \diM(u^{1/2})$ is invertible for $\zsq <1$.
The relations in \eqref{eq:pos_sigma_aux3} yield 
\begin{equation} \label{eq:pos_sigma_aux5}
 \scalar{x}{A x} = 2\normtwo{x}^2 - 2\scalara{\begin{pmatrix} x \\ x \end{pmatrix}} {( \id + \Ff)^{-1} \begin{pmatrix} x \\ x \end{pmatrix}}
\end{equation}
for all $x \in \C^n$ and $\eta>0$. In particular, since $\normtwo{\Ff} \leq 1$ by \eqref{eq:norm_F} we conclude $A \leq \id$. 
Hence, $\zsq u = 1- v_1 v_2/u <1$ for $\zsq <1$ by \eqref{eq:defining_eq_v_wt_v_u} implies that $\id- \zsq \diM(u^{1/2}) A \diM(u^{1/2})$ 
is invertible for $\zsq <1$. Thus, we apply \eqref{eq:pos_sigma_aux10} to \eqref{eq:pos_sigma_aux9} and obtain for $z \in D(0,1)$
\begin{equation} \label{eq:pos_sigma_aux12}
 \sigma(z) = \frac{2}{\pi}\lim_{\eta\downarrow 0}\scalara{\sqrt{v_1v_2}}{\left( \id- \zsq \diM(u^{1/2}) A \diM(u^{1/2})\right)^{-1} \sqrt{v_1 v_2}}.
\end{equation}

Let $\zsq_* >0$ depend only on $s_*$ and $s^*$. 
From \eqref{eq:bound_vf_small_eta} and \eqref{eq:vf_bounded_derivates}, we conclude that $\abs{\sigma} \lesssim 1$ uniformly for $z \in \Dsma$ because
of \eqref{eq:sigma_in_terms_of_vf_0}. This proves the upper bound in \eqref{eq:sigma_sim_1} for $z \in \Dsma$. 

For the proof of the lower bound, we infer some further properties of $A$ and $\id- \zsq \diM(u^{1/2}) A \diM(u^{1/2})$, respectively, from information about $\Ff$ via \eqref{eq:pos_sigma_aux5}.
In the following, we use versions of Proposition \ref{pro:estimates_v_small_z}, \eqref{eq:estimate_uf} and Lemma \ref{lem:prop_Ff} extended to
 the limiting case $\eta =0+$.  Recalling $\vf_0 = \lim_{\eta\downarrow 0} \vf$, 
these results are a simple consequence of the uniform convergence $\pt^\alpha \vf \to \pt^\alpha \vf_0$ for $\eta \downarrow 0$ and all $\alpha \in \N^2$ by \eqref{eq:vf_bounded_derivates}. 

Since $\ffm = ( \sqrt{v_1v_2/u}, - \sqrt{v_1v_2/u})+O(\eta)$ by \eqref{eq:af_approximating_eigenvector} there are $\eta_*, \eps \sim 1$ by Lemma \ref{lem:prop_Ff} 
such that 
$\spec(\Ff|_W) \subset [-1 +\eps, 1]$ on the subspace $W \defeq \{ (x,x) | x \in \C^n\}\subset \C^{2n}$
as $\ffm \perp W$ uniformly for all $\eta\in[0,\eta_*]$.
Therefore, for $\normtwo{x} = 1$, the right-hand side of \eqref{eq:pos_sigma_aux5} is contained 
in $[2(\eps -1 )/\eps, 1]$.  Since $\left( F^t ( \id- FF^t)^{-1}\right)^t = F ( \id - F^t F)^{-1}$ the matrix $A$ is real symmetric and hence the spectrum of $A$ is contained in $[2(\eps -1 )/\eps, 1]$ for 
all $\eta \in [0,\eta_*]$ as well. 

The real symmetric matrix $A$ has a positive and a negative part, i.e., there are positive 
matrices $A_+$ and $A_-$ such that $A = A_+ - A_-$.  Hence, we have 
\begin{equation} \label{eq:pos_sigma_aux11}
 \id - \zsq \diM(u^{1/2}) A \diM(u^{1/2}) = \id- \zsq \diM(u^{1/2}) A_+ \diM(u^{1/2}) + \zsq \diM(u^{1/2}) A_- \diM(u^{1/2}). 
\end{equation}
The above statements about \eqref{eq:pos_sigma_aux5} yield $\spec A_+ \subset [0,1]$ and $\spec A_-  \subset [0,2(1-\eps)/\eps]$. 
As $0 \leq u \zsq$ we conclude from \eqref{eq:pos_sigma_aux11} that the spectrum of $\id - \zsq \diM(u^{1/2}) A \diM(u^{1/2})$ is contained in $(0, 2/\eps]$ for all $\eta\in [0,\eta_*]$.
Therefore, using \eqref{eq:pos_sigma_aux12}, we obtain 
\[ \sigma(z) = \frac{2}{\pi}\lim_{\eta\downarrow 0}\scalara{\sqrt{v_1v_2}}{\left( \id- \zsq \diM(u^{1/2}) A \diM(u^{1/2})\right)^{-1} \sqrt{v_1 v_2}} \geq \frac{\eps}{\pi} \avg{\vf_0\wt\vf_0} \gtrsim 1  \]
uniformly for all $z \in \Dsma$. Here, we used \eqref{eq:bound_vf_small_eta} in the last step. 
This shows \eqref{eq:sigma_sim_1} for $z \in \Dsma$ for any $\zsq_*>0$ depending only on $s_*$ and $s^*$.

We now show that there is $\zsq_*>0$ depending only on $s_*$ and $s^*$ such that \eqref{eq:sigma_sim_1} holds true for $z \in D(0,1)\setminus \Dsma$.
This is proved by tracking the blowup of $(\id-\zsq \diM(u^{1/2})A\diM(u^{1/2}))^{-1}$ in $1-\zsq$ for $\zsq \uparrow 1$ 
in \eqref{eq:pos_sigma_aux12} and establishing a compensation through $v_1 \sim v_2 \sim (1-\zsq)^{1/2}$ due to \eqref{eq:bound_vf_small_eta}. This yields the upper and lower bound in \eqref{eq:sigma_sim_1}.
Since $\id-\zsq \diM(u^{1/2})A\diM(u^{1/2})$ in \eqref{eq:pos_sigma_aux12} is also invertible for $\eta =0$ we may directly set $\eta =0$ in the following argument.

We multiply the first component of the first relation in \eqref{eq:defining_eq_v_wt_v_u} by $\zsq$ and solve for $\zsq u$ to obtain 
\[ \zsq u = \frac{1}{2}\left(1+ \sqrt{1-4\zsq v_1 v_2} \right) = 1 - \zsq v_1 v_2 + O\left((1-\zsq)^2\right).  \]
Therefore, using $v_1 \sim v_2 \sim (1-\zsq)^{1/2}$, we have 
\[ \zsq \diM(u^{1/2})A\diM(u^{1/2}) = A - \frac{\zsq}{2} \left( \diM(v_1v_2)A + A \diM(v_1v_2)\right) + O\left((1-\zsq)^2\right). \]
Moreover, from \eqref{eq:pos_sigma_aux5} we conclude that $Aa = a$ for $a \defeq \sqrt{v_1 v_2/u}/\normtwo{\sqrt{v_1 v_2/u}}$. Here, we also used \eqref{eq:ffp} and \eqref{eq:norm_F} with $\eta=0$.

Thus, the smallest eigenvalue of the positive operator $\id - \zsq \diM(u^{1/2})A\diM(u^{1/2})$ satisfies 
\[ \lambda_{\text{min}}\left(\id-\zsq \diM(u^{1/2})A\diM(u^{1/2})\right) = \lambda_{\text{min}}\left(\id-A\right) + \zsq\avg{a^2 v_1 v_2}+ O\left((1-\zsq)^2\right) = \zsq\avg{a^2 v_1 v_2}+ O\left((1-\zsq)^2\right).\] 
Here, we used multiple times that $Aa=a$. 
Therefore, as $A$ is symmetric we conclude from \eqref{eq:pos_sigma_aux12} that 
\[ \sigma(z) =\frac{2}{\pi} \scalara{\sqrt{v_1v_2}}{\left( \id- \zsq \diM(u^{1/2}) A \diM(u^{1/2})\right)^{-1} \sqrt{v_1 v_2}} \geq \frac{\scalar{a}{\sqrt{v_1v_2}}^2}{\zsq \avg{a^2v_1v_2}} + O\left(1-\zsq\right).\]
Since $a\sim 1$ and $v_1 \sim v_2 \sim (1-\zsq)^{1/2}$ there is $\zsq_* \sim 1$ such that the lower bound in \eqref{eq:sigma_sim_1} holds true for $z \in D(0,1)\setminus \Dsma$. 
Starting from \eqref{eq:pos_sigma_aux12}, we similarly obtain
\[ \sigma(z) \leq \frac{\avg{v_1v_2}}{\zsq \avg{a^2v_1v_2}} + O\left(1-\zsq\right). \]
Using the positivity of $a$, $v_1 \sim v_2 \sim (1-\zsq)^{1/2}$ and possibly shrinking $\zsq_*\sim 1$ the upper bound in \eqref{eq:sigma_sim_1} for $z \in D(0,1)\setminus \Dsma$ follows. 
This concludes the proof of Lemma~\ref{lem:sigma_strictly_positive}.
\end{proof}

As $\sigma(z) =0$ for $\abs{z}\geq 1$ we conclude from \eqref{eq:uniform_lower_bound_sigma} that $\sigma$ is nonnegative on $\C$. 
We use \eqref{eq:relation_integral_sigma} to compute the total mass of the measure on $\C$ defined by $\sigma$. 
Clearly, $\uf_0 = \vf_0/\Sf_d \vf_0$ and using \eqref{eq:v_eta=0} and \eqref{eq:relation_integral_sigma}, we obtain 
\[ \lim_{\zsq\uparrow 1} \int_{\abs{z'}^2 \leq \zsq} \sigma(z') \di^2 z' = 1 - \lim_{\zsq\uparrow 1} \avg{\vf_0 \Sf_{o} \vf_0} = 1. \]
Here, we used that $\lim_{\zsq\uparrow 1} \vf_0 = 0$ by \eqref{eq:bound_vf_small_eta}. 
Hence, as $\sigma(z) =0$ for $\abs{z}\geq 1$ it defines a probability density on $\C$ which concludes the proof of Proposition~\ref{pro:properties_sigma}. 
\end{proof}

\begin{rem}[Jump height]
In fact, it is possible to compute the jump height of the density of states $\sigma$ at the edge $\zsq = \abs{z}^2 = 1$. 
Let $s_1$ and $s_2$ be two eigenvectors of $S^t$ and $S$, respectively, associated to the eigenvalue $1$, i.e., 
$S^t s_1 =s_1$ and $S s_2 =s_2$. Note that $s_1$ and $s_2$ are unique up to multiplication by a scalar.

With this notation, expanding $\vf^\zsq$ for $\zsq \leq 1$ around $\zsq =1$ yields 
\[ v_1 = \sqrt{1-\zsq}\left( \frac{\avg{s_1 s_2}\avg{s_2}}{\avg{s_1^2s_2^2}\avg{s_1}} \right)^{1/2} s_1 + O\left( (1-\zsq)^{3/2}\right), \quad 
v_2 = \sqrt{1-\zsq}\left( \frac{\avg{s_1 s_2}\avg{s_1}}{\avg{s_1^2s_2^2}\avg{s_2}} \right)^{1/2} s_2 + O\left( (1-\zsq)^{3/2}\right). \]
Therefore, solving \eqref{eq:defining_eq_v_wt_v_u} for $\zsq u$ and expanding in $1-\zsq$, we obtain that $\sigma$ has a jump of height 
\[ \lim_{\abs{z}^2 \uparrow 1} \sigma (z) = \frac{1}{\pi} \lim_{\zsq \uparrow 1} \pt_\zsq\left( \zsq \avg{\uf_0} \right) = \frac{1}{\pi} \frac{\avg{s_1s_2}^2}{\avg{s_1^2 s_2^2}}. \]
\end{rem}

\section{Local law} \label{sec:local_law}

We begin this section with a notion for high probability estimates.

\begin{defi}[Stochastic domination] \label{def:stochastic_domination}
Let $C \colon \R_+^2 \to \R_+$ be a given function which depends only on $a$, $\varphi$, $\zsq_*$, $\zsq^*$ and the model parameters. If $\Phi = (\Phi^{(n)})_{n}$ and $\Psi = (\Psi^{(n)})_{n}$ 
are two sequences of nonnegative random variables, then we will say that $\Phi$ is \textbf{stochastically dominated} by $\Psi$, $\Phi \prec \Psi$, if for all $\eps >0$ and $D>0$ we have 
\[ \P \left( \Phi^{(n)} \geq n^\eps \Psi^{(n)} \right) \leq \frac{C(\eps,D)}{n^D} \]
for all $n \in \N$. 
\end{defi}

As a trivial consequence of $\E x_{ij} = 0$, \eqref{eq:assumption_A} and \eqref{eq:assumption_B} we remark that 
\begin{equation} \label{eq:trivial_control_x_ij_prec}
\abs{x_{ij}} \prec n^{-1/2}. 
\end{equation}

\subsection{Local law for $\Hf^z$} \label{sec:local_law_H_z}

Let $(v_1^\zsq,v_2^\zsq)$ be the positive solution of \eqref{eq:v} and $u^\zsq$ defined as in \eqref{eq:def_u}.
 In the whole section, we will always evaluate $v_1^\zsq$, $v_2^\zsq$ and $u^\zsq$ at $\zsq = \abs{z}^2$ and mostly suppress the dependence on $\zsq$ and $\abs{z}^2$, respectively, in our notation. 
Recall that $\Mf^z$ is defined in \eqref{eq:def_Mf}. Note that although $v_1$, $v_2$ and $u$ are rotationally invariant in $z \in \C$, the dependence of $\Mf^z$ on $z$ is not 
rotationally symmetric.

For the following theorem, we remark that the sets $\Dsma$ and $\Dbig$ were introduced in \eqref{eq:def_Dsma_Dbig}.

\begin{thm}[Local law for $\Hf^z$] \label{thm:local_law_H_z}
 Let $X$ satisfy (A) and (B) and let $\Gf=\Gf^z$ be the resolvent of $\Hf^z$ as defined in \eqref{eq:def_H_z}. 
For fixed $\eps \in (0,1/2)$, the entrywise local law
\begin{equation} \label{entrywise matrix local law}
 \normb{\Gf^z(\eta) - \Mf^z(\eta)}_{\max} \prec 
  \begin{cases}
 \frac{1}{\sqrt{n\eta}} & \text{ for } z \in \mathbb{D}_<\,,\; \eta \in [n^{-1+\eps},1] \,,
 \\
 \frac{1}{\sqrt{n}} +\frac{1}{n \eta}& \text{ for } z \in \mathbb{D}_>\,,\; \eta \in [n^{-1+\eps},1] \,,
 \\
 \frac{1}{\sqrt{n}\2 \eta^2}& \text{ for } z \in \mathbb{D}_< \cup \mathbb{D}_> \,,\; \eta \in [1,\infty)\,,
 \end{cases} 
\end{equation}
 holds true. 
In particular,
\begin{equation} \label{eq:local_law_H_z_version}
 \normb{\gf(\eta) - \ii \vf(\eta)}_\infty \prec 
 \begin{cases}
 \frac{1}{\sqrt{n\eta}} & \text{ for } z \in \mathbb{D}_<\,,\; \eta \in [n^{-1+\eps},1] \,,
 \\
 \frac{1}{\sqrt{n}} +\frac{1}{n \eta}& \text{ for } z \in \mathbb{D}_>\,,\; \eta \in [n^{-1+\eps},1] \,,
 \\
 \frac{1}{\sqrt{n}\2 \eta^2}& \text{ for } z \in \mathbb{D}_< \cup \mathbb{D}_> \,,\; \eta \in [1,\infty)\,,
 \end{cases} 
\end{equation}
where $\gf=(\scalar{\boldsymbol{e}_i}{\Gf\boldsymbol{e}_{i}})_{i=1}^{2n}$ denotes the vector of diagonal entries of the resolvent $\Gf^z$.

For a non-random vector $\bs{y} \in \C^{2n}$ with $\norm{\bs{y}}_\infty \le 1$ we have 
\begin{equation} \label{eq:local_law_H_z_averaged}
\big|\scalar{\bs{y}}{\gf(\eta) - \ii \vf(\eta)} \big| \prec
 \begin{cases}
 \frac{1}{{n\eta}} & \text{ for } z \in \mathbb{D}_<\,,\; \eta \in [n^{-1+\eps},1] \,,
 \\
 \frac{1}{{n}} +\frac{1}{(n \eta)^2}& \text{ for } z \in \mathbb{D}_>\,,\; \eta \in [n^{-1+\eps},1] \,,
 \\
 \frac{1}{{n} \eta^2}& \text{ for } z \in \mathbb{D}_< \cup \mathbb{D}_> \,,\; \eta \in [1,\infty)\,.
 \end{cases} 
\end{equation}
\end{thm}

As an easy consequence we can now prove Corollary \ref{coro:eigenvector_delocalization}.
\begin{proof}[Proof of Corollary \ref{coro:eigenvector_delocalization}]
Let $y \in \C^n$ be an eigenvector of $X$ corresponding to the eigenvalue $\sigma \in \spec X$ with $\abs{\sigma}^2 \leq \rho(S) - \zsq_*$. Then 
 the $2n$-vector $(0,y)$ is contained in the kernel of $\Hf^\sigma$. Therefore, \eqref{eq:eigenvector_delocalization} is an easy consequence of \eqref{eq:local_law_H_z_version} 
(Compare with the proof of Corollary 1.14 in \cite{Ajankirandommatrix}). 
\end{proof}

We recall our normalization of the trace, $\tr \id = 1$, from \eqref{eq:def_trace}.

\begin{proof}[Proof of Theorem \ref{thm:local_law_H_z}]
Recall from the beginning of Section~\ref{sec:self_consistent_equation} how our problem can be cast  into  the setup of \cite{AjankiCorrelated}.
In the regime $z \in \mathbb{D}_<$ we follow the structure of the proof of Theorem~2.9 in \cite{AjankiCorrelated} and in the regime $z \in \mathbb{D}_>$ the proof of Proposition~7.1 in \cite{AjankiCorrelated} until the end of Step~1. In fact, the arguments from these proofs can be taken over directly with three important adjustments. The flatness assumption  \eqref{MDE:flatness} is used heavily in \cite{AjankiCorrelated} in order to establish bounds (Theorem~2.5 in \cite{AjankiCorrelated}) on the deterministic limit of the resolvent and for establishing the stability of the matrix Dyson equation, 
 cf. \eqref{Perturbed MDE} below,  (Theorem~2.6 in \cite{AjankiCorrelated}). 
Since this assumption is violated in our setup we present appropriately adjusted versions of these theorems  (Proposition~\ref{pro:estimates_v_small_z} and Proposition \ref{pro:stab_lemma} in \cite{AjankiCorrelated}).  
We will also take over the proof of the fluctuation averaging result (Proposition~\ref{prp:Fluctuation averaging} below) for $\Hf^z$ from \cite{AjankiCorrelated} since the flatness did not play a role in that proof at all. Note that the $\eta^{-2}$-decay in the spectral parameter regime $\eta\ge1$ was not covered in \cite{AjankiCorrelated}. But this decay simply follows by using the bounds $\norm{\Mf^z(\eta)}_{\rm{max}}+\norm{\Gf^z(\eta)}_{\rm{max}}\le \frac{2}{\eta}$ instead of just $\norm{\Mf^z(\eta)}_{\rm{max}}+\norm{\Gf^z(\eta)}_{\rm{max}}\le C$ along the proof. 

As in \cite{AjankiCorrelated} we choose a pseudo-metric $d$ on $\{1, \dots, 2n\}$. Here this pseudo-metric is particularly simple,
\[
d(i,j)\,\defeq\, 
\begin{cases}
0 & \text{ if } i=j \text{ or } i=j+n \text{ or } j=i+n \,,
\\
\infty & \text{otherwise}\,,
\end{cases}
\qquad i,j =1, \dots, 2n\,.
\]
With this choice of $d$ the matrix $\Hf^z$ satisfies all assumptions in \cite{AjankiCorrelated} apart from the {flatness}.

We will now show that as in \cite{AjankiCorrelated} the resolvent $\Gf^z$ satisfies the  \emph{perturbed matrix Dyson equation}
\bels{Perturbed MDE}{
-\id \,=\, (\ii\1\eta\2\id- {\boldsymbol A}^z + \widetilde{\cal{S}}[\Gf^z(\eta)])\Gf^z(\eta) + {\boldsymbol D}(\eta)\,.
}
Here, ${\boldsymbol A}^z$ is given by \eqref{MDE:Data}, 
\bels{definition of error matrix D}{
{\boldsymbol D^z}(\eta)\,\defeq\,-(\widetilde{\cal{S}}[\Gf^z(\eta)]+\Hf^z-{\boldsymbol A}^z)\Gf^z(\eta) \,,
}
 is a random error matrix and 
$\widetilde{\cal{S}}$ is a slight modification of the operator $\cal{S}$ defined in \eqref{MDE:Data},
\bels{definition of tilde S}{
\widetilde{\cal{S}}[\Wf]\,\defeq\, \E(\Hf^z-{\boldsymbol A}^z)\Wf(\Hf^z-{\boldsymbol A}^z)\,=\, 
\begin{pmatrix} \diag (S{w_2} )& T \odot W_{21}^t \\ T^* \odot W_{12}^t & \diag (S^t {w_1}) \end{pmatrix}\,.
}
Here, $\odot$ denotes the Hadamard product, i.e., for matrices $A=(a_{ij})_{i,j=1}^l$ and $B=(b_{ij})_{i,j=1}^l$, we define their Hadamard product through $(A \odot B)_{ij} \defeq a_{ij} b_{ij}$ for $i,j=1, \ldots, l$. 
Moreover, we used the conventions from \eqref{eq:convetion_Wf} for $\Wf$ and  introduced the matrix $T \in \C^{n \times n}$ with entries
\[
t_{i j}\,\defeq\, \E \2 x_{i j}^2\,.
\]
Note that in contrast to \cite{AjankiCorrelated} the matrix $\Mf$ solves \eqref{MDE on imaginary axis}, which is given in terms of the operator $\cal{S}$ and not  $\widetilde{\cal{S}}$. As we will see below this will not effect the proof, since the entries of the matrix $T$ are of order $N^{-1}$ and thus the off-diagonal terms in \eqref{definition of tilde S} of $\widetilde{\cal{S}}$ are negligible.

We will see that $\Df=\Df^z$ is small in the entrywise maximum norm
 \[
 \norm{\Wf}_{\rm{max}}\,\defeq\, \max_{i,j=1}^{2n}\abs{{w}_{i j}}\,,
 \]
 $\Wf=(w_{ij})_{i,j=1}^{2n}$, 
 and use the stability of \eqref{Perturbed MDE} to show that $\Gf(\eta)=\Gf^z(\eta)$ approaches $\Mf(\eta)=\Mf^z(\eta)$ defined in \eqref{eq:def_Mf} as $n\to \infty$, i.e., we will show that
 \bels{definition of Lambda}{
 \Lambda(\eta)\,\defeq\, \norm{\Gf(\eta)- \Mf(\eta)}_{\rm{max}}\,,
 }
 converges to zero.  
 For simplicity we will only consider the most difficult regime $z \in \mathbb{D}_<$ and $\eta \le 1$ inside the spectrum. The cases $z \in \mathbb{D}_>$ and $\eta \ge 1$ are similar but simpler and left to the reader.
We simply follow the proof in Section~3 of \cite{AjankiCorrelated} line by line until the flatness assumption is used. This happens for the first time inside the proof of Lemma~3.3. We therefore replace this lemma by the following modification.
 \begin{lemma} \label{lmm:D norm bounds}Let $z \in \Dsma$. Then
\[
\norm{\Df( \eta)}_{\rm{max}}\,\prec\, \frac{1}{\sqrt{n}}\,,\qquad  \eta \ge 1\,.
\]
Furthermore, we have
\bels{D max norm bound}{
\norm{\Df(\eta)}_{\rm{max}}\,\chi(\Lambda(\eta)\le n^{-\eps})\,\prec\, \frac{1}{\sqrt{n\eta}}\,,
\qquad\eta \in [n^{-1+\eps},1].
}
\end{lemma}

 To show Lemma~\ref{lmm:D norm bounds} we follow the proof of its analog, Lemma~3.3 in  \cite{AjankiCorrelated}, where the flatness assumption as well as the assumptions that the spectral parameter is in the bulk of the spectrum (formulated as $\rho(\zeta)\ge \delta$ in \cite{AjankiCorrelated}) are used only implicitly through the upper bound on $\bs{M}$ (Theorem~2.5 in \cite{AjankiCorrelated}). However, the conclusion of this theorem clearly still holds in our setup because $\bs{M}$ has the $2\times 2$-diagonal structure \eqref{eq:def_Mf} and the vectors $v_1,v_2$ and $u$ are bounded by Proposition~\ref{pro:estimates_v_small_z} and \eqref{eq:estimate_uf}.

We continue following the arguments of Section~3 of \cite{AjankiCorrelated}  using our Lemma~\ref{lmm:D norm bounds} above instead of Lemma~3.3 there. The next step that uses the flatness assumption is the stability 
of the MDE  (Theorem~2.6 in \cite{AjankiCorrelated}) which shows that the bound \eqref{D max norm bound} also implies
\[
\Lambda(\eta)\,\chi(\Lambda(\eta)\le n^{-\eps})\prec\frac{1}{\sqrt{n\eta}}\,.
\]
In our setup this stability result is replaced by the following lemma whose proof is postponed until the end of the proof of Theorem~\ref{thm:local_law_H_z}.

\begin{lemma}[MDE stability] \label{lmm:MDE stability} Suppose that some functions $D_{ab}, G_{ab}:  \R_+ \to \C^{n \times n}$ for $a,b=1,2$ satisfy \eqref{Perturbed MDE} with
\bels{Block decomposition for D and G}{
\bs{D}\,\defeq\, \begin{pmatrix} D_{11}  & D_{12} \\ D_{21} & D_{22}\end{pmatrix}\,,\qquad 
\bs{G}\,\defeq\, \begin{pmatrix} G_{11}  & G_{12} \\ G_{21} & G_{22}\end{pmatrix}\,,
}
and the additional constraints
\bels{symmetry constraints on G}{
\tr G_{11} \,=\, \tr G_{22}\,,\qquad \im \Gf \,=\, \frac{1}{2\ii}(\Gf-\Gf^*)\text{ is positive definite }.
}
 There is a constant $\lambda_* \gtrsim 1$, depending only on $\mathcal{P}$, such that 
\bels{MDE Stability}{
\norm{\Gf-\Mf}_{\mathrm{max}}\,\chi \,\lesssim\, \norm{\Df}_{\mathrm{max}} +\frac{1}{n}\,,
\qquad \chi\,\defeq\,\chi (\norm{\Gf-\Mf}_{\mathrm{max}}\le \lambda_*)\,,
}
uniformly for all $z \in {\mathbb D_{<}}\cup{\mathbb D_{>}}$, where $\Mf(\eta)=\Mf^z(\eta)$ is defined in \eqref{eq:def_Mf}. 

Furthermore, there exist eight matrix valued functions $R_{ab}^{(k)}: \R_+ \to \C^{n \times n}$ with $a,b,k=1,2$, depending only on $z$ and $S$, and satisfying $\norm{R^{(k)}_{ab}}_\infty\lesssim 1$, such that 
\bels{averaged matrix stability}{
\Big|\tr[\diag(\bs{y})(\Gf-\Mf)]\Big|\,\chi\lesssim\max_{a,b,k=1,2}\Big|\tr[\diag(R_{ab}^{(k)}y_k)D_{ab}]\Big|+\norm{\bs{y}}_\infty\Big(\frac{1}{n}+\norm{\bs{D}}_{\rm{max}}^2\Big),
}
uniformly for all  $z \in {\mathbb D_{<}}\cup{\mathbb D_{>}}$ and $\bs{y}=(y_1,y_2) \in \C^{2n}$.
\end{lemma}

The important difference between Theorem~2.6 in \cite{AjankiCorrelated} and Lemma~\ref{lmm:MDE stability} above is the additional assumption \eqref{symmetry constraints on G} imposed on the solution of the perturbed MDE. This assumption is  satisfied for the resolvent of the matrix $\Hf^z$ because of the $2\times2$-block structure \eqref{eq:def_H_z}.
In fact with the block decomposition for $\Gf$ as in \eqref{Block decomposition for D and G} we have
\[
G_{11}(\eta)\,=\,\frac{\ii\1\eta\1 \id}{(X-z\id)(X-z\id)^*+\eta^2\1\id} \,,\qquad G_{22}(w)\,=\,\frac{\ii\1\eta\1 \id}{(X-z\id)^*(X-z\id)+\eta^2\1\id}\,.
\]
Using Lemma~\ref{lmm:MDE stability} in the remainder of the proof of the entrywise local law in Section~3 of \cite{AjankiCorrelated} finishes the proof of \eqref{entrywise matrix local law}.

To see \eqref{eq:local_law_H_z_averaged} we use the fluctuation averaging mechanism, which was first established for generalized Wigner matrices with Bernoulli entries in \cite{EYYBern}.  The following proposition is stated and proven as Proposition~3.4 in \cite{AjankiCorrelated}. Since the flatness condition was not used in its proof at all, we simply take it over. 
\begin{proposition}[Fluctuation averaging] \label{prp:Fluctuation averaging} Let $z \in \mathbb{D}_<\cup\mathbb{D}_>$, $\eps \in (0,1/2)$, $\eta\geq n^{-1}$ and $\Psi$ a non-random control parameter such that 
$n^{-1/2} \leq \Psi \leq n^{-\eps}$. Suppose the local law holds true in the form
\[
\norm{\Gf(\eta)-\Mf(\eta)}_{\rm{max}}\,\prec\, \Psi\,.
\]
Then for any non-random vector $y \in \C^{n}$ with $\norm{{y}}_\infty\leq 1$ we have
\[
\max_{a,b=1,2}\Big|\tr[\diag(y)D_{ab}]\Big|\,\prec\,\Psi^2\,,
\]
where  $D_{ab} \in \C^{n \times n}$, $a,b=1,2$, are the blocks of the error matrix 
\[
\bs{D}(\eta)\,=\, \begin{pmatrix} D_{11}  & D_{12} \\ D_{21} & D_{22}\end{pmatrix}\,,
\]
which was defined in \eqref{definition of error matrix D}. 
\end{proposition}

Using this proposition the averaged local law \eqref{eq:local_law_H_z_averaged} follows from \eqref{entrywise matrix local law} and \eqref{averaged matrix stability}. This finishes the proof of Theorem~\ref{thm:local_law_H_z}.
\end{proof}
\begin{proof}[Proof of Lemma~\ref{lmm:MDE stability}]
We write \eqref{Perturbed MDE} in the $2\times 2$ - block structure
\bels{MDE in block form}{
&\begin{pmatrix}\diag(\ii\1\eta + Sg_2)  & z\1\id  \\  \overline{z}\1\id & \diag(\ii\1\eta + S^tg_1)\end{pmatrix}
\begin{pmatrix} G_{11}  & G_{12} \\ G_{21} & G_{22}\end{pmatrix}
\\
&\mspace{100mu}=\, 
-\begin{pmatrix} \id & 0 \\ 0 & \id \end{pmatrix}
-\begin{pmatrix} D_{11}+(T \odot G_{21}^t)G_{21} & D_{12}+(T \odot G_{21}^t)G_{22} \\ D_{21}+(T^* \odot G_{12}^t )G_{11}& D_{22}+(T^* \odot G_{12}^t )G_{22}\end{pmatrix}
,
}
where we introduced $\gf=(g_1,g_2) \in \C^{2n}$, the vector of the diagonal elements of $\Gf$. 

We restrict the following calculation to the regime where $\norm{\Gf(\eta)-\Mf(\eta)}_{\mathrm{max}}\le \lambda_*$ for some sufficiently small $\lambda_*$ in accordance with the characteristic function on the left hand side of \eqref{MDE Stability}. In particular,
\bels{MDE:diagonal a priori bound}{
\norm{\gf(\eta)-\ii\1\vf(\eta)}_\infty\,\le\, \lambda_*\,.
}
Since by \eqref{eq:v}  and \eqref{eq:def_Mf}  the identity 
\[
\begin{pmatrix}\ii\diag(\eta + Sv_2(\eta))  & z\1\id  \\  \overline{z}\1\id & \ii\diag(\eta + S^tv_1(\eta))\end{pmatrix}^{-1}
\,=\,- \Mf(\eta)\,,
\]
holds
 we infer from the smallness of $\norm{\gf-\ii \vf}_{\mathrm{max}}$ that the inverse of the first matrix  factor on the left hand side of \eqref{MDE in block form} is bounded and satisfies
\bels{approximate of lhs inverse matrix}{
\bigg\|\begin{pmatrix}\diag(\ii\1\eta + Sg_2)  & z\1\id  \\  \overline{z}\1\id & \diag(\ii\1\eta + S^tg_1)\end{pmatrix}^{-1}+\Mf\bigg\|_{\mathrm{max}}\,\lesssim\, \norm{\gf-\ii\1\vf}_{\mathrm{max}}\,.
}
Using this in \eqref{MDE in block form} yields
\bels{MDE: G equation}{
\Gf+\begin{pmatrix}\diag(\ii\1\eta + Sg_2)  & z\1\id  \\  \overline{z}\1\id & \diag(\ii\1\eta + S^tg_1)\end{pmatrix}^{-1}\mspace{-10mu}=\Mf \bs{D} +O\Big( \norm{\gf-\vf}_{\mathrm{max}}\norm{\Df}_{\rm{max}}+\norm{\Gf-\Mf}_{\mathrm{max}}^2+\frac{1}{n}\Big),\mspace{-3mu}
}
where we applied the simple estimate
\bels{estimate on TGG}{
\norm{(T \odot G_{ab}^t)G_{cd}}_{\mathrm{max}}\,\lesssim\, \norm{\Gf-\Mf}_{\mathrm{max}}^2+\frac{1}{n}\norm{\Gf-\Mf}_{\mathrm{max}}\norm{\Mf}_{\mathrm{max}}+\frac{1}{n} \norm{\Mf}_{\mathrm{max}}^2
\,\lesssim\,  \norm{\Gf-\Mf}_{\mathrm{max}}^2+\frac{1}{n}\,,
}
which follows from
\[
\norm{T}_{\mathrm{max}}\,\lesssim\, \frac{1}{n}\,.
\]

Thus  the diagonal elements $\gf$ of $\Gf$ satisfy \eqref{eq:combined_v_perturbed} with an error term $\df$ that is given by
\bels{MDE: d estimate}{
\df \,=\, ((\Mf \Df)_{ii})_{i=1}^{2n}+ O\Big(\norm{\Gf-\Mf}_{\mathrm{max}}^2+\frac{1}{n}\Big)\,.
}
Here we used $\norm{\Df}_{\mathrm{max}}\lesssim\norm{\Gf-\Mf}_{\mathrm{max}}$, which follows directly from  \eqref{Perturbed MDE}  and \eqref{MDE on imaginary axis}. With \eqref{eq:stability_estimate1} and \eqref{eq:stability_estimate_average}  in Proposition~\ref{pro:stab_lemma}, the stability result on \eqref{eq:combined_v_perturbed}, we conclude that
\bels{conclusion from entrywise stability}{
\norm{\gf-\ii\1\vf}_\infty\,\lesssim\, \norm{\Df}_{\mathrm{max}}+\norm{\Gf-\Mf}_{\mathrm{max}}^2+\frac{1}{n}\,,
}
and that
\bels{conclusion from averaged stability}{
\abs{\scalar{\bs{y}}{\gf-\ii\1\vf}}\,\lesssim\,\Big|\tr[\diag(\bs{R}\bs{y})\Mf \Df]\Big|+ \norm{\Df}_{\mathrm{max}}^2+\norm{\Gf-\Mf}_{\mathrm{max}}^2+ \frac{1}{n}\,,
}
for some bounded $\bs{R} \in \C^{2n \times 2n}$ and any $\bs{y} \in \C^{2n}$  with $\norm{\bs{y}}_\infty\le 1$, respectively.
Combining \eqref{approximate of lhs inverse matrix} with \eqref{MDE: G equation} and \eqref{conclusion from entrywise stability} yields
\[
\norm{\Gf-\Mf}_{\mathrm{max}}\,\lesssim\,  \norm{\Df}_{\mathrm{max}}+\norm{\Gf-\Mf}_{\mathrm{max}}^2+ \frac{1}{n}\,.
\]
By choosing $\lambda_*$ sufficiently small we may absorb the quadratic term of the difference $\Gf-\Mf$ on the right hand side into the left hand side and \eqref{MDE Stability} follows. Using \eqref{MDE Stability} in \eqref{conclusion from averaged stability} to estimate the term $\norm{\Gf-\Mf}_{\mathrm{max}}^2$ proves \eqref{averaged matrix stability}. 
\end{proof}

We use a standard argument to conclude from \eqref{eq:local_law_H_z_averaged} the following statement about the number of eigenvalues $\eigH_i(z)$ of $\Hf^z$ in a small interval centered at zero.
\begin{lem}\label{lemma:small} Let $\eps >0$. Then 
\begin{equation} \label{eq:estimate_number_of_eigenvalues_weak_local_law}
\# \big\{ i\; :\; \abs{\eigH_i(z)}\le \eta  \big\} \prec n \1\eta\,,
\end{equation}
uniformly for all $\eta \ge n^{-1+\eps}$ and $z \in \Dsma$.

Furthermore, we have 
\begin{equation} \label{eq:strong}
\sup_{z\in \Dbig}\frac{1}{\abs{\eigH_i(z)}}\,\prec\, n^{1/2}\,.
\end{equation}
\end{lem} 
\begin{proof} For the proof of \eqref{eq:estimate_number_of_eigenvalues_weak_local_law} we realize that \eqref{entrywise matrix local law} implies a uniform bound on the resolvent elements up to the spectral scale $\eta \ge n^{-1+\eps}$. Thus we have
\[
\frac{\#\Sigma_\eta}{2\eta}\,\le\, \sum_{i \in \Sigma_\eta}\frac{\eta }{\eta^2 + \eigH_i(z)^2} \,\le\,2n \im \tr \Gf^{z}(\eta)\,\prec\, n \,,
\]
where $\Sigma_\eta\defeq\{i: \abs{\eigH_i(z)}\le \eta\}$. Here, we used the normalization of the trace \eqref{eq:def_trace}. 

Before proving \eqref{eq:strong}, we first establish that  
\begin{equation} \label{eq:estimate_number_of_eigenvalues_weak_local_law_outside}
\frac{1}{\abs{\eigH_i(z)}}\,\prec\, n^{1/2}\,,
\end{equation}
uniformly for $z \in \Dbig$. 
We use \eqref{eq:local_law_H_z_averaged} and $\avg{\bs{v}(\eta)}\sim \eta$ to estimate
\begin{equation}\label{SG}
\frac{\eta}{\eta^2+\eigH_i(z)^2}\,\le\, 2n \im \tr \Gf^{z} (\eta)\,\prec\, n \1\eta + \frac{1}{n \eta^2}\,,
\end{equation}
with the choice $\eta\defeq n^{-1/2-\eps}$ for any $\eps>0$. This immediately implies $\abs{\eigH_i(z)}^{-1}\prec n^{1/2+\eps}$, hence \eqref{eq:estimate_number_of_eigenvalues_weak_local_law_outside}. 
For the stronger bound \eqref{eq:strong} we use that $z\mapsto \im \tr \Gf^{z}(\eta)$ is a Lipschitz 
continuous function (with a Lipschitz constant  $C\eta^{-2}$ uniformly in $z$) and that $\Dbig$ is compact, so the second bound in \eqref{SG} 
holds even after taking the supremum over $z\in \Dbig$. Thus
\[
 \sup_{z\in \Dbig} \frac{\eta}{\eta^2+\eigH_i(z)^2}\,\le\, 2n \sup_{z\in \Dbig} \im \tr \Gf^{z} (\eta)\,\prec\, n \1\eta + \frac{1}{n \eta^2}\,
\]
holds for $\eta\defeq n^{-1/2-\eps}$.
From the last inequality we easily conclude \eqref{eq:strong}.
\end{proof}

\subsection{Local inhomogeneous circular law}

We start with an estimate on the smallest singular value of $X-z\id$ which will be used to control the $\di\eta$-integral in the second term on the right-hand side of \eqref{eq:master_formula}
for $\eta \leq n^{-1+\eps}$. Notice that Proposition \ref{pro:least_singular_value} is the only result in our proof of Theorem \ref{thm:local_circular_law} which requires the entries of $X$ to have 
a bounded density.

Adapting the proof of \cite[Lemma 4.12]{bordenave2012} with the bounded density assumption to our setting, we obtain the following proposition.
\begin{pro}[Smallest singular value of $X-z\id$] \label{pro:least_singular_value}
Under the condition \eqref{eq:bounded_density}, there is a constant $C$, depending only on $\alpha$, such that 
\begin{equation} \label{eq:small_singular_value_estimate}
 \P\left( \min_{i=1}^{2n}\abs{\eigH_i(z)}\leq \frac{u}{n} \right) \leq  C u^{2\alpha/(1+\alpha)} n^{\beta + 1} 
\end{equation}
for all $u>0$ and $z \in \C$.
\end{pro}

\begin{proof}
We follow the proof in \cite{bordenave2012} and explain the differences. 
Let $R_1, \ldots, R_n$ denote the rows of $\sqrt{n}X-z\id$.  Proceeding as in \cite{bordenave2012} but using our normalization conventions, we are left with estimating 
\[ \P \left( n \abs{\scalar{R_i}{y}} \leq \frac{u}{\sqrt n}  \right) \]
uniformly for $u$ and for arbitrary $y \in \C^n$ satisfying $\normtwo{y} = 1/\sqrt{n}$. 
We choose $j \in \{1, \ldots, n\}$ such that $\abs{y_j} \geq 1/\sqrt n$ and compute the conditional probability 
\[ \P_{ij} \defeq \P \Big( n\abs{\scalar{R_i}{y}} \leq \frac{u}{\sqrt n} \Big| x_{i1}, \ldots, \wh{x_{ij}}, \ldots, x_{in} \Big)  = \int_\C \char\left( 
\absa{\frac{a}{y_{j}} + w } \leq \frac{u}{y_j \sqrt n}\right) f_{ij}(w) \di^2 w, \]
where $a$ is independent of $x_{ij}$. 
Using \eqref{eq:bounded_density} and $\abs{y_j} \geq 1/\sqrt n$, we get 
\[ \abs{\P_{ij}  } \leq \absa{\pi\frac{u}{y_j \sqrt n}}^{2\alpha/(1+\alpha)} \norm{f_{ij}}_{1+\alpha} \leq (\pi u)^{2\alpha/(1+\alpha)} n^{ \beta}. \]
Thus, $\P\left( n\abs{\scalar{R_i}{y}} \leq u/\sqrt{n} \right) \leq (\pi u)^{2\alpha/(1+\alpha)} n^{\beta}$ which concludes the proof of \eqref{eq:small_singular_value_estimate} as in \cite{bordenave2012}.
\end{proof}

For the following proof of Theorem \ref{thm:local_circular_law} we recall that without loss of generality, we are assuming that $\rho(S) =1$ which can be obtained by a simple rescaling of $X$. 
Moreover, from \eqref{eq:def_Dsma_Dbig}, for $\zsq_*>0$ and $\zsq^* > 1 + \zsq_*$, we recall the notations 
\[ \Dsma \defeq \{ z \in \C \;\mid\; \abs{z}^2 \leq 1 - \zsq_* \}, \quad \Dbig \defeq \{ z \in \C \;\mid\; 1 + \zsq_* \leq \abs{z}^2 \leq \zsq^* \}. \]

\begin{proof}[Proof of Theorem \ref{thm:local_circular_law}]
We start with the proof of part (i) of 
Theorem \ref{thm:local_circular_law}. We will estimate each term on the right-hand side of \eqref{eq:master_formula}. 
Let $z_0 \in \Dsma$. We suppress the $\zsq$ dependence of $v_1$ in this proof but it will always be evaluated at $\zsq = \abs{z}^2$. 

As $\supp f \subset D_\varphi(0)$, $a>0$ and $z_0 \in \Dsma$ we can assume that the integration domains of the $\di^2 z$ integrals in \eqref{eq:master_formula} are $\Dsma$ instead of $\C$. 
Hence, it suffices to prove every bound along the proof of (i) uniformly for $z\in \Dsma$.

To begin, we estimate the first term in \eqref{eq:master_formula}. Since 
\[ \log \abs{\det(\Hf^z-\ii T\id)} = 2 n \log T + \sum_{j=1}^n \log \left( 1 + \frac{\eigH_j^2}{T^2}\right) \] 
and the integral of $\Delta f_{z_0,a}$ over $\C$ vanishes as $f \in  C_0^2(\C)$, 
we obtain 
\begin{equation} \label{eq:master_formula_first_term}
 \absa{ \frac{1}{4\pi n} \int_\C \Delta f_{z_0,a}(z) \log \abs{\det(\Hf^z -\ii T\id)} \di^2 z} \leq \frac{1}{2\pi} \int_\C \abs{\Delta f_{z_0,a}(z)} \frac{\tr\left((\Hf^z)^2\right)}{T^2} \di^2 z.
\end{equation}
Here, we used $\log (1 + x) \leq x$ for $x\geq 0$. 
Furthermore, if $\abs{z} \leq 1$, then we have
\begin{equation} \label{eq:computation_trace_H_z}
 \tr((\Hf^z)^2) = \frac{1}{n} \sum_{i,j=1}^{n} (x_{ij} - z \delta_{ij})(\overline{x_{ij}} - \bar z \delta_{ij})  \leq \frac{2}{n} \sum_{i,j=1}^n \abs{x_{ij}}^2 + 2 \abs{z}^2 \prec 1,
\end{equation}
where we applied \eqref{eq:def_trace} in the first and \eqref{eq:trivial_control_x_ij_prec} in the last step. 
Therefore, choosing $T \defeq n^{100}$, we conclude from \eqref{eq:master_formula_first_term} and \eqref{eq:computation_trace_H_z} that the first term in \eqref{eq:master_formula} is stochastically 
dominated by $n^{-1+2a} \norm{\Delta f}_1$.

To control the second term on right-hand side of \eqref{eq:master_formula}, we define  
\begin{equation} \label{eq:estimate_eta_integral}
 I(z) \defeq \int_0^T \absa{\Im m^z(\ii \eta)  - \avg{v_1( \eta )}} \di \eta 
\end{equation}
for $z \in \Dsma$. We will conclude below the following lemma. 
\begin{lem} \label{lem:moment_bound_I_z} 
For every $\delta>0$ and $p \in \N$, there is a positive constant $C$, depending only on $\delta$ and $p$
 in addition to the model parameters and $\zsq_*$, such that  
\begin{equation} \label{eq:moments_I_z}
\sup_{z\in\Dsma} \E I(z)^p \leq C \frac{n^{\delta p}}{n^p}. 
\end{equation}
\end{lem}

We now show that this moment bound on $I(z)$ will yield that the second term in \eqref{eq:master_formula} is $\prec n^{-1 + 2a} \norm{\Delta f}_1$.
Indeed, for every $p\in\N$ and $\delta >0$, using Hölder's inequality, we estimate 
\begin{align} 
 \E \absa{ \int_\C \Delta f_{z_0,a}(z) \int_0^T \left[\Im m^z(\ii \eta)  - \avg{v_1(\eta)} \right] \di \eta\;  \di^2 z }^p  & \leq \int_\C \ldots \int_\C \prod_{i=1}^p \abs{\Delta f_{z_0,a}(z_i)} 
\prod_{i=1}^p \left(\E I(z_i)^p\right)^{1/p} \di^2 z_1 \ldots \di^2 z_p \nonumber \\
& \leq C \norm{\Delta f}_1^p\frac{n^{\delta p + 2a p}}{n^p} . \label{eq:moment_bound_second_term}
\end{align}
Applying Chebyshev's inequality to \eqref{eq:moment_bound_second_term} and using that $\delta>0$ and $p$ were arbitrary, we get 
\[  \absa{\int_\C  \Delta f_{z_0,a}(z) \int_0^T \Im m^z(\ii \eta)  - \avg{v_1(\eta)} \di \eta \; \di^2 z } \prec n^{-1 + 2a } \norm{\Delta f}_1. \]
Hence, the bound on the second term on the right-hand side of \eqref{eq:master_formula} follows once we have proven \eqref{eq:moments_I_z}.

For the third term in \eqref{eq:master_formula}, notice that the integrand is bounded by $C\eta^{-2}$ so it is bounded by $n^{2a} T^{-1}\norm{\Delta f}_1$. 
This concludes the proof of (i) of Theorem \ref{thm:local_circular_law} up to the proof of Lemma \ref{lem:moment_bound_I_z} which is given below. 

We now turn to the proof of (ii). 
We will use an interpolation between the random matrix $X$ and an independent Ginibre matrix $\wh X$  together with
 the well-known result that a Ginibre matrix does not have any eigenvalues $\abs{\lambda} \geq 1 + \zsq_*$ 
with very high probability. With the help of \eqref{eq:strong} we will control the number of eigenvalues 
outside of the disk of radius $1+\zsq^*$ along the flow. We fix $\zsq^* > 1 + \zsq_*$.

Let $(\wh x_{ij})_{i,j=1}^n$ be independent  centered complex
 Gaussians of   variance $n^{-1}$, i.e., $\E \,\wh x_{ij} = 0$ and $\E \abs{\wh x_{ij}}^2 = n^{-1}$. We set $\wh X \defeq (\wh x_{ij})_{i,j=1}^n$,
i.e.  $\wh X$ is a Ginibre matrix.  We denote the eigenvalues of $\wh X$ by $\wh \eigX_1, \ldots, \wh \eigX_n$.

For $t\in [0,1]$, we denote the spectral radius of the matrix $t S + (1-t)E$ by $\rho_t \defeq \rho(t S + (1-t)E)$, where $E$ is the $n\times n$ matrix with entries $e_{ij} \defeq 1/n$, $E = (e_{ij})_{i,j=1}^n$. 
Furthermore, we define 
 \[ X^t \defeq \rho_t^{-1/2}\left(t X + (1-t)\wh X\right), \quad \Hf^{z,t} \defeq \begin{pmatrix} 0 & X^t - z\id \\ \left(X^t - z \id\right)^* & 0 \end{pmatrix}   \]
for $t \in [0,1]$. The eigenvalues of $X^t$ and $\Hf^{z,t}$ are denoted by $\eigX^t_i$ and $\eigH_k^t(z)$, respectively, for $i=1, \ldots, n$ and $k=1,\ldots, 2n$. The one parameter family $t\mapsto X^t$ interpolates between $X$ and $\wh X$ by keeping the spectral radius 
of the variance matrix at constant one.

Note that $\normtwo{(X^t-z)^{-1}} = \max_{k=1}^{2n} \abs{\eigH_k^t(z)}^{-1}$. 
We can apply  Lemma~\ref{lemma:small} to the matrices $X^t$ for any $t$
to  get
\[ 
  \sup_{z\in \Dbig} \normtwoa{(X^t-z)^{-1}} \prec n^{1/2} 
\]
uniformly in $t$ from \eqref{eq:strong}.
In fact, the estimate can be strengthened to  
\begin{equation} \label{eq:resolvent_bound}
 \sup_{t \in [0,1]} \sup_{z\in \Dbig} \normtwoa{(X^t-z)^{-1}} \prec n^{1/2} 
\end{equation}
exactly in the same way as \eqref{eq:estimate_number_of_eigenvalues_weak_local_law_outside}
was strengthened to \eqref{eq:strong}, we only need to observe that the two parameter
family $(z, t) \mapsto \im \tr \Gf^{z,t}(\eta)$ is Lipschitz continuous in both variables, where $\Gf^{z,t}$ 
denotes the resolvent of $\Hf^{z,t}$.

Let $\gamma$ be the circle in $\C$ centered at zero with radius $1 +\zsq_*$. For $t\in [0,1]$, we have
\[ N(t) \defeq \# \{ i \mid \abs{\eigX_i^t} \leq 1 +\zsq_* \} = \frac{n}{2\pi\ii} \int_{\gamma} \tr\left((X^t-z)^{-1}\right) \di z,  \]
where $\tr \colon \C^{n\times n} \to \C$ denotes the normalized trace, i.e., $\tr \id = 1$. 
Due to \eqref{eq:resolvent_bound} $N(t)$ is a continuous function of $t$. Thus, $N(t)$ is constant as a continuous integer valued function. 

Using Corollary 2.3 of \cite{fey2008}, we obtain that $\# \{ k \mid \abs{\wh \eigX_k} \geq \zsq^*\} = 0$ with very high probability.
Furthermore, $\# \{ k \mid \wh \eigX_k \in \Dbig \} = 0$ with very high probability by \eqref{eq:resolvent_bound}. Thus, 
\[ N(1) = N(0) = n - \# \{ k \mid \wh \eigX_k \in \Dbig \} - \# \{ k \mid \abs{\wh \eigX_k} \geq \zsq^*\} = n  \]
with very high probability which concludes the proof of (ii) and hence of Theorem \ref{thm:local_circular_law}.
\end{proof}

\begin{rem}
In the above proof  we showed  that  $\| \Hf^z \|\le C$ with very high probability
via an interpolation argument using the norm-boundedness of a Ginibre
matrix and the local law for the entire interpolating family. Robust upper bounds
on the norm of random matrices are typically  proven by a simple moment method.
Such approach also applies here. For example, one  may follow  the proof 
of Lemma 7.2 in  \cite{EYYbulk},
and estimate every moment $\E |x_{ij}|^k$ by its maximum over all $i, j$.
The final constant estimating $\|\Hf^z\|$ will not be optimal due to these
crude bounds, but it will still only depend on $s^*$ and $\mu_m$ from 
\eqref{eq:assumption_A}, \eqref{eq:assumption_B}. 
This argument is very robust, in particular it does not use Hermiticity.
\end{rem}

\begin{proof}[Proof of Lemma \ref{lem:moment_bound_I_z}]
To show \eqref{eq:moments_I_z}, we use the following estimate which converts a bound in $\prec$ into a moment bound. For every nonnegative random variable satisfying $Y \prec 1/n$ and $Y \leq n^c$ for some $c>0$ 
the $p^{\text{th}}$ moment is bounded by 
\begin{equation} \label{eq:prec_to_moment}
 \E Y^p \leq \E Y^p \char(Y \leq n^{\delta-1}) + \left(\E Y^{2p}\right)^{1/2} \left(\P \left(Y > n^{\delta-1}\right)\right)^{1/2} \leq C\frac{n^{p\delta}}{n^p},
\end{equation}
for all $p\in\N$, $\delta >0$ and for some $C>0$, depending on $c$, $p$ and $\delta$.

As a first step in the proof of \eqref{eq:moments_I_z}, we choose $\eps \in (0,1/2)$, split the $\di \eta$ integral in the definition of $I(z)$, \eqref{eq:estimate_eta_integral}, and consider the regimes
$\eta \leq n^{-1+\eps}$ and $\eta \geq n^{-1+\eps}$, separately. For $\eta\le n^{-1+\eps}$, we compute 
\begin{equation*}
 \int_0^{n^{-1 + \eps}} \Im m^z(\ii \eta) \di \eta = \frac{1}{2n} \sum_{i=1}^n \log \left( 1 + \frac{n^{-2+2\eps}}{\eigH_i^2} \right).
\end{equation*}
We recall that $\eigH_1, \ldots, \eigH_{2n}$ are the eigenvalues of $\Hf^z$.
Therefore, \eqref{eq:estimate_eta_integral} yields
\begin{align}
 \int_0^T \left[\Im m^z(\ii \eta) - \avg{v_1(\eta)}\right] \di \eta & = \frac{1}{n} \sum_{\abs{\eigH_i} < n^{-l}} \log \left( 1 + \frac{n^{-2+2\eps}}{\eigH_i^2} \right) + 
\frac{1}{n} \sum_{\abs{\eigH_i} \geq n^{-l}} \log \left( 1 + \frac{n^{-2+2\eps}}{\eigH_i^2} \right) - \int_0^{n^{-1+\eps}} \avg{v_1(\eta)} \di \eta \nonumber \\ 
  & ~~ + \int_{n^{-1+\eps}}^1 \left[\Im m^z(\ii \eta) - \avg{v_1(\eta)}\right] \di \eta + \int_1^T\left[ \Im m^z(\ii \eta) - \avg{v_1(\eta)}\right] \di \eta. \label{eq:decomposition_second_term_master_formula}
\end{align}
Here, $l\in\N$ is a large fixed integer to be chosen later. 

We will estimate each of the terms on the right-hand side of \eqref{eq:decomposition_second_term_master_formula} 
individually. 
For the first term in \eqref{eq:decomposition_second_term_master_formula}, we compute   
\[ \E \left( \frac{1}{n}\sum_{\abs{\eigH_i}\leq n^{-l}} \log \left( 1 + \frac{n^{-2 +2\eps}}{\eigH_i^2} \right) \right)^p 
\leq \E \left[ \log^p \left( 1 + \frac{n^{-2+2\eps}}{\eigH_n^2}\right) \char(\eigH_n \leq n^{-l}) \right] 
\leq C \E\left[ \abs{\log \eigH_n}^{p}\char(\eigH_n \leq n^{-l}) \right]   \]
for some constant $C>0$ independent of $n$.
We compute the expectation directly
\[ \E \left[\abs{\log \eigH_n}^{p} \char(\eigH_n \leq n^{-l})\right]  = p\int_{{ l \log n}}^\infty \P \left( \eigH_n \leq \ee^{-t}\right)t^{p-1} \di t \leq C n^{\beta + 1 + 2\alpha/(1+\alpha)}
\int_{{ l \log n}}^\infty t^{p-1} \ee^{-2\alpha t/(1+\alpha)} \di t. \]
Here, we applied \eqref{eq:small_singular_value_estimate} in Proposition \ref{pro:least_singular_value} 
with $u = \ee^{-t}{n}$. Choosing $l$ large enough, depending on $\alpha$, $\beta$ and $p$, we obtain that the 
right-hand side is 
smaller than $n^{-p}$. This shows the bound \eqref{eq:moments_I_z} for the first term in \eqref{eq:decomposition_second_term_master_formula}.

We will apply \eqref{eq:prec_to_moment} for estimating the absolute value of the second, fourth and fifth term on the right-hand side of \eqref{eq:decomposition_second_term_master_formula}. 
For the first term, we will need a separate argument based on Proposition \ref{pro:least_singular_value}. 

To estimate the second term on the right-hand side of \eqref{eq:decomposition_second_term_master_formula}, we decompose the sum into three regimes, 
$n^{-l} \leq \abs{\eigH_i} < n^{-1+\eps}$, $n^{-1+\eps} \leq \abs{\eigH_i} < n^{-1/2}$ and $n^{-1/2} \leq \abs{\eigH_i}$.

For the first regime, we use \eqref{eq:estimate_number_of_eigenvalues_weak_local_law} with $\eta = n^{-1+\eps}$ and $\log( 1+ n^{-2 + 2 \eps +l}) \leq C \log n$ to get
\begin{equation} \label{eq:smallest_singular_value_term4}
\frac{1}{n} \sum_{\abs{\eigH_i} \in [n^{-l},n^{-1+\eps}] } \log\left(1 + \frac{n^{-2 +2 \eps}}{\eigH_i^2}\right) \leq  \frac{C \log n}{n} \# \{ i \colon \abs{\eigH_i} \leq n^{-1 +\eps} \} 
\prec \frac{n^\eps}{n}.
\end{equation}
As this sum is clearly polynomially bounded in $n$ we can apply \eqref{eq:prec_to_moment} to conclude that the first regime of the second term in \eqref{eq:decomposition_second_term_master_formula} fulfills the moment 
bound in \eqref{eq:moments_I_z}. 

For the intermediate regime, due to the symmetry $\spec(\Hf^z) = - \spec(\Hf^z)$, we only consider the positive eigenvalues. 
We decompose the interval $[n^{-1+\eps}, n^{-1/2}]$ into dyadic intervals of the form $[\eta_k, \eta_{k+1}]$, where $\eta_k \defeq 2^k n^{-1 + \eps}$.  Thus, we obtain 
\begin{equation} \label{eq:smallest_singular_value_term2}
\frac{1}{n} \sum_{\abs{\eigH_i} \in [n^{-1+\eps},n^{-1/2}] } \log\left(1 + \frac{n^{-2 +2 \eps}}{\eigH_i^2}\right) \leq \frac{2}{n} \sum_{k=0}^{N} \sum_{\eigH_i \in [\eta_k, \eta_{k+1}]} 
\log\left(1 + \frac{n^{-2 +2 \eps}}{\eigH_i^2}\right) \prec \frac{n^\eps}{n},
\end{equation}
where we introduced $N = O(\log n)$ in the first step. Moreover, we used the monotonicity of the logarithm, $\log ( 1 + x ) \leq x$ in the last step and the following consequence of 
\eqref{eq:estimate_number_of_eigenvalues_weak_local_law}:
\[ \# \{ i \colon \eigH_i \in [\eta_k, \eta_{k+1}]\}\leq \# \{ i \colon \abs{\eigH_i} \leq \eta_{k+1}\} \prec n^\eps 2^{k+1}. \]
The left-hand side of \eqref{eq:smallest_singular_value_term2} is trivially bounded by $\log 2$. 
Therefore, applying \eqref{eq:prec_to_moment} to the left-hand side of~\eqref{eq:smallest_singular_value_term2}, 
we conclude that it satisfies the moment estimate in \eqref{eq:moments_I_z}. 

For estimating the second term in \eqref{eq:decomposition_second_term_master_formula} in the third regime, employing $\abs{\eigH_i} \geq n^{-1/2}$ and $\log( 1 + x) \leq x$, 
we obtain
\begin{equation} \label{eq:smallest_singular_value_term3}
 \frac{1}{n}\sum_{\abs{\eigH_i} \geq n^{-1/2}} \log\left( 1+ \frac{n^{-2 + 2\eps}}{\eigH_i^2}\right) \leq \frac{1}{n} \sum_{\abs{\eigH_i} \geq n^{-1/2}} \log\left( 1+ n^{-1 + 2\eps}\right) \leq \frac{2n^{2\eps}}{n}. 
\end{equation}
Here, we used that $\Hf^z$ has $2n$ eigenvalues (counted with multiplicities). This deterministic bound and \eqref{eq:prec_to_moment} imply that the moments of this sum are bounded by the right-hand side in 
\eqref{eq:moments_I_z}.

Combining the estimates in these three regimes, 
\eqref{eq:smallest_singular_value_term4} , \eqref{eq:smallest_singular_value_term2} and \eqref{eq:smallest_singular_value_term3}, we conclude that the second term in \eqref{eq:decomposition_second_term_master_formula}
satisfies the moment bound in \eqref{eq:moments_I_z}. 

We now estimate the third term on the right-hand side of \eqref{eq:decomposition_second_term_master_formula}.  Since 
 $\vf\sim 1$ for $z \in \Dsma$ and $\eta\leq 1$ by \eqref{eq:bound_vf_small_eta}, the $p^{\text{th}}$ power of the third 
term is immediately bounded by the right-hand side of \eqref{eq:moments_I_z}. 

To bound the fourth and fifth term in \eqref{eq:decomposition_second_term_master_formula},
we note that $\Im m^z(\ii \eta) = \avg{\gf(\eta)}$ for $\eta >0$ and recalling the choice $T=n^{100}$, we obtain
\begin{equation} \label{eq:bounds_master_formula_large_eta}
 \int_{n^{-1+\eps}}^1 \absa{\Im m^z(\ii \eta) - \avg{v_1(\eta)}} \di \eta  \prec \frac{n^\eps}{n}, \qquad
\int_1^T \absa{\Im m^z(\ii \eta) - \avg{v_1(\eta)}} \di \eta  \prec \frac{1}{n}  
\end{equation}
from the first and third regime in \eqref{eq:local_law_H_z_averaged} with $\yf =1$. As the integrands are bounded by $n^2$ trivially \eqref{eq:prec_to_moment} yields that the moments of the fourth and fifth term 
in \eqref{eq:decomposition_second_term_master_formula} are bounded by the right-hand side in \eqref{eq:moments_I_z}.

Since $\eps \in (0,1/2)$ was arbitrary this concludes the proof of \eqref{eq:moments_I_z}. 
\end{proof}

\appendix
\section{Proof of Lemma~\ref{lem:existence_uniqueness_vf_equation}}

The existence and uniqueness of the solution of \eqref{eq:v} will be a consequence of the existence and uniqueness of the matrix Dyson equation 
\begin{equation} \label{eq:matrix_valued}
 - \Mf^{-1}(\eta) = \ii\eta \id -\Af+  \mathcal S[\Mf(\eta)]. 
\end{equation}
Note that $\Af\in\C^{2n\times 2n}$ and $\cal{S}\colon \C^{2n\times 2n} \to \C^{2n\times 2n}$ were defined in \eqref{MDE:Data}.

The matrix Dyson equation, \eqref{eq:matrix_valued}, has a unique solution under the constraint that the imaginary part
\[ \Im \Mf \defeq \frac{1}{2\ii}(\Mf - \Mf^*) \]
is positive definite. This was established in \cite{Helton01012007}.  In the context of random matrices, \eqref{eq:matrix_valued} was studied in \cite{AjankiCorrelated}.

In the following proof, for vectors $a, b, c, d\in \C^n$, we will denote the $2n\times 2n$ matrix having diagonal matrices with diagonals $a, b,c,d$ on its top-left, top-right, lower-left and lower-right $n\times n$ blocks, 
respectively, by \[ \begin{pmatrix} a & b \\ c & d \end{pmatrix}\defeq  \begin{pmatrix} \diag a & \diag b \\ \diag c & \diag d \end{pmatrix}\in \C^{2n\times 2n}. \]

\begin{proof}[Proof of Lemma~\ref{lem:existence_uniqueness_vf_equation}]
We show that there is a bijection between the solutions of \eqref{eq:matrix_valued} with positive definite imaginary part $\Im \Mf$
and the positive solutions of \eqref{eq:v_combined}. 

We remark that \eqref{eq:matrix_valued} implies that there are vector-valued functions $a,b,c,d \colon \R_+ \to \C^n$ such that 
for all $\eta >0$ we have
\begin{equation} \label{eq:structure_M}
 \Mf(\eta) = \begin{pmatrix} a(\eta) & b(\eta) \\ c(\eta) & d(\eta) \end{pmatrix} . 
\end{equation}

First, we show that $\Im \diag \Mf$ is a solution of \eqref{eq:v_combined} satisfying $\Im \diag \Mf >0$ if $\Mf$ satisfies \eqref{eq:matrix_valued} and $\Im \Mf$ is positive definite. 
Due to \eqref{eq:structure_M}, multiplying \eqref{eq:matrix_valued} by $\Mf$ yields that \eqref{eq:matrix_valued} is equivalent to 
\begin{equation}
\label{eq:matrix_valued_as_vector_equations}
-1  = \ii\eta a + a Sd + b \bar z 
,\quad 0  = \ii \eta b + z a + b S^t a 
,\quad 0  = \ii \eta c + \bar z d + c S d
,\quad -1  = \ii\eta d + d S^ta +  z c 
\end{equation}
Solving the second relation in \eqref{eq:matrix_valued_as_vector_equations} for $b$ and the third relation in \eqref{eq:matrix_valued_as_vector_equations} for $c$, we obtain 
\begin{equation} \label{eq:matrix_valued_aux1}
 b = -\frac{z a}{\ii\eta + S^t a}, \qquad c = -\frac{\bar z d}{\ii\eta + S d}. 
\end{equation}
Plugging the first relation in \eqref{eq:matrix_valued_aux1} into the first relation in \eqref{eq:matrix_valued_as_vector_equations} and the second relation in \eqref{eq:matrix_valued_aux1} into the fourth relation in 
\eqref{eq:matrix_valued_as_vector_equations} and dividing the results by $a$ and $d$, respectively, imply
\[ -\frac{1}{a} = \ii\eta + Sd - \frac{\abs{z}^2}{\ii \eta + S^t a} , \quad  -\frac{1}{d} = \ii\eta + S^ta - \frac{\abs{z}^2}{\ii \eta + S d}. \] 
Therefore, if $a$ and $d$ are purely imaginary then $(\Im a, \Im d) = -\ii(a,d)$ will fulfill \eqref{eq:v_combined}. 

In order to prove that $a$ and $d$ are purely imaginary, we define 
\[ \wt \Mf \defeq \begin{pmatrix}  \wt a(\eta) &  \wt b(\eta) \\  \wt c(\eta) &  \wt d(\eta) \end{pmatrix} \defeq \begin{pmatrix} - \bar a & \frac{z}{\bar z} \bar b \\ \frac{\bar z }{z} \bar c & - \bar d \end{pmatrix}.\]
The goal is to conclude $\Mf = \wt \Mf$, and hence $a = - \bar a$ and $d= - \bar d$, from the uniqueness of the solution of \eqref{eq:matrix_valued} with positive definite imaginary part. 
Since the relations \eqref{eq:matrix_valued_as_vector_equations} are fulfilled if $a$, $b$, $c$, $d$ are replaced by  $\wt a$, $\wt b$, $\wt c$, $\wt d$, respectively, $\wt \Mf$ satisfies \eqref{eq:matrix_valued}. 
For $j = 1, \ldots, n$, we define the $2\times 2$ matrices 
\[ M_j \defeq \begin{pmatrix} a_j &  b_j \\   c_j & d_j \end{pmatrix}, \quad 
 \wt M_j \defeq  \begin{pmatrix} \wt a_j & \wt b_j \\ \wt  c_j & \wt d_j \end{pmatrix}. \]
Note that $\Im \Mf$ is positive definite if and only if $\Im M_j$ is positive definite for all $j=1, \ldots, n$. 
Similarly, the positive definiteness of $\Im \wt \Mf$ is equivalent to the positive definiteness of $\Im \wt M_j$ for all $j=1, \ldots, n$. 
We have
 \[ \Im M_j = \begin{pmatrix}  \Im a_j &  \frac{1}{2\ii}( b_j - \bar c_j) \\ \frac{1}{2\ii}(c_j - \bar b_j) &  \Im d_j \end{pmatrix}, \quad  
\Im  \wt M_j \defeq  \begin{pmatrix} \Im a_j & \frac{z}{2\ii\bar z} ( \bar b_j - c_j) \\  \frac{\bar z}{2\ii z} ( \bar c_j - b_j) & \Im d_j \end{pmatrix}.   \]
As $\tr \Im \wt M_j = \tr \Im M_j$ and $\det \Im \wt M_j = \det \Im M_j$ for all $j =1, \ldots,n $ we get that $\wt \Mf$ is a solution of \eqref{eq:matrix_valued} with positive definite imaginary part $\Im \wt \Mf$. 
Thus, the uniqueness of the solution of \eqref{eq:matrix_valued} implies $\Mf = \wt \Mf$ as well as $a= - \bar a$ and $d = - \bar d$. 

Moreover, since 
\[ \Im \Mf = \begin{pmatrix}  \Im a &  ( b - \bar c)/(2\ii) \\ 
 (c - \bar b)/(2\ii) &  \Im d \end{pmatrix} \]
is positive definite we have that $\Im a >0$ and $\Im d >0$. Hence, $(\Im a, \Im d)$ is a positive solution of \eqref{eq:v_combined}. 

Conversely, let $\vf = (v_1, v_2) \in \C^{2n}$ be a solution of  \eqref{eq:v_combined} satisfying $\vf >0$ and $u$ be defined as in \eqref{eq:def_u}. 
Because of \eqref{eq:def_u}, we obtain that $\Mf= \Mf^z$, defined as in \eqref{eq:def_Mf}, 
is a solution of \eqref{eq:matrix_valued}. To conclude that $\Im \Mf$ is positive definite, it suffices to show that $\det \Im M_j>0$ for all $j = 1, \ldots,n$ with 
\[ M_j \defeq \begin{pmatrix}\ii (v_1)_j & -z u_j \\  - \bar z u_j & \ii(v_2)_j \end{pmatrix}  \]
as $\tr \Im M_j = (v_1)_j + (v_2)_j >0$. Since $z u_j - \overline{\bar z u_j} = 0$ for all $j=1, \ldots, n$ by \eqref{eq:def_u} we obtain 
\[ \det \Im M_j = (v_1)_j (v_2)_j - \frac{1}{4} \abs{z u_j - \overline{\bar z u_j}}^2 = (v_1)_j (v_2)_j >0. \]

Therefore, there is a bijection between the solutions of \eqref{eq:matrix_valued} with positive definite imaginary part and the positive solutions of \eqref{eq:v_combined}.
Appealing to the existence and uniqueness of \eqref{eq:matrix_valued} proved in \cite{Helton01012007} 
concludes the proof of Lemma~\ref{lem:existence_uniqueness_vf_equation}. 
\end{proof}

\section{Contraction-Inversion Lemma} \label{app:aux_results}

\begin{proof}[Proof of Lemma \ref{lem:rotation_inversion}] \label{proof:rotation_inversion}
 The bounds \eqref{TF} imply that $\id-\Af\Bf$ is invertible and 
\[ \normtwo{ (\id-\Af\Bf)^{-1}}\le \frac{1}{c_1\eta}. \]
The main
point of this lemma is to show  that $(\id-\Af\Bf)^{-1}\pf$ can be bounded independently of $\eta$ for 
$\pf$ satisfying \eqref{eq:assumption_fminus_d}. We introduce $\hf\defeq (\id-\Af\Bf)^{-1}\pf$. Thus, \eqref{eq:estimate_stability_lemma} is equivalent to $\normtwo{\hf} \leq C \normtwo{\pf}$ 
for some $C>0$ which depends only on $c_1, c_2, c_3$ and $\eps$.  Without loss of generality, we may assume that $\normtwo{\hf}=1$.
We decompose 
\begin{equation} \label{eq:stab_lemma_decomposition}
\hf = \alpha \bb_- + \beta \bb_+ + \gamma \xf, 
\end{equation}
where $\alpha = \scalar{\bb_-}{\hf}$, $\beta = \scalar{\bb_+}{\hf}$ and $ \xf \perp \bb_\pm$ satisfying $\normtwo{\xf} = 1$, 
thus $\abs{\alpha}^2 + \abs{\beta}^2 +\abs{\gamma}^2=1$. 
Since $\Bf=\Bf^*$, we have  $\bb_+\perp \bb_-$ and $\Bf\xf\perp \bb_\pm$. 
Hence, we obtain 
\[ \normtwo{\Af\Bf \hf}^2 \leq \normtwo{\Bf\hf}^2 \leq \abs{\alpha}^2 \normtwo{\Bf} + \abs{\beta}^2 \normtwo{\Bf} + \abs{\gamma}^2 \normtwo{\Bf\xf}^2 
\leq 1 - \eps + \eps(\abs{\alpha}^2 + \abs{\beta}^2 ),  \]
where we used $\normtwo{\Af} \leq 1$, $\normtwo{\Bf}\le 1$ and $\normtwo{\Bf\xf}\le 1-\eps$ in the last step.
Therefore, if $\abs{\alpha}^2 + \abs{\beta}^2 \leq 1 - \delta$ for some $\delta>0$ to be determined later, 
 then  $\normtwo{\Af\Bf\hf} \leq \sqrt{1 - \eps\delta}\normtwo{\hf} \leq (1 - \eps\delta/2) \normtwo{\hf}$ and 
thus \begin{equation}  \label{eq:stab_lemma_final_estimate_exact}
 1=\normtwo{\hf} \leq \frac{2}{\eps\delta} \normtwo{\pf}.  
\end{equation}

For the rest of the proof, we assume that $\abs{\alpha}^2 + \abs{\beta}^2 \geq 1- \delta$. 
In the regime, where $\abs{\alpha}$ is relatively large, we compute $\scalar{\bb_-}{ (\id-\Af\Bf)\hf}$,
capitalize on the positivity of  $\scalar{\bb_-}{ (\id-\Af\Bf)\bb_-}$ and treat all other terms as errors.
In the opposite regime, where $\abs{\beta}$ is relatively large, we use the positivity of $\scalar{\bb_+}{ (\id-\Af\Bf)\bb_+}$.

Using \eqref{eq:stab_lemma_decomposition}, we compute 
\[\scalar{\bb_-}{\pf} =  \scalar{\bb_-}{(\id-\Af\Bf)\hf} = \alpha (1 + \normtwo{\Bf}\scalar{\bb_-}{\Af\bb_-}) - \beta \normtwo{\Bf} \scalar{\bb_-}{\Af\bb_+} - \gamma \scalar{\bb_-}{\Af\Bf\xf}. \]
From $\normtwo{\Af} \leq 1$, the Hermiticity of $\Af$, $\scalar{\bb_-}{\Bf\xf} =0$, \eqref{eq:norm_Tf_plus_Tf_minus} and \eqref{eq:properties_F}, we deduce
\[
 \abs{\scalar{\bb_-}{\Af\bb_-}}  \leq 1, \quad 
\abs{\scalar{\bb_-}{\Af\bb_+}} = \abs{\scalar{\bb_- + \Af\bb_-}{\bb_+}} \leq c_2 \eta, \quad
\abs{\scalar{\bb_-}{\Af\Bf\xf}}  = \abs{\scalar{\bb_- + \Af\bb_-}{\Bf\xf}} \leq c_2 \eta(1-\eps). 
 \]
Employing these estimates, $\normtwo{\Bf} \leq 1 - c_1 \eta$ and \eqref{eq:assumption_fminus_d}, 
together with $\abs{\gamma}^2\le \delta$, 
 we obtain 
\begin{equation} \label{eq:aux_d_estimate2}
c_3 \normtwo{\pf} \geq \abs{\alpha} c_1 - \abs{\beta} c_2 - 
\sqrt{\delta} c_2 (1- \eps)
\end{equation}
after dividing through by $\eta>0$.
If $\abs{\alpha}c_1 \geq c_2 \abs{\beta} + \sqrt \delta c_2 (1-\eps) + \delta\eps c_3/2$ 
 then we obtain \eqref{eq:stab_lemma_final_estimate_exact}.

Therefore, it suffices to show \eqref{eq:stab_lemma_final_estimate_exact} in the regime
\begin{equation} \label{eq:stab_lemma_additional_assumptions}
\abs{\gamma}^2 \leq \delta, \qquad \abs{\alpha}c_1 \leq c_2 \abs{\beta} + \sqrt \delta  c_2 (1-\eps) + \delta\eps c_3/2.
\end{equation}
For this regime, we use \eqref{eq:stab_lemma_decomposition} and obtain
\begin{equation}
\scalar{\bb_+}{\pf}= 
\scalar{\bb_+}{(\id-\Af\Bf)\hf} = \beta (1- \normtwo{\Bf}\scalar{\bb_+}{\Af\bb_+}) - \alpha\normtwo{\Bf}\scalar{\bb_+}{\Af\bb_-} - \gamma \scalar{\bb_+}{\Af\Bf\xf}\label{eq:aux_scalar1}.
\end{equation}
We employ \eqref{eq:properties_F}, \eqref{eq:norm_Tf_plus_Tf_minus}, the Hermiticity of $\Af$ and $\scalar{\bb_-}{\bb_+} = 0$ to obtain
\begin{equation} \label{eq:equations_estimates_stab_lemma}
\scalar{\bb_+}{\Af\bb_+}  \leq 1- \eps, \quad 
\abs{\scalar{\bb_+}{\Af\bb_-}} = \abs{\scalar{\bb_+}{\bb_- + \Af\bb_-}} \leq c_2 \eta, \quad 
\abs{\scalar{\bb_+}{\Af\Bf\xf}} \leq 1 - \eps. \\
\end{equation}
Applying \eqref{eq:equations_estimates_stab_lemma} to \eqref{eq:aux_scalar1}, yields
\begin{equation} \label{eq:aux_d_estimate1}
 \normtwo{\pf} \geq \abs{\scalar{\bb_+}{\pf}} \geq \abs{\beta}\eps - \abs{\alpha} c_2 \eta - \abs{\gamma}
 (1-\eps) \geq \abs{\beta}\eps - \abs{\alpha} \frac{\eps c_1}{2c_2} - \sqrt{\delta} (1-\eps),
 \end{equation}
where we used the assumption $\eta\le \eps c_1/2c_2^2$.  Since  $\abs{\alpha} c_1/c_2 \le  \abs{\beta} + O(\sqrt{\delta})$
from \eqref{eq:stab_lemma_additional_assumptions},  we obtain that  $\normtwo{\pf} \ge \abs{\beta}\eps/3$ 
for any  $\delta\le \delta_0(c_1, c_2, c_3, \eps)$  sufficiently small.  Furthermore, $\abs{\alpha}^2 + \abs{\beta}^2 \ge 1-\delta$
and the fact that $\abs{\beta}$ is large compared with $\abs{\alpha}$ in the sense \eqref{eq:stab_lemma_additional_assumptions}
guarantee that $\abs{\beta}^2\ge \frac{1}{3}[1+ (c_2/c_1)^2 ]^{-1}$, if $\delta$ is sufficiently small.
In particular, $\normtwo{\pf} \ge \eps\delta/2$ can be achieved with a small $\delta$, 
i.e., 
\eqref{eq:stab_lemma_final_estimate_exact} holds true in the regime \eqref{eq:stab_lemma_additional_assumptions} as well. 
This concludes the proof of Lemma \ref{lem:rotation_inversion}. 
\end{proof}

\begin{lem} \label{lem:TwoNorms_to_InfNorms}
\begin{enumerate}[(i)]
\item Uniformly for $z\in \Dsma \cup \Dbig$ and $\eta >0$, we have 
\begin{equation} \label{eq:twoinfnorm_Tf_Ff}
 \normtwoinf{\Ff} \lesssim 1, \quad \normtwoinf{\Tf\Ff} \lesssim 1, \quad \normtwoinf{\Ff\Tf} \lesssim 1. 
\end{equation}
\item If $w \notin \spec (\Tf\Ff) \cup \{0\}$ and $\normtwo{(w \id- \Tf\Ff)^{-1} \yf} \lesssim \normtwo{\yf}$ for some $\yf \in \C^{2n}$ then 
\begin{equation} \label{eq:control_infnorm_with_twonorm_for_TF}
\norminf{(w\id- \Tf\Ff)^{-1} \yf } \lesssim \frac{1}{\abs{w}}\norminf{\yf}. 
\end{equation}
A similar statement holds true for $(\bar w \id - \Ff\Tf)^{-1} = \left[ ( w \id - \Tf\Ff)^{-1} \right]^*$.
\item 
For every $\eta_*>0$, depending only on $\zsq_*$ and the model parameters, such that 
\begin{equation} \label{eq:inverse_id_minus_TF_Q_conditions}
 \normtwo{(\id - \Tf\Ff)^{-1} \Qf} \lesssim 1, \quad 1- \normtwo{\Ff} \gtrsim \eta, \quad \normtwo{\ffm+\Tf\ffm} \lesssim \eta, \quad \norminf{\ffm}\lesssim 1
\end{equation}
uniformly for all $\eta \in (0,\eta_*]$ and $z \in \Dsma$, we have 
\begin{equation}
\norm{\left((\id-\Tf\Ff)^{-1}\Qf\right)^*}_\infty \lesssim 1
\label{eq:estimate_stability_inf_projection} 
\end{equation}
uniformly for $\eta \in (0,\eta_*]$ and $z \in \Dsma$.
Here, $\Qf$ denotes the orthogonal projection onto the subspace $\ffm^\perp$, i.e., $\Qf \yf \defeq \yf - \scalar{\ffm}{\yf}\ffm$ for every $\yf \in \C^{2n}$. 
\end{enumerate}
\end{lem}

The estimate \eqref{eq:control_infnorm_with_twonorm_for_TF} is proved similarly as (5.28) in \cite{AjankiQVE}.

\begin{proof}
As $\normtwoinf{\Sf_o} \lesssim 1$ by \eqref{eq:assumption_A}, we obtain from Proposition \ref{pro:estimates_v_small_z}, and \eqref{eq:estimate_uf}
\[ \normtwoinf{\Ff} \leq \norminf{\Vf^{-1}} \normtwoinf{\Sf_o}\normtwo{\Vf^{-1}} = \norminfa{\frac{\uf\vf}{\wt\vf}} \normtwoinf{\Sf_o}\lesssim 1\]
uniformly for all $\eta >0$ and $z\in \Dsma \cup \Dbig$. This proves the first estimate in \eqref{eq:twoinfnorm_Tf_Ff}.
From Lemma \ref{lem:pro_Tf} (i), we conclude the second and the third estimate in \eqref{eq:twoinfnorm_Tf_Ff}.

We set $\xf \defeq (w \id-\Tf\Ff)^{-1}\yf$. By assumption there is $C \sim 1$ such that 
\[ \normtwo{\xf} \leq C \normtwo{\yf} \leq C \norminf{\yf} . \] 
Moreover, since $ w\xf = \Tf\Ff \xf + \yf$ we obtain from the previous estimate  
\[ \abs{w} \norminf{\xf} \leq \norminf{\Tf\Ff\xf} + \norminf{\yf} \leq \left( \normtwoinf{\Tf\Ff} C 
+ 1\right) \norminf{\yf}. \]
Using the second estimate in \eqref{eq:twoinfnorm_Tf_Ff}, this concludes the proof of \eqref{eq:control_infnorm_with_twonorm_for_TF}.
The statement about $(\bar w\id - \Ff\Tf)^{-1}$ follows in the same way using the third estimate in \eqref{eq:twoinfnorm_Tf_Ff} instead of the second.

For the proof of \eqref{eq:estimate_stability_inf_projection}, we remark that the first condition in \eqref{eq:inverse_id_minus_TF_Q_conditions} implies that 
\begin{equation} \label{eq:estimate_adjoint} 
 \normtwoa{\left((\id-\Tf\Ff)^{-1} \Qf\right)^* } = \normtwoa{(\id-\Tf\Ff)^{-1} \Qf } \lesssim 1. 
\end{equation}
The second assumption in \eqref{eq:inverse_id_minus_TF_Q_conditions} yields 
\begin{equation} \label{eq:trivial_bound_app}
\normtwoa{(\id - \Tf\Ff)^{-1}} \lesssim \eta^{-1}.
\end{equation}

Take $\yf \in \C^{2n}$ arbitrary. We get $[\Tf, \Qf]\yf = \scalar{\Tf\ffm+\ffm}{\yf}\ffm - \scalar{\ffm}{\yf}(\Tf\ffm+\ffm)$, where 
$[\Tf,\Qf] = \Tf\Qf - \Qf\Tf$ denotes the commutator of $\Tf$ and $\Qf$. Therefore, 
\begin{equation} \label{eq:estimate_comm_TQ}
 \normtwo{[\Tf,\Qf]} \leq 2 \normtwo{\ffm + \Tf\ffm} \lesssim \eta 
\end{equation}
by the third condition in \eqref{eq:inverse_id_minus_TF_Q_conditions}.
We set $\xf \defeq \Qf(\id-\Ff\Tf)^{-1} \yf = \left((\id-\Tf\Ff)^{-1}\Qf\right)^*\yf$ and compute  
\[ \xf = \Ff\Tf\xf + \Qf\yf - \Ff[\Tf,\Qf](\id-\Ff\Tf)^{-1}\yf, \]
where we commuted $\id-\Ff\Tf$ and $\Qf$ and 
used that $\Ff$ and $\Qf$ commute. Hence, using $\normtwo{\xf} \lesssim\normtwo{\yf} \lesssim\norminf{\yf}$ by 
\eqref{eq:estimate_adjoint} , $\norminf{\Qf} \leq 1 + \norminf{\ffm}$,   \eqref{eq:estimate_comm_TQ} 
and  \eqref{eq:trivial_bound_app}, we obtain 
\[ \norminf{\xf} \lesssim \left( \normtwoinf{\Ff\Tf}  + 1 + \norminf{\ffm} + \normtwoinf{\Ff} \right) \norminf{\yf}\lesssim \norminf{\yf}.\]
Here, we used the fourth assumption in \eqref{eq:inverse_id_minus_TF_Q_conditions} and \eqref{eq:twoinfnorm_Tf_Ff}. 
Notice that the $\eta^{-1}$ factor from the trivial estimate \eqref{eq:trivial_bound_app} was 
compensated by the smallness
of the commutator $[\Tf, \Qf]$ which was a consequence of the third assumption  in  \eqref{eq:inverse_id_minus_TF_Q_conditions}. 
This concludes the proof of \eqref{eq:estimate_stability_inf_projection}.
\end{proof}

\begin{proof}[Proof of Lemma \ref{lem:approximating_eigenvector}] \label{proof:approximating_eigenvalue}
We first prove that 
\begin{equation}
\normtwo{\ffm - \af} = O(\eta) \label{eq:estimate_ff_-}.
\end{equation}
uniformly for $\eta \leq 1$ and $\zsq \in \Isma$.
To that end, we introduce the auxiliary operator  
\[ 
\Af \defeq \normtwo{\Ff}\id + \Ff. 
\]
Therefore, we obtain from $\Ff\ffm = - \normtwo{\Ff} \ffm$ and \eqref{eq:af_approximating_eigenvector} 
\[ \Af \ffm = 0, \quad \Af \af = O(\eta).\]
Let $\Qf$ be the orthogonal projection onto the subspace $\ffm^\perp$ orthogonal to $\ffm$, i.e., $\Qf \yf \defeq \yf - \scalar{\ffm}{\yf} \ffm $
 for $\yf \in \C^{2n}$. We then obtain $\Af\Qf \af = O(\eta)$ which implies $\Qf \af = O(\eta)$ as 
$\Af$ is invertible on $\ffm^\perp$ and $\normtwo{(\Af|_{\ffm^\perp})^{-1}} \sim 1$  by \eqref{eq:gap_F_order_1}. 
We infer \eqref{eq:estimate_ff_-}. 

For the proof of \eqref{eq:estimate_ff_-_inf}, we follow the proof of 
\eqref{eq:estimate_stability_inf_projection}, replace $\Tf$ by $-\id$ and use Lemma \ref{lem:prop_Ff} (i)
instead of the second and fourth condition in \eqref{eq:inverse_id_minus_TF_Q_conditions}.
\end{proof}

\bibliographystyle{amsplain}
\bibliography{literature}
\end{document}